\documentclass[12pt]{article}
\usepackage{amsmath,amssymb,amsfonts,amsthm}
\usepackage[all]{xy}

\newcounter{item}[section]
\newcounter{kirshr}
\newcounter{kirsha}
\newcounter{kirshb}
\newenvironment{enumroman}{\setcounter{kirshr}{1}
\begin{list}{(\roman{kirshr})}{\usecounter{kirshr}} }{\end{list}}
\newenvironment{enumarab}{\setcounter{kirshb}{1}
\begin{list}{(\arabic{kirshb})}{\usecounter{kirshb}} }{\end{list}}
\newtheorem{theorem}{Theorem}[section]

\newtheorem{lemma}[theorem]{Lemma}
\newtheorem{corollary}[theorem]{Corollary}
\theoremstyle{definition}
\newtheorem{remark}[theorem]{Remark}

\newtheorem{example}[theorem]{Example}
\newtheorem{definition}[theorem]{Definition}

\def\C{{\mathfrak{C}}}
\def\Fm{{\mathfrak{Fm}}}

\def\Nr{{\mathfrak{Nr}}}
\def\Fr{{\mathfrak{Fr}}}
\def\Sg{{\mathfrak{Sg}}}
\def\Fm{{\mathfrak{Fm}}}
\def\Rd{{\mathfrak{Rd}}}

\def\CA{{\bf CA}}
\def\RCA{{\bf RCA}}
\def\K{{\bf K}}
\def\L{{\bf L}}

\def\QA{{\bf QA}}
\def\QEA{{\bf QEA}}

\def\(R)RA{{\bf (R)RA}}

\def\R{\mathbb{R}}

\def\QA{{\bf QA}}
\def\QEA{{\bf QEA}}

\def\R{\mathbb{R}}
\def\N{\mathbb{N}}

\def\R{\mathbb{R}}

\def\Sc{{\bf Sc}}

\def\c #1{{\cal #1}}
 \def\CA{{\sf CA}}

\def\M{{\mathfrak{M}}}

\def\At{{\mathfrak{At}}}

\def\G{{\mathfrak{G}}}

\def\K{{\bf K}}
\def\QA{{\bf QA}}
\def\QEA{{\bf QEA}}

\def\sub#1#2{{\sf s}^{#1}_{#2}}
\def\cyl#1{{\sf c}_{#1}}
\def\diag#1#2{{\sf d}_{#1#2}}

\def\c #1{{\cal #1}}
\def\s{{\sf s}}

\def\pa{$\forall$}
\def\pe{$\exists$}

\def\ef{Ehren\-feucht--Fra\"\i ss\'e}

\def\nodes{{\sf nodes}}

\def\restr #1{{\restriction_{#1}}}

\def\A{{\mathfrak{A}}}
\def\B{{\mathfrak{B}}}
\def\C{{\mathfrak{C}}}
\def\D{{\mathfrak{D}}}
\def\P{{\mathfrak{P}}}
\def\Fm{{\mathfrak{Fm}}}
\def\Ra{{\mathfrak{Ra}}}
\def\Nr{{\mathfrak{Nr}}}
\def\F{{\mathfrak{F}}}
\def\CA{{\bf CA}}
\def\RCA{{\bf RCA}}
\def\c#1{{\mathcal #1}}

\def\set#1{ \{#1\}}

\def\Ca{{\mathfrak Ca}}
\def\b#1{{\bar{ #1}}}
\def\pe{$\exists$}
\def\pa{$\forall$}
\def\Cm{{\mathfrak Cm}}
\def\Sg{{\mathfrak Sg}}

\def\At{{\sf At}}
\def\Id{{\sf Id}}

\def\rng{{\sf rng}}
\def\dom{{\sf dom}}

\def\w{{\sf w}}
\def\g{{\sf g}}
\def\y{{\sf y}}
\def\r{{\sf r}}

\def\cyl#1{{\sf c}_{#1}}
\def\sub#1#2{{\sf s}^{#1}_{#2}}
\def\diag#1#2{{\sf d}_{#1#2}}

\def\de{Dedekind-MacNeille}

\def\ws{winning strategy}
\def\ef{Ehren\-feucht--Fra\"\i ss\'e}

\def\Rl{\mathfrak{Rl}}

\def\y{{\sf y}}
\def\g{{\sf g}}
\def\r{{\sf r}}
\def\w{{\sf w}}

\def\Sc{{\sf Sc}}

\def\TCA{{\sf TCA}}
\def\CA{{\sf CA}}

\def\R{{\sf R}}

\def\L{{\mathfrak{L}}}

\def\Z{{\mathbb{Z}}}
\def\K{{\sf K}}

\def\RCA{{\sf RCA}}
\def\PEA{{\sf PEA}}

\def\Df{{\sf Df}}
\def\Uf{{\sf Uf}}
\def\Rd{{\mathfrak{Rd}}}

\def\la#1{\langle#1\rangle}

\def\nodes{{\sf nodes}}

\def\c{{\sf c}}
\def\d{{\sf d}}
\def\TeCA{{\sf TeCA}}
\def\RTeCA{{\sf RTeCA}}

\def\d{Dedekind-MacNeille}
\def\d{{\sf d}}

\def\QEA{{\sf QEA}}
\def\nodes{{\sf nodes}}

\def\c{{\sf c}}
\def\d{{\sf d}}
\def\TeCA{{\sf TeCA}}
\def\RTeCA{{\sf RTeCA}}

\def\d{Dedekind-MacNeille}
\def\d{{\sf d}}

\def\d{Dedekind-MacNeille}
\def\d{{\sf d}}
\def\s{{\sf s}}
\def\c{{\sf c}}

\def\TL{{\mathfrak{TL}}}
\def\Tm{{\mathfrak{Tm}}}
\def\QA{{\sf QA}}
\def\M{\mathfrak{M}}
\def\RQEA{{\sf RQEA}}
\title{Atom canonicity, and omitting types in temporal 
and topological cylindric algebras}

\author{Tarek Sayed Ahmed}

\begin{document}
\maketitle

\begin{abstract}

We study 
 what we call topological cylindric algebras and tense cylindric algebras defined for every ordinal $\alpha$.
The former are cylindric algebras of dimension $\alpha$
expanded with $\sf S4$ modalities indexed by $\alpha$. The semantics
of representable topological algebras is induced by the interior operation relative to a topology
defined on their bases.
Tense cylindric algebras are cylindric algebras expanded by the modalities
$F$(future)  and $P$ (past)
algebraising predicate temporal logic. 

We show for both tense  
and topological cylindric algebras of finite dimension  $n>2$  that
infinitely many varieties containing and including the variety of representable algebras of dimension $n$ 
are not atom canonical. We show that any class containing the class of completely representable algebras 
having a weak neat embedding property is not elementary. 
From these two results we draw  the same conclusion on omitting types for finite variable
fragments of predicate topologic and temporal logic. 
We  show that the usual version of the omitting types theorem restricted to such fragments when the number of variables is $>2$
fails dramatically  even if we considerably broaden the class of models permitted to omit a single non principal type 
in countable atomic theories, namely, the non-principal type consting of co atoms.
\end{abstract}
\section{Introduction}
\subsection{General}

We alert the reader to the fact that the introduction is quite 
long since we discuss extensively many concepts related to the subject matter of the paper, in the process 
emphasizing similarities and  highlighting differences. 
The topic is rich, the paper is interdisciplinary between temporal logic, topologcal logic and algebraic logic.
Besides some history is surveyed, dating back 
to the seminal work of Tarski and McKinsey up to the present time. 

Universal logic addresses different logical systems simultaneously in essentially three ways. 
Either abstracting common features, or building new bridges between them, or 
constructing new logics.

We do the first two in what follows, though the third can also be implemented 
by combining temporal and topological logic getting a natural multidimensional modal logic 
 of space and time. Such logics do exist in the literature; an example 
is dynamic topological logic that studies the modal 
aspects of dynamical systems.
 
However, we defer this unification to another paper. 
We apply cylindric algebra theory to topological and minimal temporal 
pedicate logic, also known as tense predicate logic,  with finitely many variables, 
so that the bridges are built via algebraic logic, and their common abstraction is 
cylindric algebra theory when we expand cylindric algebras by modalities like the interior box  operator 
stimulated by a topology on the base
of a cylindric algebra that happens to be representable,
or the temporal  $F$ (future) and $P$ (past).

We will also have occasion to make 
detours into the functionally maximal temporal logic with Still and Until.

Another unifying framework here is that we repeat an extremely rewarding 
deed 
of  Henkin's performed  ages ago (in the sixties of the $20$th century), when he decided to  study cylindric algebras of 
finite dimension that can be seen in retrospect as the modal algebras of finite variable fragments of first order logic. 
Here the dimension is the same as the number of 
available variables.

The repercusions of such a view turned out absolutely astounding  
and it has initiated a lot of quite deep research till this day, 
involving algebraic and modal logicains from Andre\'ka 
to Venema;  this paper included.

We immitate Henkin, but we add a topological and temporal dimension. 
We study finite variable fragments of topological, tense and temporal predicate logic using the well developped 
machinery of algebraic logic.  We reep the harvest of our algebraic results by  deducing the metalogic one of omitting types.

We address the algebraic notions of atom-canonicity an important persistense property in modal logic,
and complete representations an important notion for cylindric-like algebras, both notions 
proved closely related to  the algebraic modal property of atom-canonicity and the meta-logical notion of omitting types \cite{Sayed}.

We do such tasks 
for
the modal algebras of both temporal and topological logics of finite dimension $>2$, culminating in a very strong 
negative result on omitting types when semantics are restricted to clique guarded semantics, 
proving the result in the abstract.

The interconnection between the three notions 
mentioned above manifests itself blatantly in this paper, 
where we draw a deep conclusion concerning  the failure of the Orey-Henkin 
omitting types. This is proved using quite sophisticated 
machinery from algebraic logic, more specifically from cylindric algebra theory, namely, rainbow constructions, which were first used 
in the context of relaton algebras. 

Using this powerful technique, we construct certain representable algebras, having
countably many atoms of finite dimension $n>2$. In the first case, we deal with  an algebra with 
no complete representations. In the second case we deal with an algebra whose \de\ 
completion is outside infinitely many varieties containing properly 
the variety of representable algebras.
These two constructions  are worked out for both temporal and topological cylindric algebras
of dimension $n$.

Topological logic with $n$ many variables can be viewed as a $2n$ dimensional 
propositional multimodal logic
or a Kripke complete logic that is the $n$ product of decidable bi-modal logics whose frames are of the form 
$(U, U\times U, R)$; with $R$ is the accessiblity relation of an $\sf S4$ modality, in short 
$\sf S5\times \sf S4$. 

On the other hand, modal algebras of temporal logic of dimension 
$m$, that are representable in a concrete sense
to be specified below, can be seen as product of
cylindric set algebras of the same dimension $n$ 
indexed by a time flow $\bold T=(T, <)$, with navigating 
functions between the different worlds that are cartesian squares. 

So, metaphorically   cylindric set agebras of dimension $n<\omega$ 
are snap shots of  temporal semantics 
at one moment of time in the time flow.

\subsection{Universal logic}

Universal logic is
the field of logic that is concerned with giving an account of what features are common to all logical structures.
If slogans are to be taken seriously, then universal logic is to logic what universal algebra is to algebra.
The term ``universal logic' was introduced in the 1990s
by Swiss Logician Jean Yves Beziau but the field has arguably existed for many decades.
Some of the works of Alfred Tarski in the early twentieth century, on metamathematics and in algebraic logic,
for example, can be regarded undoubtedly, in retrospect, as fundamental contributions to universal logic.

Indeed, there is a  whole well established
branch of algebraic logic,
that attempts to deal with the universal notion of a logic. Pioneers in this branch include Andr\'eka
and    N\'emeti \cite{AN75} and Blok and Pigozzi \cite{BP}.
The approach of Andr\'eka and N\'emeti though is more general, since, unlike the approach in \cite{BP} which
is purely syntactical,  it allows semantical notions stimulated via
so-called `meaning functions' \cite{ANS}, to be defined below. Other universal approaches to many cylindric-like algebras were implemented in
\cite{univl} in the context of the very general notion of what is known
in the literature as systems of varieties definable by a Monk's schema \cite{ST, HMT2}, and in the context of category theory
in \cite{conference}.

One aim of universal logic is to determine the domain of validity of such and such metatheorem
(e.g. the completeness theorem, the Craig interpolation
theorem, or the Orey-Henkin omitting types theorem of first order logic) and to
give general formulations of metatheorems in broader, or even entirely other contexts.
This is also done in algebraic logic, by dealing with modifications and variants of first order logic resulting in a natural way
during the process of {\it algebraisation}, witness for example the omitting types theorem proved in  \cite{Sayed}.

This kind of investigation  is extremely potent  for applications and helps to make the distinction
between what is really essential to a particular logic and what is not.

During the 20th
century, numerous logics have been created, to mention only a few: intuitionistic logic, modal logic, topological logic, topological dynamic logic,
spatial logic,  dynamic logic,  tense logic,
temporal logic, many-valued logic, fuzzy logic, relevant logic,
para-consistent logic, non monotonic logic,
etc. 

The rapid development of computer science, since the fifties of the 20th century initiated by work of giants like Godel, Church and Turing,
ultimately brought to
the front scene other logics as well, like  logics of programs 
and lambda calculas (the last can be traced back to the work of Church). 

After a while  it became noticable that certain patterns of 
concepts kept being repeated albeit in different
logics. But then the time  was ripe to make in retrospect an inevitable abstraction  
like as the case with the field of abstract model theory (Lindstrom's theorem is an example here).

\subsection{Abstract algebraic logic and institutions}

Other abstract approaches to logic, besides categorial logic, are topoi theory
and abstract algebraic  logic which focuses on the algebraic essence of numeras logics, like first order logic, modal logics, 
logics with infinitary predicates, a new addition introduced here, namely,  finite variable
fargments of tense, temporal and topological predicate logics

Abstract algebraic logic provides 
a unifying algebraic framework viz the
process of algebraisation, 
categorial logic via institutions, and topoi theory, and last but not least universal logic.

Institution theory is a major model theoretic trend of universal logic that formalizes within category theory 
the intuitive notion of a logical system, including syntax, semantics, and the satisfaction relation between them.
It owes its birth to  so-called {\it computing science} as a response to the 
explosion of new logics 
and it developed to  an  important foundational theory. 
Though computing science can be blamed for its poor intellectual value, 
but institutions is a very significant exception. In fact 
what is now labelled as computing science is hardly a science in the way mathematics 
or theoretical physics are. 
Computing science can be viewed as a playground
where several actors, most notably mathematics, logic, engineering and philosophy,  play.

Sometimes an interesting play which
has not only brought significant changes and development to the actors but also revolutionalized 
our way of  thinking. 

Institution theory, universal logic an and abstract algebraic logic are no exceptions. 
All three branches have travelled well in model theory and logic. 
The way we think of logic and model theory
will never be the same as before. 
The development of model and proof theory in the very abstract setting of arbitrary 
institutions, is free of commitment to a particular logical system.
When we live without concrete models, sentences, satisfaction, and so on, we reep
the harvest of  another level of abstraction and generality. 

A  deeper understanding of model theoretic phenomena not hindered by the largely irrelevant details of a particular logical system, 
but instead  guided by structurally by clear cut causality, necessarily follows.

The latter aspect is based upon the fact that concepts come naturally as presumed features that a ``logic'' might exhibit 
or not and are defined at the most appropriate level of abstraction; hypotheses are kept as general as possible and introduced on a by-need basis, 
and thus results and proofs are modular and easy to track down regardless of their depth.

The continuous interplay between the specific and the general in institution theory brings a 
large array of new results for particular non-conventional approaches, unifies several known results, 
produces new results in well-studied conventional areas, 
reveals previously unknown causality relations, and dismantles some which 
are usually assumed as natural. Access to highly non-trivial results is also considerably facilitated.

The similarity of instititions to the notion of {\it full fledged algebriazable logics} in the sense of Andr\'eka and 
N\'emeti \cite{ANS} is 
striking.

They both assume nothing about the specific nature of a logical system,
models and sentences are arbitary objects; the only assumption being
that there is a satisfiability relation between models and texts or sentences, 
telling whether a sentence hold in a given model; 
this is implemented in the Andr\'eka and N\'emeti approach using {\it meaning functions} to be dealt with 
below. In institutions such connections are formulated in the extremely abstract level of 
category theory. 

A crucial unifying  feature is that notions of 
models, sentences and the satisfiabitlity relation
are defined abstractly; the definition is broad, in both cases, but narrow enough to obtain significant results 
at such a level of generality. Such thereby obtained 'meta-theorems' can even illuminate
various aspects of their concrete instances.

A vital difference though is in the algebraic approach the signature, determining the logical connectives
is {\it fixed} in advance.

In institutions signatures vary, and they are 
synchronized by two functors;  a syntactical functor navigating between 
sentences in different 
signatures, and another, a semantical one, taking models in a given  signature,
to models in a possibly different signatare
in such a way that semantics, expresssed via  
a satisfiability relation defined 
between sentences and models, is preserved. 

\subsection{Topological and minimal Temporal logic, tense logic}

Motivated by questions like:  which spatial structures may be characterized by means of modal logic,
what is the logic of space, how to encode
in modal logic different geometric relations, topological logic provides a framework for studying the confluence of the topological semantics for
$\sf S4$ modalities, based on topological spaces
rather than Kripke frames, with the $\sf S4$
modality induced by the interior operator.

Topological logic was introduced by Makowsky,  Ziegler and Sgro \cite{s}.
Such logics have a classical semantics with a topological flavour, addressing spatial logics
and their study was approached  using algebraic logic by Georgescu \cite{g}, the task that we further pursue in this
paper. Topological logics are apt for dealing with {\it logic} and {\it space};
the overall point is to take a common mathematical model of space (like a topological space)
and then to fashion logical tools to work with it.

One of the things which blatantly strikes one when studying elementary topology is that notions like open, closed, dense
are intuitively very transparent, and their
basic properties are absolutely straightforward to prove. However, topology
uses second order notions as it reasons with sets and subsets of `points'. This might suggest that like
second order logic, topology ought to be computationally very complex.
This apparent dichotomy between the two paradigms
vanishes when one realizes that a large portion of
topology can be formulated as a  simple modal logic,
namely, $\sf S4$!
This is for sure an asset
for modal logics tend to be much easier to handle than first order logic
let alone second order.

The project of relating topology to modal logic begins
with work of Alfred Tarski and J.C.C McKinsey \cite{tm}. Strictly speaking Tarski and McKinsey did not work with
modal logic, but rather with its algebraic counterpart, namely, Boolean algebras
with operators which is the approach we adopt here; the operators they studied where the closure operator induced on what they called
{\it the algebra of topology}, certainly a very ambitious title, giving the impression
that the paper aspired to completely {\it algebraise topology}. This paper was published around the time
when there was an interest in algebraising metamathematics as well,
culminating in defining cylindric and polyadic algebras
obtained independently by Tarski in Berkely and Halmos in Stanford,
respectively.

Ever since there has been extensive research to
highlight the similarities and differences between such algebraisations of predicate
logic  and it is commonly accepted now that they belong to different paradigms \cite{conference},
though they are essentially the same when they are restricted algebraise first order logic with
equality, giving the so-called infinite dimensional
locally finite algebras, to be further elaborated upon in a while.

In retrospect McKinsey and Tarski showed, that the Stone representation theorem for Boolean algebras extend to algebras
with operators to give topological semantics for classical propositional modal logic,
in which the `necessity' operation is modeled by taking the interior (dual operation)
of an arbitrary subset of
topological space. Although the topological completeness of $\sf S4$ has been well known for quite
a long time, it was until recently considered as some
exotic curiosity, but certainly having mathematical value.
It was in the 1990-ies that the work of McKinsey and Tarski, came to the front scene of modal logic (particularly spatial modal logic),
drawing serious attention of many researchers and inspiring
a lot of work stimulated basically by questions concerning the `modal logic of space';
how to encode in modal logic different geometric relations?
A point of contact here between topological spaces, geometry, and cylindric algebra theory
is the notion of dimension.

From the modern point of view one introduces a basic {\it modal language} with a set $\At$ of atomic propositions,
the logical Boolean connectives $\land$, $\neg$ and a modality $I$
to be interpreted as the
{\it interior operation.}
Let $X$ be a topological space. The modal language $\L_0$ is interpreted on such a space $X$ together with an {\it interpretation map}
$i:\At \to \wp(X)$. For atomic $p\in \At$, $i(p)$ says which points satisfy $p$. We do not require that $i(p)$ is open.
$(X, i)$ is said to be a {\it topological model}. Then $i$ extends to all $\L_0$ formulas by interpreting negation as
complement relative
to $X$, conjunction as intersection and $I$ as the interior operator.
In symbols we have:
\begin{align*}
i(\neg \phi)&=X\sim i(\phi),\\
i(\phi\land \psi)&=i(\phi)\cap i(\psi),\\
i(I(\phi))&={\sf int}i(\phi).
\end{align*}
The main idea here is that the basic properties of the
Boolean operations
on sets as well as the salient
topological operations like interior and its dual the closure, correspond to to schemes of sentences.
For example, the fact that the interior operator is idempotent
is expressed by $$i((II\phi)\leftrightarrow (I\phi))=X.$$
The natural question to ask about this language and its semantics is:
Can we characterize in an enlightening way the sentences $\phi$ with the property that for
all topological models $(X, i)$, $i(\phi)=X$; these are
the topologically valid sentences.
They are true at all points in all spaces under whatever  interpretation.
More succinctly, do we have a nice {\it completeness} theorem?

Tarski and McKinsey proved that the topologically valid sentences are exactly those provable in the modal logic
$\sf S4$.
$\sf S4$ has a seemingly different semantics using standard {\it Kripke frames.}
Now $X$ is viewed as the set of {\it possible worlds}. In $\sf S4$, $I$ is read as {\it all points which the current point relates to}.
To get a sound interpretation
of $\sf S4$ we should require that the current
point is {\it related to itself}. Without this requirement we get only a transitive frame; if the order is further
a linear order, then this represents  what is known
in temporal logic as a 'flow of time' getting a transitive model. (Strictly speaking, in a linear flow of time
we have two transitive linear relations $<$  and $>$ that are converses,
reflecting past and future; it is a bi- modal logic).

The reflexivity condition, together with transitivity leads naturally to pre-ordered model.

A {\it pre-ordered model}  is defined to be a triple $(X, \leq, i)$ where $(X, \leq)$ is a pre-order and $i: \At\to \wp(X)$
where
$$i(I(\phi))=\{x: \{y:x\leq y\}\subseteq i(\phi)\}.$$
Temporally world $x'\in X$ is a {\it successor} of world $x\in X$ if $x\leq x'$,  $x$ and $x'$ are equivalent worlds if further
$x\leq x'$ and $x'\leq x$.
We have a completely analogous result here; $\phi$ is valid in pre-ordered models if
$\phi$ is provable in $\sf S4$.

One can prove the equivalence of the two systems using only topologies on finite sets.
Let $(X, \leq)$ be a pre-order. Consider the {\it Alexandrov topology} on $X$, the open sets are the sets closed upwards in the order.
This gives a topology, call it $O_{\leq}.$ A correspondence
between topological models and pre-ordered models can thereby be obtained,
and as it happens we have
for any pre-ordered model $(X, \leq, i)$, all $x\in X$, and all $\phi\in \L_0$
$$x\models \phi \text { in } (X,\leq, i) \Longleftrightarrow x\models \phi \text { in } (X, O_{\leq}, i).$$

Using this result together with the fact that sentences satisfiable in $\sf S4$ have finite topological models, thus
they are Alexandrov topologies, one can show that the semantics
of both systems each is interpretable in the other; they are equivalent.
We can summarize the above discussion in the following neat theorem, that we can and will
attribute to McKinsey, Tarski
and  Kripke; this historically is not very accurate. For a topological space $X$ and $\phi$ an $\sf S4$
formula we write $X\models \phi$, if $\phi$ is valid topologically in
$X$ (in either of the senses above). For example, $w\models \Box \phi$ iff
for all  $w'$ if $w\leq w'$, then $w'\models \phi$, where $\leq$ is the relation $x\leq y$
iff $y\in {\sf cl}\{x\}$.

\begin{theorem} (McKinsey-Tarski-Kripke)
Suppose that $X$ is a dense in itself metric space (every point is a limit point)
and $\phi$ is a modal $\sf S4$ formula. Then the following are equivalent
\begin{enumarab}
\item $\phi\in \sf S4.$
\item$\models \phi.$
\item $X\models \phi.$
\item $\mathbb{R}\models \phi.$
\item $Y\models \phi$ for every finite topological space $Y.$
\item $Y\models \phi$ for every Alexandrov space $Y.$
\end{enumarab}
\end{theorem}

One can say that finite topological space or their natural extension to Alexandrov topological spaces reflect faithfully
the $\sf S4$ semantics, and that arbitrary topological spaces generalize $\sf S4$ frames. On the other hand,
every topological space gives rise to a normal modal logic. Indeed $\sf S4$ is the modal logic of $\mathbb{R}$, or
any metric that is  separable and dense in itself
space, or all topological spaces, as indicated above. Also a
recent result is that it is also the modal logic of the Cantor set, which is known to be Baire isomorphic to
$\mathbb{R}$.

But, on the other hand,  modal logic is too weak to detect interesting properties
of $\mathbb{R}$, for example it cannot distinguish between $[0, 1]$ and $\mathbb{R}$ despite
their topological dissimilarities, the most striking one being compactness; $[0, 1]$ is compact, but $\mathbb{R}$ is
not.

The huge field of {\it Temporal logic},
of which dynamic topological logic is an example is broadly used to
cover approaches to the representation of temporal information with a logica (modal) framework.

When asked to
think of time in an abstract way, many people will probably figure our a picture of a line. The mathematics of this picture is given by a set
of time points, together with an irreflexive transitive order. To represent time on a frame
we take a relational structure of the form
$\mathfrak{T}=(T, <)$ where $T$ is the set of worlds namely the `moments' and $<$ is
is an irreflexive transitive relation called the {\it precedence relation}.
If $(s, t)\in <$, we say that $s$ is earlier than $t$.
Essentially temporal logic extends classical propositional logic with a set of temporal operators
that navigate between worlds using the accessibility relation $<$.

To capture such structures modally,
in the  syntax, one introduces two modal operators $G, H$ with
intended meanings for $G, H$ are are follows :

$G$ is {\it it will always be the case that} and $H$
is {\it it has always been the case that. }

The spice of temporal logic,
however, lies in a basic idea,
namely, to use {\it new}, non classical
connectives to relate the truth of formulated in possibly distinct moments.

The operators denoted by $F$ (short for future)
and $P$ (short for past) are such, defined by $F\phi$ is $\neg G\neg \phi$ and $P\phi$ is
$\neg H\neg \phi$, where $\phi$ is a formula in the modal language.
Here  $F\phi$ is to be read as `at some time in the future $\phi$ will be the case,
and $P\phi$ is read as `at some time in the past $\phi$ holds'. This can be easily seen by reading their duals
$G$ and $H$ simply  as `henceforth'  and `hitherto', respectively.

Formally semantics are captured by Kripke frames of the form $(T, <)$
where $<$ is an irreflexive transitive order. A model is a triple $\M=(T, <, i)$, where $i$ is a map from the propositional variables
to the set of worlds (moments), and  just truth relation
defined inductively at moment $t$ by:
$$\M, t\models q \Longleftrightarrow \phi \i(t)(q)=1$$
The Booleans are the usual and the temporal operators:
$$M,t\models G\phi \text { iff }(\forall s)(t<s, \M, s\models \phi).$$
$$M, t\models H\phi \text { iff }(\forall s)(s<t, \M,s\models \phi.)$$
Now some valid formulas in the intended interpretation are
$p\to HFp$: {\it what is has always been going to be,}
$p\to GPp$: {\it what is will always have been,}
$G(p\to q)\to Gp\to Gq$: {\it  whatever will always follow
from what always will be.}

And these in fact, together with $H(p\to q)\to Hp\to Hq)$,
constitute a complete set of axioms for the specially
significant {\it Minimal Tense Logic $K_t$}
which has the the following four axioms:
\begin{enumerate}
\item $p\to HF p,$
\item  $p\to GPp,$
\item $H(p\to q)\to (Hp\to Hq),$
\item $G(p\to q)\to (Gp\to Gq).$
\end{enumerate}
together with two rules of necessitation, or temporal inference, namely:

From $p$ derive $Hp$ and from $p$ derive $Gp$
and of course the rules of ordinary propositional
logic.

The theorems of $K_t$ express, essentially, those
properties of the tense operators which do not depend on any specific assumptions about
the temporal order. Of obvious interest
is {\it  tensed predicate logic}, where the tense operators are added to classical
first order logic. This enables us to express distinctions
concerning the logic of {\it both} time {\it and existence}, the latter reflected by the existential quantifier.
For example, the statement  {\it a minister  will be president} can be interpreted as
'Someone who is now a minister  will be a president at some future time'
Formally: $$\exists x(minister(x)\land F(president(x)).$$
If we replace a minister by philosopher and president by king then another away to express
{\it A philosopher  will be a king}
is
$$\exists xF(mininster(x))\land king(x)).$$
In words: there now exists someone who will
at some future time be both a philosopher and a king.

$F(\exists x(minister(x)\land FKing(x))$ works for both; it reads `there will exist someone who is a minister (philosopher) and later will be a
president (king)'.

We shall study finite variable fragments of both predicate and topologica logic 

\subsection {Atom canonicity and omitting types}

Assume that we have a class of Boolean algebras with operators for which we have a semantical notion of representability
(like Boolean set algebras or cylindric set algebras).
A weakly representable atom structure is an atom structure such that at least one atomic algebra based on it is representable.
It is strongly representable  if all atomic algebras having this atom structure are representable.
The former is equivalent to that the term algebra, that is, the algebra generated by the atoms,
in the complex algebra is representable, while the latter is equivalent  to that the complex algebra is representable.

Could an atom structure be possibly weakly representable but {\it not} strongly representable?
Ian Hodkinson \cite{Hodkinson}, showed that this can indeed happen for both cylindric  algebras of finite dimension $\geq 3$, and relation algebras,
in the context of showing that the class of representable algebras, in both cases,  is not closed under \d\ completions.
In fact, he showed that this can be witnessed on an atomic algebras, so that
the variety of representable relation algebras algebras and cylindric algebras of finite dimension $>2$
are not atom-canonical. (The complex algebra of an atom structure is
the completion of the term algebra.)
This construction is somewhat complicated using a rainbow atom structure. It has the striking consequence
that there are two atomic algebras sharing the same
atom structure, one is representable the other is not.

This construction was simplified and streamlined,
by many authors, including the author, but Hodkinson's  construction,
as we indicate below, has the supreme advantage that it has a huge potential to prove analogous
theorems on \de\ completions, and atom-canonicity
for several varieties of topological algebras
including properly the variety of representable cylindric-like algebras, whose members have a neat embedding property,
such as polyadic algebras with and without equality and Pinter's substitution algebras.
In fact, in such cases atomic {\it representable countable algebras}
will be constructed so that their \de\ completions are outside such varieties.

Let $\sf K_n$ be the class of either topological or tense cylindric algebras of dimension $n$, 
$\sf RK_n$ the class of representable $K_n$s, 
and $\Rd_{ca}$ stands for forming cylindric reducts.
We show,
that for $n>2$ finite, there is $\A\in \sf RK_n$ wiuth countably many atoms  such that
$\Rd_{ca}\Cm\At\A\notin S\Nr_n\CA_{n+4}$ inferring that the varieties
$S\Nr_n\K_{n+k}$, for $n>2$ finite for any $k\geq 4,$
not closed under \d\ completions.
Such results, as illustrated below
will have non-trivial (to say the least)  repercussions on omitting types for finite variable fragments.

Lately, it has become fashionable in algebraic logic to
study representations of abstract algebras that has a complete representation \cite{Sayed} for an extensive overview.
A representation of $\A$ is roughly an injective homomorphism from $f:\A\to \wp(V)$
where $V$ is a set of $n$-ary sequences; $n$ is the dimension of $\A$, and the operations on $\wp(V)$
are concrete  and set theoretically
defined, like the Boolean intersection and cylindrifiers or projections.
A complete representation is one that preserves  arbitrary disjuncts carrying
them to set theoretic unions.
If $f:\A\to \wp(V)$ is such a representation, then $\A$ is necessarily
atomic and $\bigcup_{x\in \At\A}f(x)=V$.

Let us focus on cylindric algebras for some time to come.
It is known that there are countable atomic $\RCA_n$s when $n>2$,
that have no complete representations;
in fact, the class of completely representable $\CA_n$s when $n>2$, is not even elementary \cite[corollary 3.7.1]{HHbook2}.

Such a phenomena is also closely
related to the algebraic notion of {\it atom-canonicity}, as indicated, which is an important persistence property in modal logic
and to the metalogical property of  omitting types in finite variable fragments of first order logic
\cite[theorems 3.1.1-2, p.211, theorems 3.2.8, 9, 10]{Sayed}.
Recall that a variety $V$ of Boolean algebras with operators is atom-canonical,
if whenever $\A\in V$, and $\A$ is atomic, then the complex algebra of its atom structure, $\Cm\At\A$ for short, is also
in $V$.

If $\At$ is a weakly representable but
not strongly  representable, then
$\Cm\At$ is not representable; this gives that $\RCA_n$ for $n>2$ $n$ finite, is {\it not} atom-canonical.
Also $\Cm\At\A$  is the \d\ completion of $\A$, and so obviously $\RCA_n$ is not closed under \d\ completions.

On the other hand, $\A$
cannot be completely  representable for, it can be shown without much ado,  that
a complete representation of $\A$ induces a representation  of $\Cm\At\A$ \cite[definition 3.5.1, and p.74]{HHbook2}.

Finally, if $\A$ is countable, atomic and has no complete representation
then the set of co-atoms (a co-atom is the complement of an atom), viewed in the corresponding Tarski-Lindenbaum algebra,
$\Fm_T$, as a set of formulas, is a non principal-type that cannot be omitted in any model of $T$; here $T$ is consistent
if $|A|>1$. This last connection  was first established by the  author leading up to \cite{ANT} and more, see e.g \cite{HHbook2}.

The reference \cite{Sayed} contains an extensive discussion of such notions.

\section{Guarded and locally guarded fragments of topological logic}

We also focus on cylindric algebras of dimension $m$ in this part.
We explain why  non atom-canonicity of the varieties
$S\Nr_m\CA_{n+k}$ $k\geq 3$,
non-elementarily
of any class containing the class of completely representable algebras and contained in $S_c\Nr_n\CA_{n+3}$;
these are algebras having $n+3$ local relativized representations
and failure of omitting types in $n+3$ flat frames, all three proved here go together.

The uniform cause of this phenomena, is that in a two player game over a rainbow atom strcture, \pa\ can win an \ef\ forth
game using and re-using $n+3$ pebbles, preventing \pa\ from
establishing  the negation of the above
negative properties.

This is too abrupt so let us start from the beginning of the story ultimately reaching the inter-connections
of such results navigating between algebra and logic. Indeed results in algebraic logic are most attractive when they have impact on logic,
particularly first order logic. But here the feedback between algebraic logic and logic works both ways.

We prove that  infinitely many varieties containing and including the variety of representable algebras
are {\it not} atom canonical; such varieties posses {\it $n$  relativized}
representations.

Here $n$ measures the degree of the  {\it $n$ squareness or $n$ flatness} of the model. This is a locally guarded
relativization;  $n$ is the measure of how much we have to zoom in by a `movable window'
to this $n$ relativized representation to mistake it for a real genuine
one which is  both infinitely flat and square.

When $n<\omega,$ $n$ flatness is
stronger than $n$ squareness by a threatening `Church Rosser condition', but at the limit of
ordinary representations this huge
discrepancy disappears.

In $n$ squareness witnesses for cylindrifiers can only be found in $n$
{\it Gaifman graphs} on $m$ cliques $m<n$, where an $m$ clique is an $m$
version of usual cliques in graph theory. In $n$ flatness it is further required
that  cylindrifiers also commute on the $n$ square (the Church Rosser condition).

Another way of viewing an algebra $\A$ possessing such representations is that such representations
are induced by $n$ hypergraphs,
whose $m$ cliques are labelled by the top element of the algebra.
If $\A$ is countable then an $\omega$ relativized (whether flat or square) representation is just
a classical (Tarskian) representation.

Actually the $n$ flat relativization of an algebra $\A$
also measures the degrees of freedom that $\A$ has.

These are hidden but are coded
in the Gaifman graph which can be looked upon as an $n$ hypergraph
whose hyperedges are $n$ ary assignments from the base of the representation
to formulas in a
relational first order  signature  having an $m$ relation symbols for each element $a\in\A$ and using
$n$ variables.

Algebraically, $n$ measures the dimension
for which there is an algebra having  this dimension
in which $\A$ neatly embeds into, it is how much we truncate
$\omega$, it is the distance between $S\Nr_n\CA_{\omega}=\sf RCA_m$
and $S\Nr_m\CA_n$, which is infinite in the sense that for any finite $n>m$ the variety $S\Nr_m\CA_n$ is not finitely axiomatizable
over $\sf RCA_m$. This can be proved using
Monk-like algebras constructed by Hirsch and Hodkinson \cite{HHbook}.

This is a syntactical measure, and in this case we have the two following weak $n$
completeness theorems that can be seen as an $n$ approximations of Henkin's algebraic  completeness
theorem implemented via the neat embedding
theorem ($S\Nr_n\CA_{\omega}=\sf RCA_m$):

$\A\in S\Nr_m\CA_n$ iff it has an $n$ flat representation
and $\A\in S_c\Nr_m\CA_n$
if it has a complete $m$ flat representation.

On the other hand in case of $n$ squareness
the algebra also has $n$ degrees of freedom,  but such degrees of freedom are `looser' so to speak,
for the algebra neatly embeds into another $n$ dimensional algebra $\D$,
but such a 'dilation' $\D$
is not a classical $\CA_n$ for  it may fail commutativity of cylindrifiers, though not entirely in the sense that a weakened version
of commutativity of cylindrifiers, or a confluence property, holds here.

In fact, such dilations, which happen to be representable in a relativized sense,
are globally guarded fragments of $n$ variable first order logic.
However, in the case of $n$ flatness,
such dilations are $\CA_n$s; so that they are not guarded.

This discrepancy in the formed  dilations blatantly manifests itself in
a very important property.
The Church Rosser condition in the $n$ flat case of commutativity of cylindrifiers , when $n\geq m+3$
make this locally guarded fragment strongly undecidable,
not finitely axiomatizable  in $k$th order logic for any finite $k$,
and there are finite algebras that have infinite $n$ flat representations but does not
have finite ones, and the equational theory of algebras having
only  infinite
$n$ flat representations is not recursively enumerable because the equational theory
of those algebras having finite $n$ flat representations is recursively enumerable and
the problem of telling whether a finite  algebra has an $n$ flat representation is undecidable.

In  $n$ squareness it is still undecidable and not finitely axiomatizable,  but
every finite algebra that has  an $n$ square representation has a finite one.

The reason for such a discrepancy is that the equational theory of the $n$ dimensional dilations is undecidable while
the universal (hence equational) theory of the dilations in case of $n$
squareness is decidable.
The first is just the equational theory of $\CA_n$ the second is the equation theory of
$\sf D_n$, the class of algebras introduced by Andr\'eka et al in \cite{AT}.

We also show in that infinitely many classes containing and including
the completely representable algebras
and characterized by having {\it complete $n$ relativized whether square or flat representations} are not elementary.
Here complete (relativized) representation refers to the fact
that such representations preserve arbitrary (possibly infinitary meets) carrying them to
set theoretic intersection. A necessary (but not sufficient) condition
for an algebra to posses such complete relativized representations is
atomicity, and so such complete relativized representations are also atomic in the following sense.

Every sequence in the top element of the representing algebra is in the range of an atom,
equivalently the pre-image of such  sequence, always an ultrafilter,
turns out to be a principal one.

Such representations
when $n=\omega$, and $\A$ is countable,
then a complete $\omega$ representation of $\A$ is just a complete representation.

Analogous remarks of properties mentioned above for $n$ flat representations
works here by replacing
representation by complete representation and neat embeddings by complete neat
embeddings.

In both cases of an existence of $n$ square or $n$ flat
representations of an algebra $\A$
such semantics can be viewed as  both {\it locally guarded}
and  clique guarded semantics.

But here guards, semantically,
exist in the $n$ dilation, the
$n$-ary assignments that satisfy formulas in $\L(\A)^n$ are
restricted to the the $n$ Gaifman hypergraph
$$\C(\M)=\{s\in {}^nM :\rng(s) \text { is an $m$ clique }\}.$$
Here $M=\bigcup_{s\in V} \rng(s)$ where $V$ is the set of permissable assignments.
Every $m$ variable formula $\phi$ can  be effectively translated to one in the packed fragment of $m$ variable
first order logic, call it ${\sf packed}(\phi)$.

Here we have a very interesting connection (*)
$$M, \C(\M), s\models \phi \Longleftrightarrow
M, s\models {\sf packed}(\phi),$$
which says that locally
guarded fragments are a subfragment
of the {\it packed fragment}
which is an extension of the loosely
${\sf GF}$.

Let us twist to topological logic $\TL_m$.
The analogous algebraic results discussed above for $\L_m (\CA_m)$
hold here as well. Here $\L_m$ denotes first order logic restricted to the first $m$ variables.

Quite surprisingly perhaps at first glance
we will draw the same conclusion from both results, namely,
that a version of the famous Orey-Henkin omitting types theorem ($OTT$)
fails even if count in
$m+k$, $k\geq 3$ square, {\it a priori} square models as potential candidates for omitting
single non principle types; and these can be chosen to consist
of co-atoms, so that what we actually have, is that the  $\TL_m$ theory for which $OTT$ fails is an atomic one, that does not have an $m+k$
flat atomic model, so  we are showing that a famous theorem of Vaught which holds for first order logic -
countable atomic theory have countable models-
fails for $\TL_m$ when $m>2$.

The topologizing of such multimodal logics are not finitely axiomatizable,
nor decidable but their modal algebras have the finite algebra finite base property which means that if a finite algebra has
an $n$ square representation,
then it has a finite one.

But $n$ flat representations in the topological context as well is far more intricate.
In fact, for $n=3$ it is undecidable  to tell whether a finite frame is has an $3+m$
flat representation when $m\geq 3$.

In the case of global guarding
negative properties of finite variable
fragments of first order logic having at least
three variables,  in topological predicate logic starting
from two variables,  like robust undecidability simply disappear, while one retains
a lot of  positive properties concerning definability properties, like interpolation, Beth definability,
completeness
and  $OTT$.
So in guarding we throw away some of the worlds but the accessibility relations are kept as they are
but now restricted to the remaining worlds.

\section{Preliminaries}

Throught $\alpha, \beta$ denote arbitary ordinals. We formulate the basic definitions in
full generality, by throughout the paper ordinals considered will be finite or countable.

\subsection{Topological cylindric algebras}

Instead of taking ordinary set algebras,
as in the case of cylindric algebras, with units of the form $^{\alpha}U$, one
may require {\it that the base $U$ is endowed with some topology}. This enriches the algebraic structure.
For given such an algebra,
for each $k<\alpha$, one defines an {\it interior operator} on $\wp(^{\alpha}U)$ by
$$I_k(X)=\{s\in {}^{\alpha}U; s_k\in {\sf int}\{a\in U: s_a^k\in X\}\}, X\subseteq {}^{\alpha}U.$$
Here $s_a^k$ is the sequence that agrees with $s$ except possibly at  $k$ where its value is $a$.
This gives a {\it topological cylindric set algebra of dimension $\alpha$}.

Now such algebras lend itself to an abstract formulation aiming to capture the concrete set algebras; or rather the variety
generated by them.

This consists of expanding the signature of cylindric algebras
by unary operators,  or modalities, one for each $k<\alpha$, satisfying certain identities.

We start with the standard definition of cylindric algebras \cite[Definition 1.1.1]{HMT1}:

\begin{definition}
Let $\alpha$ be an ordinal. A {\it cylindric algebra of dimension $\alpha$}, a  $\CA_{\alpha}$ for short,
is defined to be an algebra
$$\C=\langle C, +, \cdot, -, 0, 1, \cyl{i}, {\sf d}_{ij}  \rangle_{i,j\in \alpha}$$
obeying the following axioms for every $x,y\in C$, $i,j,k<\alpha$

\begin{enumerate}

\item The equations defining Boolean algebras,

\item $\cyl{i}0=0,$

\item $x\leq \cyl{i}x,$

\item $\cyl{i}(x\cdot \cyl{i}y)=\cyl{i}x\cdot \cyl{i}y,$

\item $\cyl{i}\cyl{j}x=\cyl{j}\cyl{i}x,$

\item ${\sf d}_{ii}=1,$

\item if $k\neq i,j$ then ${\sf d}_{ij}=\cyl{k}({\sf d}_{ik}\cdot {\sf d}_{jk}),$

\item If $i\neq j$, then ${\sf c}_i({\sf d}_{ij}\cdot x)\cdot {\sf c}_i({\sf d}_{ij}\cdot -x)=0.$

\end{enumerate}
\end{definition}
For a cylindric algebra $\A$, we set ${\sf q}_ix=-{\sf c}_i-x$ and ${\sf s}_i^j(x)={\sf c}_i({\sf d}_{ij}\cdot x)$.
Now we want to abstract equationally the prominent features of the concrete interior operators defined on cylindric
set and weak set algebras.
We expand the signature of $\CA_{\alpha}$
by a unary operation $I_i$
for each $i\in \alpha.$ In what follows $\oplus$ denotes the operation of symmetric difference, that is,
$a\oplus b=(\neg a+b)\cdot (\neg b+a)$.
For $\A\in \CA_{\alpha}$ and $p\in \A$, $\Delta p$, {\it the dimension set of $p$}, is defined to be the set
$\{i\in \alpha: {\sf c}_ip\neq p\}.$ In polyadic terminology $\Delta p$ is called {\it the support of $p$}, and if $i\in \Delta p$, then $i$
is said to {\it  support $p$} \cite{g}.

\begin{definition}\label{topology} A {\it topological cylindric algebra of dimension $\alpha$}, $\alpha$ an ordinal,
is an algebra of the form $(\A,I_i)_{i<\alpha}$ where $\A\in \sf CA_{\alpha}$
and for each $i<\alpha$, $I_i$ is a unary operation on
$A$ called an {\it  interior operator} satisfying the following equations for all $p, q\in A$ and $i, j\in \alpha$:
\begin{enumerate}
\item ${\sf q}_i(p\oplus q)\leq {\sf q}_i (I_ip\oplus I_iq),$
\item $I_ip\leq p,$
\item $I_ip\cdot I_ip=I_i(p\cdot q),$
\item $p\leq I_iI_ip,$
\item $I_i1=1,$
\item ${\sf c}_kI_ip=I_ip, k\neq i, k\notin \Delta p,$
\item ${\sf s}_j^iI_ip=I_j{\sf s}_j^ip, j\notin \Delta p.$
\end{enumerate}
\end{definition}
The class of all such topological cylindric
algebras are denoted by ${\sf TCA}_{\alpha}.$

For $\B=(\A, I_i)_{i<\alpha}\in {\sf TCA}_{\alpha}$ we write
$\Rd_{ca}\B$ for $\A$. Notice too that every $\CA_{\alpha}$ can be extended to a
$\sf TCA_{\alpha}$, by defining for all $i<\alpha$,
$I_i$ to be the identity function.

Topological algebras  in the form we defined are {\it not }
Boolean algebras with operators because the interior operators do not distribute
over the Boolean join.

We also need the notion of {\it compressing} dimensions
and, dually,  {\it dilating them}; expressed by the notion of neat reducts.

\begin{definition}
\begin{enumarab}
\item Let $\alpha<\beta$ be ordinals and $\B\in \sf TCA_{\beta}$. Then $\Nr_{\alpha}\B$ is the algebra with universe
$Nr_{\alpha}\A=\{a\in \A: \Delta a\subseteq \alpha\}$ and operations obtained by discarding the operations of $\B$
indexed by ordinals in $\beta\sim \alpha$.
$\Nr_{\alpha}\B$ is called the {\it neat $\alpha$ reduct of $\B$}. If $\A\subseteq \Nr_{\alpha}\B$, with $\B\in \sf TCA_{\beta}$,
then we say that $\B$ is  {\it a $\beta$
dilation of $\A$}, or simply {\it a dilation} of $\A$.
\item An {\it injective} homomorphism $f:\A\to \Nr_{\alpha}\B$ is called a {\it neat embedding}; if such an $f$ exists, then we say
that $\A$ {\it neatly embeds into its dilation $\B$}. In particular, if $\A\subseteq \Nr_{\alpha}\B,$ then $\A$ neatly embeds into
$\B$ via the inclusion map.
\end{enumarab}
\end{definition}
Note that the algebra $\Nr_{\alpha}\B$ is well defined; it is closed under the cylindric operations; this is well
known and indeed easy to show, and it also closed under all the {\it interior operators} $I_i$ for $i<\alpha$, for
if $x\in \Nr_{\alpha}\B$, and $k\in \beta\sim \alpha$, then by axiom (6) of definition \ref{topology},
$k\notin \alpha\supseteq \Delta x\cup \{i\}\supseteq \Delta(I_i(x))$,
hence ${\sf c}_k(I_i(x))=I_i(x).$

\subsection{Cylindric tense algebras}

We start with semantics for tense cylindric algebras.

A time flow is a pair $(T, <)$ where $T$ is a non empty (set of moments) and $<$ is an irreflexive
transitive relation. If $s<t$ we say that $<$ is earlier than $\A$.
\begin{definition}\label{tensesemantics}
A {\it tense  system based on a time flow $(T, <)$}
is a tuple $\mathfrak{K}=(X_t, V_{st}, <, >, Q_1, Q_2, \bold 0)_{s,t\in T}$
such that
\begin{enumroman}
\item $T$ is a non-empty set, the set of moments
\item $<$ and $ >$ are two binary relations on $T,$
\item  $\bold 0\in T$ and $Q_1, Q_2\subseteq T,$
\item $X_t\neq \emptyset$ for all $t\in T,$
\item if $t<s$ or $s>t$ then $V_{ts}: X_t\to X_s$ is a function such that
\begin{enumarab}
\item  $V_{tt} =Id,$

\item if $V_{st}, V_{tr}$ are defined then $V_{tr}=V_{st}\circ V_{tr},$

\item If $t<s$ and $s>t$
then $V_{ts}$ is a bijection and $V_{ts}=V_{st}^{-1}.$
\end{enumarab}
\end{enumroman}
\end{definition}

Let $\alpha$ be an ordinal $>0$.
For $s, t\in T$ and  $x\in {}^{\alpha}X_t$,  $V_{ts}(x)$ denotes
the set $(V_{st}(x_i): i<\alpha)$
The one defines an algebra
$$\mathfrak{F}_{\mathfrak{K}}=\{(f_w:w\in W); f_w:{}^{\alpha}D_w\to \mathfrak{O}\}.$$
The operations are defined as follows:
If $x,y\in {}Y_w$ and $j\in \alpha$ then we write $x\equiv _jy$ if $x(i)=y(i)$ for all $i\neq j$.
We write $(f_w)$ instead of $(f_w:w\in W)$.
In $\mathfrak{F}_{\mathfrak{K}}$ we consider the following operations:
$$(f_w)\lor (g_w)=(f_w\lor g_w)$$
$$(f_w)\land (g_w)=(f_w\land g_w.)$$
For any $(f_w)$ and $(g_w)\in \mathfrak{F}$, define
$$\neg (f_w)=(-f_w).$$
For any $i, j\in  {}^{\alpha}\alpha$, we define
$${\sf s}_i^j:\mathfrak{F}\to \mathfrak{F}$$by
$${\sf s}_i^j(f_w)=(g_w)$$
where
$$g_w(x)=f_w(x\circ [j,i])\text { for any }w\in W\text { and }x\in {}^{\alpha}D_w.$$
For any $j\in \alpha$ and $(f_w)\in \mathfrak{F}$
define
$${\sf c}_{j}(f_w)=(g_w)$$
where for $x\in {}^{\alpha}D_w$
$$g_w(x)=\bigvee\{f_w(y): y\in {}^{\alpha}D_w, y\equiv_jx\}.$$

$G(f_t: t\in T)=:(g_t: t\in T)$ such that for $x\in Y_t$ $g_t(x)=1$ iff $t\in Q_1$ and for all
$s>t$, $f_s(T_{ts}(x))=1.$
Expressed otherwise for any $x\in {}^{\alpha}D_t$
$$g_t(x)=\bigwedge\{f_{s}(x): t<s\}.$$
and likewise
$H(f_t: t\in T)=(g_t: t\in T)$
where
$$g_t(x)=\bigwedge\{f_{s}(x): s<t\}.$$
$${\sf d}_{ij}=\{(f_t: t\in T): f_t(i)=f_t(j) \forall t\in T\}.$$

The class of abstract $\TeCA_{\alpha}$ is obtained by a process of abstraction.
It consists of the cylindric axioms together with the following equations for
the modalities $G$ and $H$;
\begin{enumarab}
\item [$(T1)$]$G(x\cdot y)=G(x)\cdot G(y).$
\item [$(T2)$]$H(x\cdot y)=H(x)\cdot G(y).$
\item [$(T3)$]$Gx\leq GGx$
\item [$(T4)$]$x\leq GPx\text{ and }x\leq HF.$
\end{enumarab}
and interaction axioms, or rather non-interaction axioms,
for each $i<\alpha$
$${\sf c}_iG(x)=G(x)\& {\sf c}_iH(x)=H(x).$$

It is tedious, but basically routine to verify  that
$\mathfrak{F}_{\mathfrak{K}}$ endowed with the above operations satisfies the above
equations.
Let $\C=\mathfrak{F}_{\mathfrak{K}}$
be as above then it has a cylindric reduct which is isomorphic to $\prod_{t\in T}\wp(^{\alpha}X_t).$

We can also study maximal temperal logic with Still and Until, using the axiomatization in [V].
 \subsection{Temporal cylindric algebras}

A time flow is a pair $(T, <)$ where $T$ is a non empty (set of moments) and $<$ is an irreflexive
transitive relation. If $s<t$ we say that $<$ is earlier than $\A$.
\begin{definition}\label{tensesemantics}
A {\it tense  system based on a time flow $(T, <)$}
is a tuple $\mathfrak{K}=(X_t, V_{st}, <, >, Q_1, Q_2, \bold 0)_{s,t\in T}$
such that
\begin{enumroman}
\item $T$ is a non-empty set, the set of moments
\item $<$ and $ >$ are two binary relations on $T,$
\item  $\bold 0\in T$ and $Q_1, Q_2\subseteq T,$
\item $X_t\neq \emptyset$ for all $t\in T,$
\item if $t<s$ or $s>t$ then $V_{ts}: X_t\to X_s$ is a function such that
\begin{enumarab}
\item  $V_{tt} =Id,$

\item if $V_{st}, V_{tr}$ are defined then $V_{tr}=V_{st}\circ V_{tr},$

\item If $t<s$ and $s>t$
then $V_{ts}$ is a bijection and $V_{ts}=V_{st}^{-1}.$
\end{enumarab}
\end{enumroman}
\end{definition}

Let $\alpha$ be an ordinal $>0$. 
For $s, t\in T$ and  $x\in {}^{\alpha}X_t$,  $V_{ts}(x)$ denotes
the set $(V_{st}(x_i): i<\alpha)$
The one defines an algebra
$$\mathfrak{F}_{\mathfrak{K}}=\{(f_w:w\in W); f_w:{}^{\alpha}D_w\to \mathfrak{O}\}.$$
The operations are defined as follows:
$U[(f_t: t\in T), (h_t:t\in T)]=:(g_t: t\in T)$ 
such that for $x\in ^{\alpha}D_t$, 
$g_t(x)=1$ iff $t\in Q_1$ and 
$$(\exists s)[s>t, f_s(T_{ts}(x))=1\& (\forall u)(t<u<s, 
h_u(T_{us}(x))=1].$$
$S[(f_t: t\in T), (h_t: t\in T))=:(g_t: t\in T)$
such that for all $x\in {}^{\alpha}D_t$,
$g_t(x)=1$ iff $t\in Q_2$ and
 $$(\exists s)[s<t, f_s(T_{ts}(x))=1\& (\forall u)(s<u<t, 
h_u(T_{us}(x))=1].$$

It is tedious, but basically routine to verify  that
$\mathfrak{F}_{\mathfrak{K}}$ endowed with the above operations is a temporal algebra.

The completeness theorem is as before, so forming the 
family of dilations inductively. We have  
$U[(f_t: t\in T), (h_t:t\in T)]=:(g_t: t\in T)$ 
such that for $x\in ^{\alpha}D_t$, 
$g_t(x)=1$ iff $t\in Q_1$ and 
the following double infinitary join meet holds:
 $$g_t(x)=\bigvee_{s>t}[ f_s(x)=1\land  \bigwedge_{u<s}h_u(x)=1)],$$
and its mirror for $S$ is : 
$$g_t(x)=\bigvee_{s<t}[ f_s(x)=1\land  \bigwedge_{u>s}h_u(x)=1)],$$


Venema proved that the folowing axioms are complete for temporal logic with since and until over 
$(\mathbb{N}, <)$. We pick up this axiomatization for the propositional part and prove a completeness theorem for predicate temporal 
logic.

\begin{definition}

\begin{enumerate}

\item $G(p\to q)\to U(p,r)\to U(q, r)$

\item $G(p\to q)\to (U(r, p)\to U(r,q))$

\item $p\land U(q,r)\to U(a\land S(p, r), r)$

\item $U(p,q)\land \neg U(p,r)\to U(q\land \neg r, q)$

\item $U(p,q)\to U(p,rq\land U(p, q))$

\item $U(p, q)\land U(r, s)\to U(p\land r, q\land s)\lor U(p\land s, q\land s)\lor U(q\land r, q\land s)$  

\item mirror images

\item $F\top\to U(T, \bot)\land (PT\to S(T, \bot)$

\item $H\bot \lor PH\bot$
\item  $Fp\to U(p, \neg \phi)$

\end{enumerate}
\end{definition}

Let $\alpha$ be an infinite ordinal. Let $L$ be the signature consisting of the Boolean operations, cylindrifiers, diagonal elements, 
with indices from $\alpha$
and the two binary modalities
$S$ and $U$. The above axioms can be translated to equations 
in $L$. Now we take the $\CA_{\alpha}$ 
axioms together with such equations, call the finite resulting schema of equations
$\Sigma$. 
\subsection{ Discrete topologizing and static temporalizing}

We state  very simple
fact that allows us to recursively associate with every $\CA$ of any dimension a $\TCA$, a $\TeCA$
and a $\sf Tem\CA$
of the same dimension, such that the last threealgebras are representable if and inly if the original
$\CA$ is. This mechanical procedure wil be the  main technique twe use
to obtain negative results for both $\TCA$s, $\TeCA$s and $\sf Tem\CA$s
by bouncing them back to their cylindric  counterpart.

\begin{definition} Let $\A\in \CA_{\alpha}$.
\begin{enumarab}
\item  $\A^{\sf top}\in \TCA_{\alpha}$
is a  {\it topologizing of  $\A$} if
$\Rd_{ca}\A^{\sf top}=\A$. The {\it discrete topologizing} of $\A$ is the $\sf TCA_{\alpha}$ obtained
from $\A$ by expanding $\A$ with $\alpha$ many identity
operators.
\item $\A^{\sf tense}$ is a {\it tense expansion} of $\A$, if $\Rd_{ca}\A^{\sf tense}=\A$. The {\it static temporal expansion} of $\A$
is the $\TeCA_{\alpha}$ obtainded from $\A$
by defining $G=H=Id,$, taking the time $T$ to consist of only one moment, that is,
$T=\{t\}$ say, and the flow $<$
is taken to
be the empty set. 

\item In the same way we can form maximal temporal extensions of a cylindric algebra, by defininig $S(a,b)=U(a,b)=Id$
and defining the flow like in the previous item, call the resulting algebra $\A^{\sf temp}$. 
In this case, $\A^{\sf tense}=\A^{\sf temp}$ so that 
the minimal and maximal resulting  temporal logic coincide.
\end{enumarab}
\end{definition}

\begin{theorem} The discrete topologizing of $\A\in \CA_{\alpha}$ is unique up to isomorphism.
Furthermore, if $\A^{\sf top}$ is the discrete topologizing of $\A$,
then $\A$ is representable if and only if $\A^{\sf top}$
is representable. A completely analogous statement
holds for $\A\in \CA_{\alpha}$ in connection to $\A^{\sf tense}$ and $\A^{\sf temp}.$
\end{theorem}
\begin{proof} The first part is trivial.
The second part is also very easy.
If $\A^{\sf top}$ is representable then obviously $\A=\Rd_{ca}\A^{\sf top}$ is representable.
For the last part if $\A$ is representable with base
$U$ then $\A^{\sf top}$ have the same universe of $\A$,
hence it is representable by endowing $U$ with
the discrete topology,
which induces the identity interior operators.
\end{proof}

\begin{theorem} Let $1<m<n$
\begin{enumarab}
\item Assume that $\A\in \Nr_m\CA_n$,
then $\A^{\sf top}\in \Nr_m\TCA_n$
\item Assume that $\A\in S\Nr_m\CA_n,$ 
then $\A^{\sf top} \in S\Nr_m\CA_n$. 
\item Same if we apply the operation $S_c$ of forming complete subalgebars in 
$\Nr_m\CA_n$.
\item Completely analogous statements hold 
for $\sf TeCA_m$s and $\sf Tem\CA$s by static temporalizing.
\end{enumarab}
\end{theorem}
\begin{proof} We prove only the first same. The rest of the proof is the same.
Assume that $\A=\Nr_m\B$, $\B\in \CA_n$, then $\A^{\sf top}=\Nr_m\B^{\sf top}$ and we are done.
\end{proof}

Form this easy lemma one can infer quite deep results proved for cylindric algebras. We mention two.

\begin{itemize}

\item For any  
pair of ordinal $1<m<n$ the class  
$\Nr_m\TCA_n$ not closed under elementary subalgebras \cite{Sayedneat}.

\item  For any pair 
of ordinals $2<m<n$ (infinite incuded), for any $r\in \omega$ and 
for any finite  $k\geq 1$, 
there is a $\B^r\in S\Nr_m\TCA_{m+k}$ such that $\Rd_{ca}\A\notin S\Nr_m\CA_{m+k+1}$ and
$\Pi_{r\in \omega}\B^r/U\in \sf RTCA_m$ for any non principal ultrafilter on $\omega$.

In particular, for finite for $m$ and any finite $k\geq 1$, we can infer that  
the variety $S\Nr_m\TCA_{m+k+1}$ is not finitely axiomatizable 
over the variety $S\Nr_m\TCA_{m+k}.$

\item The same result holds  for infinite dimensions by replacing 
finite axiomatizability 
by finite schema axiomatizability \cite{HHbook, t}.

\end{itemize}

The process of discrete topologizing , and for that matter static
temporal expansions,
work best in recovering negative results
proved for cylindric algebras to the topological or temporal
expansion by bouncing it back to the
cylindric case. We will witness such a phenomena quite frequently.

\section{Complete representability, atom canonicity and neat embeddings}

\subsection{Use of rainbows}

In the proof of our main results mentioned in the abstract (concerning atom canonicity and compolete representations)
we use advanced sophisticated machinery of cylindric algebra theory, like so called  rainbow constructions
invented by Hirsch and Hodkinson \cite{HH, HHbook, Hodkinson, HHbook2},
obtaining new results for cylindric-like algebras, strengthening results proved for cylindric algebras
in \cite{ANT, HH, Hodkinson}, and then
lifting them to the topological and tense
context, by discrete  topologizing and static temporal expansions.

Rainbow algebras are only superficially similar to what is known in the literature of Monk-like algebras.
Monk-like algebras are efficient in proving that certain algebras may not be representable and 
this type of results is proved by an application of the cominatorial Ramey's theorem.
The idea of Mon-like algebras is not too hard. Such algebras are finite, hence atomic and the atoms are coloured in 
such a way to forbid monochromatic trangles (traingles all of its three sides are labelled 
by the same colour) If the atoms are more than the colours then a represenation 
will force a monochromatic triangle which a impossible.

On a very basic level in rainbow algebras, or more accurately on 
the atom structures of rainbow algebras, deterministic
games are played between two players Elloise \pe\ and Abelard \pa\; one of them has to win, there are no draws.

We have almost all rainbow colours, red, green, white, black, if one considers that black
and white are colours etc.
These games lift very simple forth \ef\ game played on
two coloured relational structures (usually complete irreflexive graphs)
$A$ is the `greens' and $B$ is the `reds'  to the cylindric rainbow
algebras.

It is a forth \ef\  pebble game such that \ws\ s for either player in the private \ef\ game
are preserved in the lifted rainbow algebra $\CA_{A,B}$, but the number of used pebbles and rounds
increases.

Because we can control the number of pebbles in play, rainbow algebras
prove very delicate results via quite sophisticated constructions from the cylindric algebra
point of view,  but such constructions which tend to seem complicated enough,  actually
use simple games \ef\ games that serves the task at hand by  appropriately
choosing the structures $A$ and $B$.

Reducing complicated constructions, proving  sophisticated subtle
statements solving really hard
problems, to a manageable simple
case, namely, a forth  \ef\ pebble game,
is precisely the ingenuity of such constructions.

This technique
proved to be a nut cracker in addressing difficult problems, and it especially proved highly efficient
in contexts when it is not obvious how to use Monk-like algebras, in both the relation algebra and cylindric algebra
cases like for instance  proving that $\sf RA_{n+1}$ is not finitely axiomatizable
over $\sf RA_n$, when $n\geq $
where $\sf RA_m$($m>2)$ is the variety  of relation algebras that embed
into algebras having a relational basis in the sense of Maddux \cite{HHbook}.

We will introduce a cylindric analogue of $\sf RA_n$, further on to be topologized, these are
cylindric algebras of dimension $m$, that also has an $n$ dimensional basis
that can be characterized in many ways that are
equivalent.

First by simple games played on finite graphs with a set of nodes $n$, in which \pa\ is offered only a
a cylindrifier move' there are no amalgamation moves,
second by certain $n$ dimensional basis that are obtained from the cylindric basis of
Maddux by discarding the amalgamation condition, third
by a weak neat embedding property; such algebras embed
into $n$ dimensional algebras and
they have {\it localized $n$ square representations} to be elaborated
upon in a short  while.

It is always the case, like just indicated,
that the number of pebbles used in the \ef\ private
game, appears on the algebra level,
so for example if we want to show that a rainbow  algebra is not in $S\Nr_m\CA_n$ then
we use $n$ pebbles.

It will turn out that the following three results:

\begin{itemize}

\item  non atom canonicity proved by constructing a rainbow atom structure 
whose term algebra is representable but its \de\ completion 
does not neatly embed into 
an $n+3$ dimensional algebra, 

\item the non-existence of $n+3$ flat models

\item and the failure of $OTT$ even if we allow clique
guarded $n+3$ flat semantic,
\end{itemize}
is witnessed by a \ws\ for \pa\
in a certain essentially forth \ef\
translated to the coloured graphs involved in the rainbow construction,
where he can use and re-use
$n+3$ pebbles, though in the case of atom canonicity \pa\
can win only in finitely many rounds because the game is played on a finite algebra.

In this case actually \pa\ can win
without having to reuse pebbles, and this makes the situation worse.
Even $n+3$ square models are not enough for omitting
types.

In the other case \pa\ can win only in an $\omega$ rounded game.
He cannot win the game truncated to finitely many rounds, because his strategy is forcing
\pa\ a decreasing sequence in the red $\N$; the double
indices of the red colours come from $\N$. The strategy  of \pa\
is bombarding \pe\ with `cones' having the same base and green tints.
Here the suffices of the greens come from the green $\Z$.

By the  rules of the game \pe\  wil have to choose a red colour to label the edges between apexes of cones having
the same base.
This cannot be achieved in finitely many rounds.

For such a construction we can infer that algebras having $n$ flat
complete representations when $n\geq m+3$ is not elementary. 
The $n$ square case does not follow here, though it can be proved using a rainbow
construction, and this particular instance Monk-like algebras can do the job just as well.

Coloured graphs are complete graphs whose edges are labelled by the rainbow 
colours and some of its $n-1$ hyperedges are also coloured.
Shades of yellow are reserved for that.  Such coluored graphs can be also seen 
as models for an $L_{\omega_1, \omega}$ 
theory formulated in the rainbow signature, which consists of binary relation, green, red, white, etc 
and $n-1$-ary relation representing the shades of yellow\cite{HHbook2}. 

When the greens are finite (as is the case with our first encounter with rainbows), the rainbow theory is
a first order theory. If an edge in a coloured graph is coloured by a green, 
then this is interpreted model-theoretically, that  this edge holds in the green binary relation in the corresponding model.

The atoms of the rainbow
algebra are very roughly {\it finite} coloured graphs. Cones are special coloured graphs. The strategy
for \pa\ is always bombarding \pe\ with cones whose tints are green
and because,\pe\ can never play greens according to the rules of the game,
so when she is to label an edge between apexes of cones having the same base
she has to choose a red. It is always the case that \pa\
wins on a `red clique.'

\subsection{Atom canonicity}

We will show using the so called {\it blow up and blur construction}, a very indicative name suggested 
in \cite{ANT},
that for any finite $n>2$, any $\sf \K\in \{\Sc, \CA, \sf PA, \PEA\}$, and any 
$k\geq 3$,
$S\Nr_n\sf K_{n+k}$ is {\it not} atom canonical. Here $\Sc$ denotes Pinter's substiutuion algebars, $\sf PA$ denotes polyadic algebras
and $\sf PEA$ denotes polyadic algebras with equality.

We will blow up and blur a finite 
{\it rainbow algebra}.

We give the general idea for cylindric algebras, though the idea is much more universal as we will see.
Assume that $\sf RCA_n\subseteq \K$, and $\K$ is closed under forming subalgebras.
Start with a finite algebra $\C$ outside $\K$. Blow up and blur $\C$, by splitting
each atom to infinitely many, getting a new atom structure $\At$. In this process a (finite) set of blurs are used.

They do not blur the complex algebra, in the sense that $\C$ is there on this global level.
The algebra $\mathfrak{Cm}\At$ will not be in $\K$
because $\C\notin \K$ and $\C$ embeds into $\mathfrak{Cm}\At$.
Here the completeness of the complex algebra will play a major role,
because every element of $\C$,  is mapped, roughly, to {\it the join} of
its splitted copies which exist in $\mathfrak{Cm}\At$ because it is complete.

These precarious joins prohibiting membership in $\K$ {\it do not }exist in the term algebra, only finite-cofinite joins do,
so that the blurs blur $\C$ on 
this level; $\C$ does not embed in $\Tm\At.$

In fact, the the term algebra will  not only be in $\K$, but actually it will be in the possibly smaller $\sf RCA_n$.
This is where the blurs play their other role. Basically non-principal ultrafilters, the blurs
are used as colours to represent  $\Tm\At$.

In the process of representation we cannot use {\it only} principal ultrafilters,
because $\Tm\At$ {\it cannot be completely representable}, that is, it cannot have a representation that preserves
all (possibly infinitary) meets carrying them to set theoretic intersections, 
for otherwise this would give that $\mathfrak{Cm}\At$
is representable.

But the blurs will actually provide a {\it complete representation} of the {\it canonical extension}
of $\Tm\At$, in symbols $\Tm\At^+$; the algebra whose underlying set consists of all ultrafilters of $\Tm\At$. The atoms of $\Tm\At$
are coded in the principal ones,  and the remaining non- principal ultrafilters, or the blurs,
will be finite, used as colours to completely represent $\Tm\At^+$, in the process representing $\Tm\At$.

We start off with a conditional theorem: giving a concrete instance of a blow up and blur construction for relation algebras due to Hirsch and Hodkinson. 
The proof is terse  highlighting only the main ideas.
\begin{theorem}\label{decidability} Let $m\geq 3$. Assume that for any simple atomic relation algebra with atom structure $S$,
there is a cylindric atom structure $H$, constructed effectively from $S$,  such that:
\begin{enumarab}
\item If $\Tm S\in \sf RRA$, then $\Tm H\in \RCA_m$,
\item If $S$ is finite, then $H$ is finite,
\item $\Cm S$ is embeddable in $\Ra$ reduct of $\Cm H$.
\end{enumarab}
Then  for all $k\geq 3$, $S\Nr_m\CA_{m+k}$ is not closed under completions,
\end{theorem}

\begin{proof}
Let $S$ be a relation atom structure such that $\Tm S$ is representable while $\Cm S\notin \sf RA_6$.
Such an atom structure exists \cite[Lemmas 17.34-35-36-37]{HHbook}.

We give a brief sketch at how such algebras are constructed by allowing complete irreflexive graphs having an arbitrary finite set
nodes, slightly generalizing the proof in {\it op.cit}, though the proof idea is essentially the same.

Another change is that we refer to non-principal ultrafilters (intentionally) by {\it blurs} to emphasize the connection with the blow up and
blur construction in \cite{ANT} as well as with the blow up and blur construction outlined above, to be encountered 
in full detail in a litle while, witness theorem \ref{can}.

In all cases a finite algebra is blown up and blurred to give a representable algebra (the term algebra on the blown up and blurred finite atom
structure) whose \d\ completion does not have a neat embedding
property.

We use the notation of the above cited lemmas in \cite{HHbook2} without warning, and our proof will be very brief just stressing the main ideas.
$G_r^n$ denotes the usual atomic $r$ rounded 
game played on atomic networks having $n$ nodes of an atomic relation algebra, 
where $n,r\leq \omega$, and  $K_r$ ($r\in \omega)$ 
denotes the complete ireflexive graph with
$r$ nodes. 

Let $\R$ be the rainbow algebra $\A_{K_m, K_n}$, $m>n>2$.
Let $T$ be the term algebra obtained by splitting the reds. Then $T$ has exactly two blurs
$\delta$ and $\rho$. $\rho$ is a flexible non-principal ultrafilter consisting of reds with distinct indices and $\delta$ is the reds
with common indices.
Furthermore, $T$ is representable, but $\Cm\At T\notin  S\Ra\CA_{m+2}$, in particular, it is not representable 
\cite [Lemma 17.32]{HHbook2}.
Now we use  the \ef\ forth pebble game $EF_r^k(A, B)$
where $A$ and $B$ are relational structures. This game has 
$r$ rounds and $k$ pebbles. 
The rainbow theorem \cite[Theorem 16.5]{HHbook}
says that \pe\ has a \ws\ in the game $G_{1+r}^{2+p}(\A_{A, B})$ if and only if she has a \ws\ in $EF_r^p(A,B)$.

Using this theorem it is obvious that  \pe\ has a \ws\ over $\At\R$  in $m+2$ rounds,
hence $\R\notin \sf RA_{m+2}$, hence is not
in $S\Ra\CA_{m+2}$.  $\Cm\At T$ is also not in the latter class
for $\R$ embeds into it, by mapping every red
to the join of its copies. Let $D=\{r_{ll}^n: n<\omega, l\in n\}$, and $R=\{r_{lm}^n, l,m\in n, l\neq m\}$.
If $X\subseteq R$, then $X\in T$ if and only if $X$ is finite or cofinite in $R$ and same for subsets of $D$ \cite[lemma 17.35]{HHbook}.
Let $\delta=\{X\in T: X\cap D \text { is cofinite in $D$}\}$,
and $\rho=\{X\in T: X\cap R\text { is cofinite in $R$}\}$.
Then these are {\it the} non principal ultrafilters, they are the blurs and they are enough
to (to be used as colours), together with the principal ones, to represent $T$ as follows \cite[bottom of p. 533]{HHbook2}.
Let $\Delta$ be the graph $n\times \omega\cup m\times \{{\omega}\}$.
Let $\B$ be the full rainbow algebras over $\At\A_{K_m, \Delta}$ by
deleting all red atoms $r_{ij}$ where $i,j$ are
in different connected components of $\Delta$.

Obviously \pe\ has a \ws\ in ${\sf EF}_{\omega}^{\omega}(K_m, K_m)$, and so it has a \ws\ in
$G_{\omega}^{\omega}(\A_{K_m, K_m})$.
But $\At\A_{K_m, K_m}\subseteq \At\B\subseteq \At_{K_m, \Delta}$, and so $\B$ is representable.

One  then defines a bounded morphism from $\At\B$ to the the canonical extension
of $T$, which we denote by $T^+$, consisting of all ultrafilters of $T$. The blurs are images of elements from
$K_m\times \{\omega\}$, by mapping the red with equal double index,
to $\delta$, for distinct indices to $\rho$.
The first copy is reserved to define the rest of the red atoms the obvious way.
(The underlying idea is that this graph codes the principal ultrafilters in the first component, and the non principal ones in the second.)
The other atoms are the same in both structures. Let $S=\Cm\At T$, then $\Cm S\notin S\Ra\CA_{m+2}$ \cite[lemma 17.36]{HHbook}.

Note here that the \d\ completion of $T$ is not representable while its canonical extension is
{\it completely representable}, via the representation defined above.
However, $T$ itself is {\it not} completely representable, for a complete representation of $T$ induces a representation of its \d\ completion,
namely, $\Cm\At\A$.

Now let $H$ be the $\CA_m$ atom structure obtained from $\At T$ provided by the hypothesis of the 
theorem.
Then $\Tm H\in \RCA_m$. We claim that $\Cm H\notin {\bf S}\Nr_m\CA_{m+k}$, $k\geq 3$.
For assume not, i.e. assume that $\Cm H\in {\bf S}\Nr_m\CA_{m+k}$, $k\geq 3$.
We have $\Cm S$ is embeddable in $\Ra\Cm H.$  But then the latter is in ${\bf S}\Ra\CA_6$
and so is $\Cm S$, which is not the case.
\end{proof}

Hodkinson constructs atom structures of cylindric and polyadic algebras of any pre-assigned finite dimension $>2$ 
from atom structures
of relation algebras \cite{AU}. One could well be tempted to use such a construction with the above proof to obtain an analogous
result for cylindric and polyadic algebras. 
However, we emphasize that the next result {\it cannot} be obtained by lifting the relation algebra case 
\cite[ lemmas 17.32, 17.34, 17.35, 17.36]{HHbook}
to cylindric algebras 
using  Hodkinson's construction in \cite{AU} as it stands. It is true that Hodkinson constructs from 
every atomic relation algebra an atomic cylindric algebra of dimension $n$, for any $n\geq 3$, 
but the relation algebras {\it does not} embed into the $\sf Ra$ reduct of the constructed
cylindric algebra when $n\geq 6$. If it did, then the $\sf Ra$ result would lift as indeed is the case with 
$n=3$.  

Now we are faced with two options. Either modify 
Hodkinson's construction, implying that the embeddability of the given relation algebra
in the $\sf Ra$ reduct of the constructed cylindric algebras, or avoid completely 
the route via relation algebras. We tend to think that it is impossible to adapt Hodkinson's construction the way needed, because if $\A$
is a non representable relation algebra, and $\Cm \At(\A)$ embeds into the $\Ra$ reduct of a cylindric algebra of every dimension $>2$, then
$\A$ will be representable, which is a contradiction. 

Therefore we choose the second option. We instead start from scratch. We blow up and blur a finite  rainbow cylindric algebra.

In \cite{HHbook2} the rainbow cylindric algebra of dimension $n$ on a graph $\Gamma$ is denoted by $\R(\Gamma)$.
We consider $\R(\Gamma)$ to be in ${\sf PEA}_n$ by expanding it with the polyadic operations defined the obvious way (see below).
In what follows we consider $\Gamma$ to be the indices of the reds, and for a complete irreflexive graph
$\G$, by ${\sf PEA}_{\G, \Gamma}$
we mean the rainbow cylindric algebra $\R(\Gamma)$  of dimension $n$,
where ${\sf G}=\{\g_i: 1\leq i<n-1\}\cup \{\g_0^i: i\in \G\}$.

More generally, we consider a rainbow polyadic algebra based on relational structures 
$A, B$, to be the rainbow algebra
with signature the binary
colours (binary relation symbols)
$\{\r_{ij}: i,j\in B\}\cup \{\w_i: i<n-1\}\cup \{\g_i:1\leq i<n-1\}\cup \{\g_0^i : i\in A\}$ and $n-1$
shades of yellow ($n-1$ ary relation symbols)
$\{\y_S: S\subseteq_{\omega} A, \text { or } S=A\}.$ 

We look at models of the rainbow theorem as coloured graphs \cite{HH}.
This class is denoted by ${\sf CRG}_{A,B}$ or simply $\sf CRG$ when $A$ and $B$ are clear from context.

A coloured graph is a graph such that each of its edges is labelled by one of the first three colours mentioned above, namely, 
greens, whites or reds, and some $n-1$ hyperedges are also
labelled by the shades of yellow.
Certain coloured graphs will deserve 
special attention.

\begin{definition}
Let $i\in A$, and let $M$ be a coloured graph  consisting of $n$ nodes
$x_0,\ldots,  x_{n-2}, z$. We call $M$ an {\it $i$ - cone} if $M(x_0, z)=\g_0^i$
and for every $1\leq j\leq n-2$, $M(x_j, z)=\g_j$,
and no other edge of $M$
is coloured green.
$(x_0,\ldots, x_{n-2})$
is called {\it the center of the cone}, $z$ {\it the apex of the cone}
and {\it $i$ the tint of the cone.}
\end{definition}

\begin{definition}\label{def}
The class of coloured graphs $\sf CRG$ are

\begin{itemize}

\item $M$ is a complete graph.

\item $M$ contains no triangles (called forbidden triples)
of the following types:
\vspace{-.2in}
\begin{eqnarray}
&&\nonumber\\
(\g, \g^{'}, \g^{*}), (\g_i, \g_{i}, \w_i)
&&\mbox{any }1\leq i< n-1\  \\
(\g^j_0, \g^k_0, \w_0)&&\mbox{ any } j, k\in A\\
\label{forb:match}(\r_{ij}, \r_{j'k'}, \r_{i^*k^*})&& i,j,j',k',i^*, k^*\in B\\ \mbox{unless }i=i^*,\; j=j'\mbox{ and }k'=k^*
\end{eqnarray}
and no other triple of atoms is forbidden.

\item If $a_0,\ldots,   a_{n-2}\in M$ are distinct, and no edge $(a_i, a_j)$ $i<j<n$
is coloured green, then the sequence $(a_0, \ldots, a_{n-2})$
is coloured a unique shade of yellow.
No other $(n-1)$ tuples are coloured shades of yellow.

\item If $D=\set{d_0,\ldots,  d_{n-2}, \delta}\subseteq M$ and
$M\upharpoonright D$ is an $i$ cone with apex $\delta$, inducing the order
$d_0,\ldots,  d_{n-2}$ on its base, and the tuple
$(d_0,\ldots, d_{n-2})$ is coloured by a unique shade
$\y_S$ then $i\in S.$
\end{itemize}

One then can define a polyadic equality atom structure
of dimension $n$ from the class $\sf CRG$. It is a {\it rainbow atom structure}.  Rainbow atom structures  are what Hirsch and Hodkinson call
atom structures built from a class of models \cite{HHbook2}.
Our models are, according to the original more traditional view \cite{HH}
coloured graphs. So let $\sf CRG$ be the class of coloured graphs as defined above.
Let $$\At=\{a:n \to M, M\in \sf CRG: \text { $a$ is surjective}\}.$$
We write $M_a$ for the element of $\At$ for which
$a:n\to M$ is a surjection.
Let $a, b\in \At$ define the
following equivalence relation: $a \sim b$ if and only if
\begin{itemize}
\item $a(i)=a(j)\Longleftrightarrow b(i)=b(j),$

\item $M_a(a(i), a(j))=M_b(b(i), b(j))$ whenever defined,

\item $M_a(a(k_0),\dots, a(k_{n-2}))=M_b(b(k_0),\ldots, b(k_{n-2}))$ whenever
defined.
\end{itemize}
Let $\At$ be the set of equivalences classes. Then define
$$[a]\in E_{ij} \text { iff } a(i)=a(j).$$
$$[a]T_i[b] \text { iff }a\upharpoonright n\smallsetminus \{i\}=b\upharpoonright n\smallsetminus \{i\}.$$
Define accessibility relations corresponding to the polyadic (transpositions) operations as follows:
$$[a]S_{ij}[b] \text { iff } a\circ [i,j]=b.$$
This, as easily checked, defines a $\sf PEA_n$
atom structure. The complex algebra of this atom structure is denoted by ${\sf PEA}_{A, B}$ where $A$ is the greens and 
$B$ is the reds.
\end{definition}

One can define a  $\sf TCA_n$ atom struture
by discrete topologizaing
setting  for all $i<n$, 
$[a] In_i [b] \text { iff } a=b$, where $In_i$ is the accessibility relation corresponding to $I_i$
Similar remark hold for tense and temporal algebrs.
For example in the former case one defines $[a]G[b]$ iff $a=b$ and same for $H$. The time 
condists of one moment and the flow is the 
empty set.

Now  consider the following two games on coloured graphs, each with $\omega$ rounds, and limited number of pebbles
$m>n$. They are translations of $\omega$ atomic games played on atomic networks
of a rainbow algebra using a limited number of nodes $m$.
Both games offer \pa\ only one move, namely, a cylindrifier move.

From the graph game perspective both games \cite[p.27-29]{HH} build a nested sequence $M_0\subseteq M_1\subseteq \ldots $.
of coloured graphs.

First game $G^m$.
\pa\ picks a graph $M_0\in \sf CRG$ with $M_0\subseteq m$ and
$\exists$ makes no response
to this move. In a subsequent round, let the last graph built be $M_i$.
\pa\ picks
\begin{itemize}
\item a graph $\Phi\in \sf CRG$ with $|\Phi|=n,$
\item a single node $k\in \Phi,$
\item a coloured graph embedding $\theta:\Phi\smallsetminus \{k\}\to M_i.$
Let $F=\phi\smallsetminus \{k\}$. Then $F$ is called a face.
\pe\ must respond by amalgamating
$M_i$ and $\Phi$ with the embedding $\theta$. In other words she has to define a
graph $M_{i+1}\in C$ and embeddings $\lambda:M_i\to M_{i+1}$
$\mu:\phi \to M_{i+1}$, such that $\lambda\circ \theta=\mu\upharpoonright F.$
\end{itemize}
$F^m$ is like $G^m$, but \pa\ is allowed to resuse nodes.

$F^m$ has an equivalent formulation on atomic networks of atomic algebras.

Let $\delta$ be a map. Then $\delta[i\to d]$ is defined as follows. $\delta[i\to d](x)=\delta(x)$
if $x\neq i$ and $\delta[i\to d](i)=d$. We write $\delta_i^j$ for $\delta[i\to \delta_j]$.

\begin{definition}
Let $2< n<\omega.$ Let $\C$ be an atomic ${\sf PEA}_{n}$.
An \emph{atomic  network} over $\C$ is a map
$$N: {}^{n}\Delta\to \At\C,$$
where $\Delta$ is a non-empty set called a set of nodes,
such that the following hold for each $i,j<n$, $\delta\in {}^{n}\Delta$
and $d\in \Delta$:
\begin{itemize}
\item $N(\delta^i_j)\leq {\sf d}_{ij},$
\item $N(\delta[i\to d])\leq {\sf c}_iN(\delta),$
\item $N(\bar{x}\circ [i,j])= {\sf s}_{[i,j]}N(\bar{x})$ for all $i,j<n$.

\end{itemize}
\end{definition}
\begin{definition}\label{def:games}
Let $2\leq n<\omega$. For any ${\sf Sc}_n$
atom structure $\alpha$ and $n<m\leq
\omega$, we define a two-player game $F^m(\alpha)$,
each with $\omega$ rounds.

Let $m\leq \omega$.
In a play of $F^m(\alpha)$ the two players construct a sequence of
networks $N_0, N_1,\ldots$ where $\nodes(N_i)$ is a finite subset of
$m=\set{j:j<m}$, for each $i$.

In the initial round of this game \pa\
picks any atom $a\in\alpha$ and \pe\ must play a finite network $N_0$ with
$\nodes(N_0)\subseteq  m$,
such that $N_0(\bar{d}) = a$
for some $\bar{d}\in{}^{n}\nodes(N_0)$.

In a subsequent round of a play of $F^m(\alpha)$, \pa\ can pick a
previously played network $N$ an index $l<n$, a {\it face}
$F=\langle f_0,\ldots, f_{n-2} \rangle \in{}^{n-2}\nodes(N),\; k\in
m\sim\set{f_0,\ldots, f_{n-2}}$, and an atom $b\in\alpha$ such that
$$b\leq {\sf c}_lN(f_0,\ldots, f_i, x,\ldots, f_{n-2}).$$
The choice of $x$ here is arbitrary,
as the second part of the definition of an atomic network together with the fact
that $\cyl i(\cyl i x)=\cyl ix$ ensures that the right hand side does not depend on $x$.

This move is called a \emph{cylindrifier move} and is denoted
$$(N, \langle f_0, \ldots, f_{n-2}\rangle, k, b, l)$$
or simply by $(N, F,k, b, l)$.
In order to make a legal response, \pe\ must play a
network $M\supseteq N$ such that
$M(f_0,\ldots, f_{i-1}, k, f_{i+1},\ldots, f_{n-2}))=b$
and $\nodes(M)=\nodes(N)\cup\set k$.

\pe\ wins $F^m(\alpha)$ if she responds with a legal move in each of the
$\omega$ rounds.  If she fails to make a legal response in any
round then \pa\ wins.
\end{definition}

We start by proving the following algebraic result. Recall that $\TCA_n$, $\TeCA_n$ and $\sf TemCA_n$ 
stand for the classes of topological cylindric, dense cylindric, and  temporal cylindric algebras of dimension $n$, respectively.
For $\K$ any of the above, $\sf RK$ stands 
for the class of representable algebras in $\sf K$ 
of dimension $n$.

\begin{theorem}\label{can}
For any finite $n>2$, for any $\K$ between $S\Nr_n\TCA_{n+3}$ and $\sf TRCA_n$
is not atom-canonical, hence is not closed under minimal completions, and is not Sahlqvist
axiomatizable.  In more detail, there exists an atomic countable
completely additive algebra $\A\in \sf RTCA_n$ such its \de\ completion, namely, the complex algebra of its atom structiure
is not in $S\Nr_n\TCA_{n+3}$. A completely  analogous result holds for $\TeCA_n$ and $\sf TemCA_n$.
\end{theorem}
\begin{proof}
The proof uses a rainbow algebra \cite{HHbook2}.

We {\it blow up and blur} in the sense of \cite{ANT} a finite rainbow cylindric algebra namely $R(\Gamma)$
where $\Gamma$ is the complete irreflexive graph $n+1$, and the greens
are  ${\sf G}=\{\g_i:1\leq i<n-1\}
\cup \{\g_0^{i}: 1\leq i\leq n+1\},$ we denote this finite algebra endowed by the topologization
induced  $n$ identity
interior operators by $\TCA_{n+1, n}.$

Let $\At$ be the rainbow atom structure similar to that in \cite{Hodkinson} except that we have $n+1$ greens and
only $n$ indices for reds, so that the rainbow signature now consists of $\g_i: 1\leq i<n-1$, $\g_0^i: 1\leq i\leq n+1$,
$\w_i: i<n-1$,  $\r_{kl}^t: k<l\in n$, $t\in \omega$,
binary relations and $\y_S$, $S\subseteq n+1$,
$n-1$ ary relations.

We also have a shade of red $\rho$; the latter is a binary relation but is {\it outside the rainbow signature},
though  it is used to label coloured graphs during a certain game devised to prove representability
of the term algebra \cite{Hodkinson}, and in fact \pe\ can win the $\omega$ rounded game
and build the $n$ homogeneous model $M$ by using $\rho$ whenever
she is forced a red, as will be shown in a while.

So $\At$ is obtained from the rainbow atom structure of the algebra $\A$ defined in \cite[section 4.2 starting p. 25]{Hodkinson}
truncating the greens to be finite (exactly $n+1$ greens). In \cite{Hodkinson} it shown that the complex algebra $\Cm\At\A$ is not representable;
the result obtained now, because the greens are finite but still outfit the red, is sharper;
it will imply that
$\Cm\At\notin S\Nr_n\TCA_{n+3}$.

The logics $L^n, L^n_{\infty \omega}$ are taken in the rainbow
signature (without $\rho$).

Now $\Tm\At\in \sf TRCA_n$; this can be proved exactly like  in \cite{Hodkinson}.
Strictly speaking the cylindric reduct of $\Tm\At$ can be proved representable like in
\cite{Hodkinson}; giving, as usual, the base of the representation the discrete topology we get representability of the
interior operators as well.
The colours used for coloured graphs involved in building  the finite atom structure of the algebra $\sf TCA_{n+1, n}$, which is the rainbow
algebra constructed on the irreflexive complete graphs $n+1$, the greens, and $n$, the reds, are:
\begin{itemize}

\item greens: $\g_i$ ($1\leq i\leq n-2)$, $\g_0^i$, $1\leq i\leq n+1$,

\item whites : $\w_i: i\leq n-2,$

\item reds:  $\r_{ij}$ $i<j\in n$,

\item shades of yellow : $\y_S: S\subseteq n+2$.

\end{itemize}
with {\it forbidden triples}
\vspace{-.2in}
\begin{eqnarray*}
&&\nonumber\\
(\g, \g^{'}, \g^{*}), (\g_i, \g_{i}, \w_i),
&&\mbox{any }1\leq i\leq  n-2  \\
(\g^j_0, \g^k_0, \w_0)&&\mbox{ any } 1\leq j, k\leq n+1\\
\label{forb:match}(\r_{ij}, \r_{j'k'}, \r_{i^*k^*}), &&i,j,i', k', i^*, j^*\in n,\\ \mbox{unless }i=i^*,\; j=j'\mbox{ and }k'=k^*.
\end{eqnarray*}
and no other triple is forbidden.

Coloured graphs using such colours and the finite atom structure of $TCA_{n+1, n}$,
build up of (quotients of ) surjections from $n$ into coloured graphs are defined the usual way
\cite{HH}.

A coloured graph is red if  at least one of its edges is labelled red.
For brevity write $\r$ for $\r_{jk}$($j<k<n$).
If $\Gamma$ is a coloured graph using the colours in $\At \TCA_{n+1, n}$, and $a:n\to \Gamma$ is in $\At\TCA_{n+1,n}$,
then $a':n\to \Gamma'$ with $\Gamma'\in \sf CGR$
is a {\it copy} of $a:n\to \Gamma$ if $|\Gamma|=|\Gamma'|$,  all non red edges and $n-1$ tuples have the same colour (whenever defined)
and  for all $i<j<n$, for every red $\r$, if  $(a(i), a(j))\in \r$,
then there exits $l\in \omega$ such that $(a'(i), a'(j))\in \r^l$. Here we implicitly require that for distinct $i,j,k<n$, if
$(a(i),a(j))\in \r$, $(a(j), a(k))\in \r'$, $(a(i), a(k))\in \r''$, and
$(a'(i), a'(j))\in \r^l_1$, $(a'(j), a'(k))\in [\r']^{l_2}$ and $(a'(i), a'(k))\in [\r'']^{l_3}$, then $l_1=l_2=l_3=l$, say,
so that $(\r^l, [\r']^l, [\r'']^l)$
is a consistent triangle in $\Gamma'$.
If $a':n\to \Gamma'$ and $\Gamma'$ is a red graph
using the colours of the rainbow signature of $\At$, whose reds are $\{\r_{kj}^l: k<j<n, l\in \omega\},$
then there is a unique $a: n\to \Gamma$, $\Gamma$
a red graph using the red colours in the rainbow signature of $\sf TCA_{n+1,n}$, namely, $\{\r_{kj}: k<j< n\}$
such that $a'$ is a copy of $a$.
We denote $a$ by $o(a')$, $o$ short for {\it original}; $a$ is the original of its copy $a'$.

For $i<n$, let $T_i$ be the accessibility relation corresponding to the $i$th cylindrifier in $\At$.
Let
$T_i^{s}$, be that corresponding to the $i$th cylindrifier in ${\sf TCA}_{n+1, n}$.
Then if $c:n\to \Gamma$ and $d: n\to \Gamma'$ are surjective maps
$\Gamma, \Gamma'$ are coloured graphs for ${\sf TCA}_{n+1, n}$, that are not red, then for any $i<n$, we have
$$([c],[d])\in T_i\Longleftrightarrow ([c],[d])\in T_i^s.$$

If $\Gamma$ is red using the colours for the rainbow signature of $\At$ (without $\rho$)
and $a':n\to \Gamma$,then for any $b:n\to \Gamma'$ where $\Gamma'$ is not red and any $i<n$, we have
$$([a'], [b])\in T_i\Longleftrightarrow  ([o(a')], [b])\in T_i^{s}.$$
Extending the notation, for $a:n\to \Gamma$ a graph that is not red in $\At$, set $o(a)=a$.
Then for any $a:n\to \Gamma$, $b:n\to \Gamma'$, where $\Gamma, \Gamma'$ are
coloured graphs at least  one of which is not red in $\At$ and any $i<n$, we have
$$[a]T_i[b]\Longleftrightarrow [o(a)]T_i^s[o(b)].$$

Now we deal with the last case, when the two graphs involved are red.
Now assume that $a':n\to \Gamma$ is as above, that is $\Gamma\in {\sf CGR}$ is red,
$b:n\to \Gamma'$ and $\Gamma'$ is red too, using the colours in the rainbow signature $\At$.

Say that two maps $a:n\to \Gamma$, $b:n\to \Gamma'$, with $\Gamma$ and $\Gamma'\in \sf CGR$ having the same size
are  $\r$ related if all non red edges and $n-1$ tuples have the same colours (whenever defined), and
for all every red $\r$, whenever  $i<j<n$, $l\in \omega$,  and $(a(i), a(j))\in \r^l$, then there exists
$k\in \omega$ such that $(b(i), b(j))\in \r^k$.
Let $i<n$. Assume that $([o(a')], [o(b)])\in T_i^s$. Then there exists $c:n\to \Gamma$ that is $\r$ related to $a'$
such that $[c]T_i[b]$.  Conversely, if $[c]T_i[b]$, then $[o(c)]T_i[o(b)].$

Hence, by complete additivity of cylindrifiers,  the map $\Theta: \At({\sf TCA}_{n+1, n})\to \Cm\At$ defined via
\[
 \Theta(\{[a]\})=
  \begin{cases}
    \{ [a']: \text { $a'$   copy of $a$}\}  \text { if $a$ is red, } \\
        \{[a]\} \text { otherwise. } \\

  \end{cases}
\]
\\
induces an embedding from ${\sf TCA}_{n+1, n}$ to $\Cm\At$, which we denote also by
$\Theta$.

We first check preservation of diagonal elements.
If $a'$ is a copy of $a$, $i, j<n$, and  $a(i)=a(j)$, then $a'(i)=a'(j)$.

We next  check cylindrifiers. We show that for all $i<n$ and $[a]\in \At(\TCA_{n+1, n})$ we have:
$$\Theta({\sf c}_i[a])=\bigcup \{\Theta([b]):[b]\in \At\TCA_{n+1,n}, [b]\leq {\sf c}_i[a]\}= {\sf c}_i\Theta ([a]).$$
Let $i<n$. If $[b]\in \At\TCA_{n+1,n}$,  $[b]\leq {\sf c}_i[a]$, and $b':n\to \Gamma$, $\Gamma\in \sf CGR$, is a copy of $b$,
then there exists  $a':n\to \Gamma'$, $\Gamma'\in \sf CGR$, a copy of  $a$ such that
$b'\upharpoonright  n\setminus \{i\}=a'\upharpoonright n\setminus \{i\}$. Thus $\Theta([b])\leq {\sf c}_i\Theta([a])$.

Conversely, if $d:n\to \Gamma$, $\Gamma\in \sf CGR$ and $[d]\in {\sf c}_i\Theta([a])$, then there exist $a'$ a copy of $a$ such that
$d\upharpoonright n\setminus \{i\}=a'\upharpoonright n\setminus \{i\}$.
Hence $o(d)\upharpoonright n\setminus  \{i\}=a\upharpoonright n\setminus \{i\}$, and so $[d]\in \Theta({\sf c}_i[a]),$
and we are done.

But now we can show that \pa\ can win a certain game on $\At({\sf TCA}_{n+1,n})$ in only $n+2$ rounds
as follows.
Viewed as an \ef\ forth game  pebble game, with finitely many rounds and pairs of pebbles,
played on the two complete irreflexive graphs $n+1$ and $n$, in each round $0,1\ldots n$, \pa\ places a  new pebble  on  an element of $n+1$.
The edge relation in $n$ is irreflexive so to avoid losing
\pe\ must respond by placing the other  pebble of the pair on an unused element of $n$.
After $n$ rounds there will be no such element,
and she loses in the next round.
This game, denoted by $F^{n+3}$ is the usual graph game in \cite{HH}
except that the nodes are limited to $n+3$ and \pa\ can re-use nodes.
So it is an atomic game with a limited number of pebbles allowing \pa\ to reuse them.
But in fact \pa\ will win without needing to re use pebbles.

We show that \pa\ can win the graph game on $\At({\sf TCA}_{n+1,n})$ in $n+2$ rounds using  $n+3$ nodes.

\pa\ forces a win on a red clique using his excess of greens by bombarding \pe\
with $\alpha$ cones having the same base ($1\leq \alpha\leq n+2)$.

In his zeroth move, \pa\ plays a graph $\Gamma$ with
nodes $0, 1,\ldots, n-1$ and such that $\Gamma(i, j) = \w_0 (i < j <
n-1), \Gamma(i, n-1) = \g_i ( i = 1,\ldots, n-2), \Gamma(0, n-1) =
\g^0_0$, and $ \Gamma(0, 1,\ldots, n-2) = \y_{n+2}$. This is a $0$-cone
with base $\{0,\ldots , n-2\}$. In the following moves, \pa\
repeatedly chooses the face $(0, 1,\ldots, n-2)$ and demands a node
$\alpha$ with $\Phi(i,\alpha) = \g_i$, $(i=1,\ldots n-2)$ and $\Phi(0, \alpha) = \g^\alpha_0$,
in the graph notation -- i.e., an $\alpha$-cone, without loss $n-1<\alpha\leq  n+1$,  on the same base.
\pe\ among other things, has to colour all the edges
connecting new nodes $\alpha, \beta$ created by \pa\ as apexes of cones based on the face $(0,1,\ldots, n-2)$, that is $\alpha,
\beta\geq n-2$.
By the rules of the game
the only permissible colours would be red. Using this, \pa\ can force a
win in $n+2$ rounds, using $n+3$ nodes  without needing to re-use them,
thus forcing \pe\ to deliver an inconsistent triple
of reds.

Let $\B={\sf TCA}_{n+1, n}$.
Then $\B$   is
outside $S\Nr_n\TCA_{n+3}$ for if it was then because
iit is  finite it would be in $S_c\Nr_n\TCA_{n+3}$
because $\B$ is the same as its canonical extension $\D$, say, and $\D\in S_c\Nr_n\TCA_{n+3}$.
But then \pe\ would have won \cite[Theorem 33]{r}. The last theorem is formulated for relation algebras but it can be easily modified
to the cylindric case.

Hence $\mathfrak{Cm}\At\notin S\Nr_n\TCA_{n+3}$,
because $\Rd_{ca}\B$ is embeddable in its $\CA$ reduct
and $S\Nr_n\CA_{n+3}$ is a variety; in particular, it is closed
under forming subalgebras.
It now readily follows that $Cm\At\notin S\Nr_n{\sf TCA}_{n+3}$.

Notice that  $\Rd_{df}\A$ is not completely representable, because if it were then $\A$,
generated by elements whose dimension sets $<n$,  as a $\TCA_n$ would be completely representable
and this induces a representation of its
Dedekind-MacNeille completion, namely,  $\Cm\At\A$.

We have shown that for any
finite $n>2$, any class $\sf K$ between the varieties
$S\Nr_n\TCA_{n+3}$ and $\sf TRCA_n$
is not atom-canonical hence not Sahqvist axiomatizable,
and we are done.

The same proof works for $\sf TeCA_n$ and 
$\sf TemCA_n$ by statically temporalizing the algebras dealt with.

Let us prove the tense case, the temporal case is the same.For $\A\in \CA_n$ let $\A^{\sf te}$ denotes its static temporalization.

Now let  $\C$ be the $\CA$ reduct of the topological complex algebra constructed,
then $\C^{\sf te}\notin S\Nr_n\sf TeCA_{n+3}$ for if it was then $\C^{\sf te}\subseteq \Nr_n\D$
for some $\D\in \sf TeCA_{n+3}$, hence
$\C\subseteq \Rd_{ca}\Nr_n\D=\Nr_n\Rd_{ca}\D$ which is impossible.

Statically temporalizing  the term algebra $\Rd_{ca}\A$, too we have that $\A^{\sf te}$
is representable (as a $\sf TeCA_n$), and we still have
that $\C^{\sf te}$ is the minimal completion of $\A^{\sf te}$,
via the same map that embeds $\A$ into $\C$.
By the neat embeding theorem for tense cylindric
algebras we are
done.
\end{proof}

Now we use another rainbow construction. Coloured graphs and rainbow algebras are defined like
above. The algebra constructed now is very similar to ${\sf CA}_{\omega,\omega}$ but is not identical; for in coloured graphs
we add a new triple of forbidden colours involving two greens and one red synchronized by an order preserving
function. In particular, we consider the underlying set of $\omega$ endowed
with two orders, the usual order and its converse. To prove the analogous result for relation algebra Robin Hirsch \cite{r}
uses
a rainbow-like algebra based on the ordered structures $\Z, \N$ (with natural order) .
We could have used these structures, but we chose to replace $\Z$ with the natural number endowed with the converse to $\leq$,
to make the analogy tighter with the rainbow algebra used in \cite{HH} proving a weaker result.

\begin{definition} Let $\A\in \sf TeCA_n$, and assume that
$\A\subseteq \prod_{i\in t}\A_i$ such that $\Rd_{ca}\A_i\in {\sf Cs}_n$
Then $\A$ is completely representable if for each $t\in T$, we have  $\Rd_{ca}\A_t$ is completely representable
\end{definition}
The following is easy to prove by using known results on complete representability for cylindric algebras \cite{HH}.
One such result is that complete representability of a cylindric algebra implies that is atomic, the converse however is not true
and $f:\A\to \wp(X)$ is a complete representation if and only if is a complete one.

We start with the following easy but very useful lemma:
\begin{theorem}
Assume that $n<m$. Let $\A^{\sf top}\in \sf TCA_n$ be obtained from $\A\in \CA_n$ by topologization. If $\A^{\sf top}\notin S\Nr_n \TCA_m$,
then $\A\notin S\Nr_n\CA_m.$ A completely analogous result holds for $\TeCA_n$.
\end{theorem}
\begin{proof} We prove the contrapositive. Assume that $\A\in S\Nr_n\CA_m$. Then $\A\subseteq \Nr_n\B$
where $\B\in \CA_m$, By re-topologizing $\A$ and
topologizing $\B$ we get the required.
\end{proof}

We approach the notion of complete representations for $\sf TCA_n$ when $n$ is finite. Rainbows \cite{HHbook, HHbook2}
will offer solace
here. Throughout this subsection $n$ will be finite and $>1$.
We identify notationally set algebras with their universes.

Let $\A\in \sf TCA_{n}$ and $f:\A\to \wp(V)$ be a representation of $\A$, where $V$ is
a generalized space of dimension $n$.
If $s\in {}V$ we let
$$f^{-1}(s)=\{a\in \A: s\in f(a)\}.$$
An {\it atomic representation} $f:\A\to \wp(V)$ is a representation such that for each
$s\in V$, the ultrafilter $f^{-1}(s)$ is principal.

A {\it complete representation} of $\A$ is a representation $f$ satisfying
$$f(\prod X)=\bigcap f[X]$$
whenever $X\subseteq \A$ and $\prod X$ is defined.

\begin{theorem}\label{complete} Let $\A\in \sf TCA_n$.
Let $f:\A\to \wp(V)$ be a representation of $\A$. Then $f$ is
a  complete representation iff $f$ is an atomic one.
Furthermore, if $\A$ is completely representable, then $\A$
is atomic and $\A\in S_c\Nr_n\sf TCA_{\omega}$.
\end{theorem}
\begin{proof} Witnesss \cite[Theorems 5.3.4, 5.3.6]{Sayedneat}, \cite[Theorem 3.1.1]{HHbook2} for the first three parts.
It remains to  show that if $\A$ is completely representable, then $\A\in S_c\Nr_n\TCA_{\omega}$.
Assume that $M$ is the base of
a complete representation of $\A$, whose
unit is a generalized space,
that is, $1^M=\bigcup_{i\in I} {}^nU_i$, where $U_i\cap U_j=\emptyset$ for distinct $i$ and $j$ in
$I$ where $I$ is an
index set $I$. Let $t:\A\to \wp(1^M)$ be the complete representation.
For each $i\in I$, $U_i$ carries the subspace topology of  $M$.
For $i\in I$, let $E_i={}^nU_i$, pick $f_i\in {}^{\omega}U_i$, let $W_i=\{f\in  {}^{\omega}U_i: |\{k\in \omega: f(k)\neq f_i(k)\}|<\omega\}$,
and let ${\C}_i$ be the $\TCA_n$ with universe $\wp(W_i)$, with the $\CA$ operations defined the usual way
on weak set algebras, and the interior operator is induced by the topology on $U_i$.
Then $\C_i$ is atomic; indeed the atoms are the singletons.

Let $x\in \Nr_n\C_i$, that is ${\sf c}_jx=x$ for all $n\leq j<\omega$.
Now if  $f\in x$ and $g\in W_i$ satisfy $g(k)=f(k) $ for all $k<n$, then $g\in x$.
Hence $\Nr_n \C_i$
is atomic;  its atoms are $\{g\in W_i:  \{g(0),\ldots g(n-1)\}\subseteq U_i\}.$
Define $h_i: \A\to \Nr_n\C_i$ by
$$h_i(a)=\{f\in W_i: \exists a\in \At\A: (f(0)\ldots f(n-1))\in t(a)\}.$$

Let $\C=\prod _i \C_i$. Let $\pi_i:\C\to \C_i$ be the $i$th projection map.
Now clearly  $\C$ is atomic, because it is a product of atomic algebras,
and its atoms are $(\pi_i(\beta): \beta\in \At(\C_i)\}$.
Now  $\A$ embeds into $\Nr_n\C$ via $I:a\mapsto (\pi_i(a) :i\in I)$.
and we may assume that the map is surjective.

If $a\in \Nr_n\C$,
then for each $i$, we have $\pi_i(x)\in \Nr_n\C_i$, and if $x$ is non
zero, then $\pi_i(x)\neq 0$. By atomicity of $\C_i$, there is a tuple $\bar{m}$ such that
$\{g\in W_i: g(k)=[\bar{m}]_k\}\subseteq \pi_i(x)$. Hence there is an atom
$a$ of $\A$, such that $\bar{m}\in t(a)$,  so $x\land  I(a)\neq 0$, and so the embedding is complete
and we are done. Note that in this argument {\it no cardinality condition} is required.
(The reverse inclusion does not hold in general for uncountable algebras, as will be shown in theorem \ref{complete},
though it holds for atomic algebras with countably many atoms as shown in our next theorem).
Hence $\A\in S_c\Nr_n\TCA_{\omega}$.

\end{proof}

Conversely, the following theorem can be obtained by topologizing \cite[Theorem 5.3.6]{Sayedneat}
and applying the omitting types theorem proved in Part 1.
\begin{theorem}\label{completerepresentation} If $\A$ is countable and atomic
and  $\A\in S_c\Nr_n\TCA_{\omega},$ then $\A$ is completely representable.
\end{theorem}

Now we will prove that the notion of complete representability, like in the case of cylindric algebras
of dimension $>2$, is {\it not} first order definable.
In fact, we prove much more, namely, that any class $\sf K$
containing the class of completely representable algebras of finite dimension $n>2$
and contained in $\bold S_c\Nr_n\TCA_{n+3}$ is not elementary.
This is meaningful since any completely representable algebra of dimension $n$
is in $\bold S_c\Nr_n\CA_{\omega}$ by theorem \ref{complete}
and obviously the latter class
contained in $\bold S_c\Nr_n\CA_{n+3}$.

Indeed this result is much stronger because for finite $n>2$ and any $k\leq \omega$, we have
$\bold S_c\Nr_n\TCA_{n+k+1}\subset \bold S\Nr_n\TCA_{n+k}$ for all $k\in \omega$ where $\subset$ denotes {\it proper inclusion}.

Also the strictness of such inclusions can be witnessed by `topologizing' the finite
algebras constructed  in \cite{t}, that is,  expanding them with identity functions as interior
operators; recall that this expansion preserves representability which can
be induced by imposing the discrete topology on the base of the representing
algebra of the original $\CA$.

Now we use another rainbow construction. Coloured graphs and rainbow algebras are defined like
above. The algebra constructed now is very similar to ${\sf CA}_{\omega,\omega}$ but is not identical; for in coloured graphs
we add a new triple of forbidden colours involving two greens and one red synchronized by an order preserving
function. In particular, we consider the underlying set of $\omega$ endowed
with two orders, the usual order and its converse. To prove the analogous result for relation algebra Robin Hirsch \cite{r}
uses
a rainbow-like algebra based on the ordered structures $\Z, \N$ (with natural order) .
We could have used these structures, but we chose to replace $\Z$ with the natural number endowed with the converse to $\leq$,
to make the analogy tighter with the rainbow algebra used in \cite{HH} proving a weaker result.

\begin{definition} Let $\A\in \sf TeCA_n$, and assume that
$\A\subseteq \prod_{i\in t}\A_i$ such that $\Rd_{ca}\A_i\in {\sf Cs}_n$
Then $\A$ is completely representable if for each $t\in T$, we have  $\Rd_{ca}\A_t$ is completely representable
\end{definition}
The following is easy to prove by using known results on complete representability for cylindric algebras \cite{HH}.
One such result is that complete representability of a cylindric algebra implies that is atomic, the converse however is not true
and $f:\A\to \wp(X)$ is a complete representation if and only if is a complete one.

\begin{theorem}

\begin{enumarab}
\item If $\A\in \TeCA_n$ is completely representable, $\A$  based on $T$, and  $\Rd_{ca}\A\cong \prod_{t\in T}\A_t$, then every component
$\A_t$ is atomic $(t\in T)$. Hence $\A$ itself is atomic, since it is a product of atomic algebras.
\item
Furthermore, if  $\pi_t:\A\to \A_t$ is the projection map, then
$\A$ is completely reprsentable if and only if
for all $X\subseteq \A$ for all non-zero $a\in \A$,  whenever  $\sum X=1$, then there exists a homomorphism $f:\A\to \mathfrak{F}_{\mathfrak{K}}$,
for some tense system $K$ such that   $\bigcup_{x\in X}f(\pi_t(x))=1^{\A_t}$ and $f(a)\neq 0$.
\item if $\A$ has countably many atoms then $\A$ is completely representable if and only if $\A\in S_c\Nr_n\TeCA_{\omega}.$
\end{enumarab}
\end{theorem}
\begin{proof}
We prove only the last item. The idea is simple.
Lifting complete representations of the component to one for their product and vice versa by noting that
each  component is a complete subalgebra of the product.
Assume that  $\Rd_{ca}\A=\prod_{t\in T}\A_t$ and every $\A_t$ is
completely representable hence is in $S_c\Nr_nCA_{\omega}$
but the latter class as easily checked is closed under products.

The converse follows from that every
$\A_t$ is a complete subalgebra of $\A$, the latter  is in $\bold S_c\Nr_n\TeCA_{\omega}$, hence $\A_t\in S_c\Nr_n\TeCA_{\omega}$.
But this means that every $\A_t$ is completely representable with the homomorphism $f_t:\A\to \wp({}^nX_t)$
say then so is $\A$, via $f:\A\to \prod_{t\in T}\wp(^nX_t)=\wp(\bigcup(^nX_t))$ where $\bigcup$ here  denotes disjoint
union.
\end{proof}

\section*{Part2}
In the next theorem we show that atomicty does not imply complete represenatibility.
In fact we show that for finite $n>2$ that  the class of completely representable $\TeCA_n$s
is not elementary, {\it a fortiori} it is not atomic because atomicity is a first order definable notion.
We use a rainbow contruction and in fact we prove a result stronger
than non elementarity
of the class of completey representable algebras; we prove it
for both $\TCA_n$ and $\TeCA_n$

\begin{theorem}\label{rainbow}

\begin{enumarab}
\item Let $3\leq n<\omega$. Then there exists an atomic $\C\in \sf TCA_n$ with countably many atoms
such that $\C\notin S_c\Nr_n\TCA_{n+3}$, and there exists
a countable $\B\in S_c\Nr_n\sf TCA_{\omega}$ such that $\C\equiv \B$ (hence $\B$ is also atomic).
Hence for any class $\sf L$, such that $S_c\Nr_n\TCA_{\omega}\subseteq \sf L\subseteq S_c\Nr_n\TCA_{n+3}$, $\sf L$ is not elementary, and
the class of completely representable $\K$ algebras of dimension $n$
is not elementary.

\item The same holds
for $\TeCA_n$. For every time flow $(T, <)$ there is an atomic $\sf RTeCA_n$
based on $T$ that is not completely representable, in fact not in $S_c\Nr_n \TeCA_{n+3}$ but is elementary equivalent
to a completey representable tense cylindric algebra of dimension $n$.
\end{enumarab}

\end{theorem}
\begin{proof}
\begin{enumarab}

\item Let $\N^{-1}$ denote $\N$ with reverse order, let $f:\N\to \N^{-1}$ be the identity map, and denote $f(a)$ by
$-a$, so that for $n,m \in \N$, we have $n<m$ iff $-m<-n$.
We assume that $0$  belongs to $\N$ and we denote the domain of $\N^{-1}$ (which is $\N$)  by $\N^{-1}$.
We alter slightly the standard rainbow construction. The colours we use are the same colours used in rainbow constructions:
\begin{itemize}
\item greens: $\g_i$ ($1\leq i\leq n-2)$, $\g_0^i$, $i\in \N^{-1}$,

\item whites : $\w_i: i\leq n-2,$

\item reds:  $\r_{ij}$ $(i,j\in \N)$,

\item shades of yellow : $\y_S: S\subseteq_{\omega} \N^{-1}$ or $S=\N^{-1}$.

\end{itemize}

We define a {\it subclass} of coloured  graphs $M$ such that
\begin{enumarab}

\item $M$ is a complete graph.

\item $M$ contains no triangles (called forbidden triples)
defined exactly as in \cite{HH} together with the additional forbidden triple
$(\g^i_0, \g^j_0, \r_{kl})$, in more detail

\vspace{-.2in}
\begin{eqnarray}
&&\nonumber\\
(\g, \g^{'}, \g^{*}), (\g_i, \g_{i}, \w_i)
&&\mbox{any }1\leq i\leq  n-2  \\
(\g^j_0, \g^k_0, \w_0)&&\mbox{ any } j, k\in \N\\
\label{forb:pim}(\g^i_0, \g^j_0, \r_{kl})&&\mbox{unless } \set{(i, k), (j, l)}\mbox{ is an order-}\\
&&\mbox{ preserving partial function }\N^{-1}\to\N\nonumber\\
\label{forb:match}(\r_{ij}, \r_{j'k'}, \r_{i^*k^*})&&\mbox{unless }i=i^*,\; j=j'\mbox{ and }k'=k^*.
\end{eqnarray}
and no other triple of atoms is forbidden.


and no other triple of atoms is forbidden.

\item The last two items concerning shades of yellow are as before.

\end{enumarab}

But the forbidden triple $(\g^i_0, \g^j_0, \r_{kl})$
is not present in standard rainbow constructions, adopted example in \cite{HH} and in a more general form in
\cite{HHbook2}. Therefore, we cannot use the usual rainbow argument adopted in \cite{HH}; we have to be selective for the choice of the indices
of reds if we are labelling the appexes of two cones having green tints; {\it not any red that works for usual rainbows} will do.

One then can define (what we continue to call) a rainbow atom structure
of dimension $n$ from the class $\G$.
Let $$\At=\{a:n \to M, M\in \G: \text { $a$ is surjective }\}.$$
We write $M_a$ for the element of $\At$ for which
$a:n\to M$ is a surjection.
Let $a, b\in \At$ define the
following equivalence relation: $a \sim b$ if and only if
\begin{itemize}
\item $a(i)=a(j)\Longleftrightarrow b(i)=b(j),$

\item $M_a(a(i), a(j))=M_b(b(i), b(j))$ whenever defined,

\item $M_a(a(k_0),\dots, a(k_{n-2}))=M_b(b(k_0),\ldots, b(k_{n-2}))$ whenever
defined.
\end{itemize}
Let $\At$ be the set of equivalences classes. Then define
$$[a]\in E_{ij} \text { iff } a(i)=a(j).$$
$$[a]T_i[b] \text { iff }a\upharpoonright n\smallsetminus \{i\}=b\upharpoonright n\smallsetminus \{i\}.$$
And topologizing the atom structure we set
$$[a]I_i [b]\text { iff }a=b,$$
that adds nothing to the cylindric structure inducing the interior
topology.
This, as easily checked, defines a $\sf TCA_n$
atom structure. Let $\C$ be the complex algebra.
Let $k>0$ be given. We show that \pe\ has a \ws\ in the usual graph game in $k$ rounds
(now there is no restriction here on the size of the graphs) on $\At\C$.
We recall the `usual atomic' $k$ rounded game $G_k$ played on coloured graphs.
\pa\ picks a graph $M_0\in \G$ with $M_0\subseteq n$ and
$\exists$ makes no response
to this move. In a subsequent round, let the last graph built be $M_i$.
\pa\ picks
\begin{itemize}
\item a graph $\Phi\in \G$ with $|\Phi|=n,$
\item a single node $m\in \Phi,$
\item a coloured graph embedding $\theta:\Phi\smallsetminus \{m\}\to M_i.$
Let $F=\phi\smallsetminus \{m\}$. Then $F$ is called a face.
The \pe\  has to define a
graph $M_{i+1}\in C$ and embeddings $\lambda:M_i\to M_{i+1}$
$\mu:\phi \to M_{i+1}$, such that $\lambda\circ \theta=\mu\upharpoonright F.$
\end{itemize}
Consider the following restricted version of $G_{\omega}$; only $m>n$ nodes are available.
We denote this $\omega$ rounded game using $m$ nodes  by $G^m$.
In more detail, in a play of $G^m$ \pa\ picks a graph $M_0\in \G$ with $M_0\subseteq m$ and
$\exists$ makes no response
to this move. In a subsequent round, let the last graph built be $M_i$.
\pa\ picks
\begin{itemize}
\item a graph $\Phi\in \sf CRG$ with $|\Phi|=n,$
\item a single node $k\in \Phi,$
\item a coloured graph embedding $\theta:\Phi\smallsetminus \{k\}\to M_i.$
\pe\ must respond like the in the usual atomic game
by amalgamating
$M_i$ and $\Phi$ with the embedding $\theta$.
\end{itemize}

$F^m$ is like $G^m$, but \pa\ is allowed to resuse nodes.

Inspite of the restriction of adding a forbidden triple relating two greens and a red
(that makes it harder for \pe\ to win),  we will show that \pe\ will always succeed to choose a
suitable red in the finite rounded atomic games $G^m$. On the other hand, this bonus for \pa\ will enable him to win the game $F^{n+3}$ by forcing
\pe\ to play a decreasing sequence in $\N$. Using and re-using $n+3$ nodes will suffice for this
purpose.
We define \pe\ s strategy for choosing labels for edges and $n-1$ tuples in response to \pa\ s moves.
Assume that we are at round $r+1$. Our  arguments are similar to the arguments in \cite[Lemmas, 41-43]{r}.

Let $M_0, M_1,\ldots M_r$, $r<k$ be the coloured graphs at the start of a play of $G^k(\alpha)$ just before round $r+1$.
Assume inductively that \pe\ computes a partial function $\rho_s:\N^{-1}\to \N$, for $s\leq r$, that will help her choose
the suffices of the chosen red in the critical case. In our previous rainbow construction we had the additional shade of red
$\rho$ that did the job. Now we do not have it, so we proceed differently.  Inductively
for $s\leq r$ we assume:

\begin{enumarab}
\item  If $M_s(x,y)$ is green then $(x,y)$ belongs  \pa\ in $M_s$ (meaning he coloured it),

\item $\rho_0\subseteq \ldots \rho_r\subseteq\ldots,$
\item $\dom(\rho_s)=\{i\in \N^{-1}: \exists t\leq s, x, x_0, x_1,\ldots, x_{n-2}
\in \nodes(M_t)\\
\text { where the $x_i$'s form the base of a cone, $x$ is its appex and $i$ its tint }\}.$

The domain consists of the tints of cones created at an earlier stage,

\item $\rho_s$ is order preserving: if $i<j$ then $\rho_s(i)<\rho_s(j)$. The range
of $\rho_s$ is widely spaced: if $i<j\in \dom\rho_s$ then $\rho_s(i)-\rho_s(j)\geq  3^{m-r}$, where $m-r$
is the number of rounds remaining in the game,

\item For $u,v,x_0\in \nodes(M_s)$, if $M_s(u,v)=\r_{\mu,\delta}$, $M_s(x_0,u)=\g_0^i$, $M_s(x_0,v)=\g_0^j$,
where $i,j$ are tints of two cones, with base $F$ such that $x_0$ is the first element in $F$ under the induced linear order,
then $\rho_s(i)=\mu$ and $\rho_s(j)=\delta,$

\item $M_s$ is a a  coloured graph,

\item If the base of a cone $\Delta\subseteq M_s$  with tint $i$ is coloured $y_S$, then $i\in S$.

\end{enumarab}

To start with if \pa\ plays $a$ in the initial round then $\nodes(M_0)=\{0,1,\ldots, n-1\}$, the
hyperedge labelling is defined by $M_0(0,1,\ldots, n)=a$.

In response to a cylindrifier move for some $s\leq r$, involving a $p$ cone, $p\in \N^{-1}$,
\pe\ must extend $\rho_r$ to $\rho_{r+1}$ so that $p\in \dom(\rho_{r+1})$
and the gap between elements of its range is at least $3^{m-r-1}$. Properties (3) and (4) are easily
maintained in round $r+1$. Inductively, $\rho_r$ is order preserving and the gap between its elements is
at least $3^{m-r}$, so this can be maintained in a further round.
If \pa\ chooses a green colour, or green colour whose suffix
already belong to $\rho_r$, there would be fewer
elements to add to the domain of $\rho_{r+1}$, which makes it easy for \pe\ to define $\rho_{r+1}$.

Now assume that at round $r+1$, the current coloured graph is $M_r$ and that   \pa\ chose the graph $\Phi$, $|\Phi|=n$
with distinct nodes $F\cup \{\delta\}$, $\delta\notin M_r$, and  $F\subseteq M_r$ has size
$n-1$.  We can  view \pe\ s move as building a coloured graph $M^*$ extending $M_r$
whose nodes are those of $M_r$ together with the new node $\delta$ and whose edges are edges of $M_r$ together with edges
from $\delta$ to every node of $F$.

Now \pe\ must extend $M^*$ to a complete graph $M^+$ on the same nodes and
complete the colouring giving  a graph $M_{r+1}=M^+$ in $\G$ (the latter is the class of coloured graphs).
In particular, she has to define $M^+(\beta, \delta)$ for all nodes
$\beta\in M_r\sim F$, such that all of the above properties are maintained.

\begin{enumarab}

\item  If $\beta$ and $\delta$ are both apexes of two cones on $F$.

Assume that the tint of the cone determined by $\beta$ is $a\in \N^{-1}$, and the two cones
induce the same linear ordering on $F$. Recall that we have $\beta\notin F$, but it is in $M_r$, while $\delta$ is not in $M_r$,
and that $|F|=n-1$.
By the rules of the game  \pe\ has no choice but to pick a red colour. \pe\ uses her auxiliary
function $\rho_{r+1}$ to determine the suffices, she lets $\mu=\rho_{r+1}(p)$, $b=\rho_{r+1}(q)$
where $p$ and $q$ are the tints of the two cones based on $F$,
whose apexes are $\beta$ and $\delta$. Notice that $\mu, b\in \N$; then she sets $N_s(\beta, \delta)=\r_{\mu,b}$
maintaining property (5), and so $\delta\in \dom(\rho_{r+1})$
maintaining property (4). We check consistency to maintain property (6).

Consider a triangle of nodes $(\beta, y, \delta)$ in the graph $M_{r+1}=M^+$.
The only possible potential problem is that the edges $M^+(y,\beta)$ and $M^+(y,\delta)$ are coloured green with
distinct superscripts $p, q$ but this does not contradict
forbidden triangles of the form involving $(\g_0^p, \g_0^q, \r_{kl})$, because $\rho_{r+1}$ is constructed to be
order preserving.  Now assume that
$M_r(\beta, y)$ and $M_{r+1}(y, \delta)$ are both red (some $y\in \nodes(M_r)$).
Then \pe\ chose the red label $N_{r+1}(y,\delta)$, for $\delta$ is a new node.
We can assume that  $y$ is the apex of a $t$ cone with base $F$ in $M_r$. If not then $N_{r+1}(y, \delta)$ would be coloured
$\w$ by \pe\   and there will be no problem. All properties will be maintained.
Now $y, \beta\in M$, so by by property (5) we have $M_{r+1}(\beta,y)=\r_{\rho+1(p), \rho+1(t)}.$
But $\delta\notin M$, so by her strategy,
we have  $M_{r+1}(y,\delta)=\r_{\rho+1(t), \rho+1(q)}.$ But $M_{r+1}(\beta, \delta)=\r_{\rho+1(p), \rho+1(q)}$,
and we are done.  This is consistent triple, and so have shown that
forbidden triples of reds are avoided.

\item If there is no $f\in F$ such that $M^*(\beta, f), M^*(\delta,f)$ are coloured $g_0^t$, $g_0^u$ for some $t, u$ respectively,
then \pe\ defines $M^+(\beta, \delta)$ to be $\w_0$.

\item If this is not the case, and  for some $0<i<n-1$ there is no $f\in F$ such
that $M^*(\beta, f), M^* (\delta, f)$ are both coloured $\g_i$,
she chooses $\w_i$ for  $M^+{(\beta,\delta)}$.

It is clear that these choices in the last two items  avoid all forbidden triangles (involving greens and whites).

\end{enumarab}

She has not chosen green maintaining property (1).  Now we turn to colouring of $n-1$ tuples,
to make sure that $M^+$ is a coloured graph maintaining property (7).

Let $\Phi$ be the graph chosen by \pa\, it has set of node $F\cup \{\delta\}$.
For each tuple $\bar{a}=a_0,\ldots a_{n-2}\in {M^+}^{n-1}$, $\bar{a}\notin M^{n-1}\cup \Phi^{n-1}$,  with no edge
$(a_i, a_j)$ coloured green (we already have all edges coloured), then  \pe\ colours $\bar{a}$ by $\y_S$, where
$$S=\{i\in A: \text { there is an $i$ cone in $M^*$ with base $\bar{a}$}\}.$$
We need to check that such labeling works, namely that last property is maintained.

Recall that $M$ is the current coloured graph, $M^*=M\cup \{\delta\}$ is built by \pa\ s move
and $M^+$ is the complete labelled graph by \pe\, whose nodes are labelled by \pe\ in response to \pa\ s moves.
We need to show that $M^+$ is labelled according to
the rules of the game, namely, that it is in $\G$.
It can be checked $(n-1)$ tuples are labelled correctly, by yellow colours using
the same argument in \cite[p.16]{Hodkinson} \cite[p.844]{HH}
and \cite{HHbook2}.

We show that \pa\ has a \ws\ in $F^{n+3}$, the argument used is the $\CA$ analogue of \cite[Theorem 33, Lemma 41]{r}.
The difference is that in the relation algebra case, the game is played on atomic networks, but now it is translated to playing on coloured graphs,
\cite[lemma 30]{HH}.

In the initial round \pa\ plays a graph $\Gamma$ with nodes $0,1,\ldots n-1$ such that $\Gamma(i,j)=\w_0$ for $i<j<n-1$
and $\Gamma(i, n-1)=\g_i$
$(i=1, \ldots, n-2)$, $\Gamma(0,n-1)=\g_0^0$ and $\Gamma(0,1\ldots, n-2)=\y_{B}$.

In the following move \pa\ chooses the face $(0,\ldots n-2)$ and demands a node $n$
with $\Gamma_2(i,n)=\g_i$ $(i=1,\ldots, n-2)$, and $\Gamma_2(0,n)=\g_0^{-1}.$
\pe\ must choose a label for the edge $(n+1,n)$ of $\Gamma_2$. It must be a red atom $r_{mn}$. Since $-1<0$ we have $m<n$.
In the next move \pa\ plays the face $(0, \ldots, n-2)$ and demands a node $n+1$, with $\Gamma_3(i,n)=\g_i$ $(i=1,\ldots, n-2)$,
such that  $\Gamma_3(0,n+2)=\g_0^{-2}$.
Then $\Gamma_3(n+1,n)$ and  $\Gamma_3(n+1,n-1)$ both being red, the indices must match.
$\Gamma_3(n+1,n)=r_{ln}$ and $\Gamma_3(n+1, n-1)=r_{lm}$ with $l<m$.
In the next round \pa\ plays $(0,1,\ldots, n-2)$ and reuses the node $2$ such that $\Gamma_4(0,2)=\g_0^{-3}$.
This time we have $\Gamma_4(n,n-1)=\r_{jl}$ for some $j<l\in \N$.
Continuing in this manner leads to a decreasing sequence in $\N$.

Now that \pa\ has a \ws\ in $F^{n+3},$ it follows by an easy adaptation of lemma \cite[Lemma 33]{r},
baring in mind that games on coloured graphs are equivalent to games on atomic networks,
that $\Rd_{sc}\C\notin S_c\Nr_n\Sc_{n+3}$, but it is elementary equivalent
to a countable completely representable algebra. Indeed, using ultrapowers and an elementary chain argument,
we obtain $\B$ such  $\C\equiv \B$ \cite[lemma 44]{r},
and \pe\ has a \ws\ in $G_{\omega}$ (the usual $\omega$ rounded atomic game)
on $\B$ so by \cite[Theorem 3.3.3]{HHbook2}, $\B$ is completely representable.

In more detail we have \pe\ has a \ws\ $\sigma_k$ in $G_k$.
We can assume that $\sigma_k$ is deterministic.
Let $\D$ be a non-principal ultrapower of $\C$.  One can show that
\pe\ has a \ws\ $\sigma$ in $G(\At\D)$ --- essentially she uses
$\sigma_k$ in the $k$'th component of the ultraproduct so that at each
round of $G(\At\D)$,  \pe\ is still winning in co-finitely many
components, this suffices to show she has still not lost.

We can also assume that $\C$ is countable. If not then replace it by its subalgebra generated by the countably many atoms
(the term algebra); \ws\ s that depended only on the atom structure persist
for both players.

Now one can use an
elementary chain argument to construct a chain of countable elementary
subalgebras $\C=\A_0\preceq\A_1\preceq\ldots\preceq\ldots \D$ inductively in this manner.
One defines  $\A_{i+1}$ be a countable elementary subalgebra of $\D$
containing $\A_i$ and all elements of $\D$ that $\sigma$ selects
in a play of $G(\At\D)$ in which \pa\ only chooses elements from
$\A_i$. Now let $\B=\bigcup_{i<\omega}\A_i$.  This is a
countable elementary subalgebra of $\D$, hence necessarily atomic,  and \pe\ has a \ws\ in
$G(\At\B)$, so by \cite[Theorem 3.3.3]{HHbook2}, noting that $\B$ is countable,
then $\B$ is completely representable; furthermore $\B\equiv \C$.

The same proof works for $\sf TeCA_{\alpha}$ for the last part
using a static expansion of $\C$ and a static expansion  of $\B$ expansion.

\item As for the last part, let $(T, <)$ be a given flow of time. Choose representable atomic
cylindric algebras $\C_t, \B_t$ for each $t\in T$  such that  $\B_t \equiv \C_t$, $\C_t\notin S\Nr_n\CA_{n+3}$ while $\B_t$ is completely representable.

We know by theorem \ref{rainbow} that such algebras exist. In fact, we can take $\B_t=\B$ nd $\C_t=\C$ for each
moment  $t\in T$, where $\B$ and $\C$ are algebras used in theorem \ref{rainbow}.
By the Ferefman Vaught theorem we have $\B^+=prod_{t\in T}\B_t\equiv \prod_{t\in T} \C_{t}=\C^+$. Now define
$<$ in $T$ using $G$ and $H$
both defined as the identity map, getting the required flow of time, the
tense system based on it, and finally the tense algebra constructed
from such a tense sytem.
The elementary equivalence remain to hold, we need to check that $\prod B_t$ is completely represenable,
this is not hard.

Now assume for contadiction that $\C^+=\prod_{t\in T}\C_t\in S_c\Nr_n\sf TeCA_{n+3}$. Recall that $\C_t=\C$ for any $t\in T$
But it is clear tht the projection  map from $\C\to \C^+$ is a complete one, hence $\C^{te}\in S_c\Nr_n\sf TeCA_{n+3}$ , but
that $\C=\Rd_{ca}\C^{te}\in S_c\Nr_n\CA_{n+3}$ which is a contradiction.
\end{enumarab}
\end{proof}
Consider the following two statements for finite $n>2$:
\begin{enumarab}

\item Any class $\K$ such that $S_c\Nr_n\TCA_{\omega}\subseteq \sf K\subseteq S_c\TCA_{n+3}$ is not elementary.
\item Any $\K$ such that $\Nr_n\CA_{\omega}\subseteq \K\subseteq S_c\Nr_n\CA_{n+3}$ is not elementary.
 \end{enumarab}

The first statement is proved in our previos theorem. The second statement is stronger since  
$\Nr_n\CA_{\omega}\subseteq S_c\Nr_n\CA_{\omega}$ and the inclusion is strict. 
This inclusion can be distilled from a construction in \cite{SL, conference} where an infinite set algebra is used
and which we now recall. Later, witness example \ref{finitepa} 
we will give an example of a finite algebra that witnesses the strictness of this inclusion for finite $n>2$. 
Our next example though works for 
any $\alpha>1$, but algebras used are infinite. Of courese this is to be expected for infinite dimensions but 
such algebras are also infinite 
for any finite dimension $>1$.

\begin{example}\label{SL}

Here we give an example that works for any  signature is between $\Sc$ and $\sf QEA$, where $\Sc$ denotes Pinter's substitution algebras and 
$\sf QEA$ denotes quasipolydic algebras.
(Here we are using the notation $\sf QEA$ instead of $\PEA$ because we are
allowing infinite dimensions).

$FT_{\alpha}$ denotes the set of all finite transformations on $\alpha$.
Let $\alpha$ be an ordinal $>1$; could be infinite. Let $\F$ is field of characteristic $0$.
$$V=\{s\in {}^{\alpha}\F: |\{i\in \alpha: s_i\neq 0\}|<\omega\},$$
$${\C}=(\wp(V),
\cup,\cap,\sim, \emptyset , V, {\sf c}_{i},{\sf d}_{i,j}, {\sf s}_{\tau})_{i,j\in \alpha, \tau\in FT_{\alpha}}.$$
Then clearly $\wp(V)\in \Nr_{\alpha}\sf QPEA_{\alpha+\omega}$.
Indeed let $W={}^{\alpha+\omega}\F^{(0)}$. Then
$\psi: \wp(V)\to \Nr_{\alpha}\wp(W)$ defined via
$$X\mapsto \{s\in W: s\upharpoonright \alpha\in X\}$$
is an isomorphism from $\wp(V)$ to $\Nr_{\alpha}\wp(W)$.
We shall construct an algebra $\A$, $\A\notin \Nr_{\alpha}{\sf QPEA}_{\alpha+1}$.
Let $y$ denote the following $\alpha$-ary relation:
$$y=\{s\in V: s_0+1=\sum_{i>0} s_i\}.$$
Let $y_s$ be the singleton containing $s$, i.e. $y_s=\{s\}.$
Define
${\A}\in {\sf QPEA}_{\alpha}$
as follows:
$${\A}=\Sg^{\C}\{y,y_s:s\in y\}.$$

Now clearly $\A$ and $\wp(V)$ share the same atom structure, namely, the singletons.
Then we claim that
$\A\notin \Nr_{\alpha}{\sf QPEA}_{\beta}$ for any $\beta>\alpha$.
The first order sentence that codes the idea of the proof says
that $\A$ is neither an elementary nor complete subalgebra of $\wp(V)$. We use $\land$ and $\to$ in the meta language
with their usual meaning.
Let $\At(x)$ be the first order formula asserting that $x$ is an atom.
Let $$\tau(x,y) ={\sf c}_1({\sf c}_0x\cdot {\sf s}_1^0{\sf c}_1y)\cdot {\sf c}_1x\cdot {\sf c}_0y.$$
Let $${\sf Rc}(x):={\sf c}_0x\cap {\sf c}_1x=x,$$
$$\phi:=\forall x(x\neq 0\to \exists y(\At(y)\land y\leq x))\land
\forall x(\At(x) \to {\sf Rc}(x)),$$
$$\alpha(x,y):=\At(x)\land x\leq y,$$
and  $\psi (y_0,y_1)$ be the following first order formula
$$\forall z(\forall x(\alpha(x,y_0)\to x\leq z)\to y_0\leq z)\land
\forall x(\At(x)\to {\sf At}({\sf c}_0x\cap y_0)\land {\sf At}({\sf c}_1x\cap y_0))$$
$$\to [\forall x_1\forall x_2(\alpha(x_1,y_0)\land \alpha(x_2,y_0)\to \tau(x_1,x_2)\leq y_1)$$
$$\land \forall z(\forall x_1 \forall x_2(\alpha(x_1,y_0)\land \alpha(x_2,y_0)\to
\tau(x_1,x_2)\leq z)\to y_1\leq z)].$$
Then
$$\Nr_{\alpha}{\sf QPEA}_{\beta}\models \phi\to \forall y_0 \exists y_1 \psi(y_0,y_1).$$
But this formula does not hold in $\A$.
We have $\A\models \phi\text {  and not }
\A\models \forall y_0\exists y_1\psi (y_0,y_1).$
In words: we have a set $X=\{y_s: s\in V\}$ of atoms such that $\sum^{\A}X=y,$ and $\A$
models $\phi$ in the sense that below any non zero element there is a
{\it rectangular} atom, namely a singleton.

Let $Y=\{\tau(y_r,y_s), r,s\in V\}$, then
$Y\subseteq \A$, but it has {\it no supremum} in $\A$, but {\it it does have one} in any full neat reduct $\B$ containing $\A$,
and this is $\tau_{\alpha}^{\B}(y,y)$, where
$\tau_{\alpha}(x,y) = {\sf c}_{\alpha}({\sf s}_{\alpha}^1{\sf c}_{\alpha}x\cdot {\sf s}_{\alpha}^0{\sf c}_{\alpha}y).$

In $\wp(V)$ this last is $w=\{s\in {}^{\alpha}\F^{(\bold 0)}: s_0+2=s_1+2\sum_{i>1}s_i\},$
and $w\notin \A$. The proof of this can be easily distilled from \cite[main theorem]{SL}.
For $y_0=y$, there is no $y_1\in \A$ satisfying $\psi(y_0,y_1)$.
Actually the above proof proves more. It proves that there is a
$\C\in \Nr_{\alpha}{\sf QEA}_{\beta}$ for every $\beta>\alpha$ (equivalently $\C\in \Nr_n\QEA_{\omega}$), and $\A\subseteq \C$, such that
$\Rd_{\Sc}\A\notin \Nr_{\alpha}\Sc_{\alpha+1}$.
See \cite[theorems 5.1.4-5.1.5]{Sayedneat} for an entirely different example.
Now topologizing $\Rd_{ca}\C$ we get the required.

Indeed we have $\Rd_{ca}\A\in \Nr_n\CA_{\omega}$, hence there exists $\B'\in \CA_{\omega}$ such that $\A=\Nr_n\B'$. 
Let $\B$ be the subalgebra of $\B'$ generated by $\A$. 
Then $\B$ is localy finite, hence representable  
with base $U$ say; give $U$ the discrete topology, then $\Rd_{ca}\A^{\sf top}=\Nr_n\B^{\sf top}.$
Now it is obvious that $\A^{\sf top}\notin \Nr_n\TCA_{n+1}$ for if it were then its 
$\Sc$ reduct would be in $\Nr_n\Sc_{n+1}$ 
and we know is  not the case.

Now give a more sophisticated example. 
We prove that the strictness of the inclusion can be witnessed by a finite algebra: 

\subsection{Neat games}

Let $\delta$ be a map. Then $\delta[i\to d]$ is defined as follows. $\delta[i\to d](x)=\delta(x)$
if $x\neq i$ and $\delta[i\to d](i)=d$. We write $\delta_i^j$ for $\delta[i\to \delta_j]$.

We recall the definition of {\it polyadic equality} networks.
\begin{definition}\label{network}
Let $2< n<\omega.$ Let $\C$ be an atomic ${\sf PEA}_{n}$.
An \emph{atomic  network} over $\C$ is a map
$$N: {}^{n}\Delta\to \At\C,$$
where $\Delta$ is a non-empty set called a set of nodes,
such that the following hold for each $i,j<n$, $\delta\in {}^{n}\Delta$
and $d\in \Delta$:
\begin{itemize}
\item $N(\delta^i_j)\leq {\sf d}_{ij}$
\item $N(\delta[i\to d])\leq {\sf c}_iN(\delta)$
\item $N(\bar{x}\circ [i,j])= {\sf s}_{[i,j]}N(\bar{x})$ for all $i,j<n$.
\end{itemize}
\end{definition}

\begin{theorem}\label{neat1}
There is an $\omega$ rounded game $J$ such that if \pe\ can win $J$ on a $\PEA_n$ atom structure $\alpha$
with $2<n<\omega$, then there exists a locally finite $\D\in \sf QEA_{\omega}$ such that
$\alpha\cong \At\Nr_n\D$.
Furthermore, $\D$ can be chosen to be complete, in which case we have $\mathfrak{Cm}\alpha=\Nr_n\D$.
\end{theorem}
\begin{proof}
We define a stronger game $J$.
This game is played on {\it $\lambda$ neat hypernetworks}.
We need some preparing to do.
For an atomic network and for  $x,y\in \nodes(N)$, we set  $x\sim y$ if
there exists $\bar{z}$ such that $N(x,y,\bar{z})\leq {\sf d}_{01}$.
Define the  equivalence relation $\sim$ over the set of all finite sequences over $\nodes(N)$ is defined by
$\bar x\sim\bar y$ iff $|\bar x|=|\bar y|$ and $x_i\sim y_i$ for all
$i<|\bar x|$.(It can be checked that this indeed an equivalence relation.)

A \emph{ hypernetwork} $N=(N^a, N^h)$ over an atomic polyadic equality algebra $\C$
consists of a network $N^a$
together with a labelling function for hyperlabels $N^h:\;\;^{<
\omega}\!\nodes(N)\to\Lambda$ (some arbitrary set of hyperlabels $\Lambda$)
such that for $\bar x, \bar y\in\; ^{< \omega}\!\nodes(N)$

if
$$\bar x\sim\bar y \Rightarrow N^h(\bar x)=N^h(\bar y).$$

If $|\bar x|=k\in \N$ and $N^h(\bar x)=\lambda$ then we say that $\lambda$ is
a $k$-ary hyperlabel. $(\bar x)$ is referred to a a $k$-ary hyperedge, or simply a hyperedge.
(Note that we have atomic hyperedges and hyperedges)

When there is no risk of ambiguity we may drop the superscripts $a, h$.
There are {\it short} hyperedges and {\it long} hyperedges (to be defined in a while). The short hyperedges are constantly labelled.
The idea (that will be revealed during the proof), is that the atoms in the neat reduct are no smaller than the atoms
in the dilation. (When $\A=\Nr_n\B,$ it is common to call $\B$ a dilation of $\A$.)

There is a one to one correspondence between networks and coloured graphs \cite[second half of p. 76]{HHbook2}.
If $\Gamma$ is a coloured graph, then by $N_{\Gamma}$
we mean the corresponding network.
\begin{itemize}
\item A hyperedge $\bar{x}\in {}^{<\omega}\nodes (\Gamma)$ of length $m$ is {\it short}, if there are $y_0,\ldots, y_{n-1}\in \nodes(N)$, such that
$$N_{\Gamma}(x_i, y_0, \bar{z})\leq {\sf d}_{01}\text { or } N(_{\Gamma}(x_i, y_1, \bar{z})\ldots$$
$$\text { or }N(x_i, y_{n-1},\bar{z})\leq {\sf d}_{01}$$
for all $i<|x|$, for some (equivalently for all)
$\bar{z}.$

Otherwise, it is called {\it long.}
\item A hypergraph $(\Gamma, l)$
is called {\it $\lambda$ neat} if $N_{\Gamma}(\bar{x})=\lambda$ for all short hyper edges.
\end{itemize}
This game is similar to the games devised by Robin Hirsch in \cite[definition 28]{r}, played on relation algebras.
However, lifting it to cylindric algebras
is not straightforward at all, for in this new context the moves involve hyperedges of length $n$ (the dimension),
rather than edges. We have to be careful here for we have two types of hyperedges.
Hyperedges having length $n$
that are labelled by atoms of the algebra, and there are other hyperedges labelled from another set (of labels).
The latter hyperedges can get arbitrarily long.

In the $\omega$ rounded game $J$,  \pa\ has three moves.
$J_k$ will denote the game $J$ truncated to $k$ rounds.

The first is the normal cylindrifier move. There is no polyadic move for the polyadic information is coded in the networks.
The next two are amalgamation moves.
But now the games are not played on coloured graphs,
they are played on coloured {\it hypergraphs}, consisting of two parts,
the graph part
that can be viewed as an $L_{\omega_1, \omega}$ model for the rainbow signature, and the part dealing with hyperedges with a
labelling function.
The game is played on $\lambda$ neat hypernetworks,  translated to $\lambda$ neat hypergraphs,
where $\lambda$ is a label for the short hyperedges.

For networks $M, N$ and any set $S$, recall that we write $M\equiv^SN$
if $N\restr S=M\restr S$, and we write $M\equiv_SN$
if the symmetric difference
$$\Delta(\nodes(M), \nodes(N))\subseteq S$$ and
$M\equiv^{(\nodes(M)\cup\nodes(N))\setminus S}N.$ We write $M\equiv_kN$ for
$M\equiv_{\set k}N$.

Let $N$ be a network and let $\theta$ be any function.  The network
$N\theta$ is a complete labelled graph with nodes
$\theta^{-1}(\nodes(N))=\set{x\in\dom(\theta):\theta(x)\in\nodes(N)}$,
and labelling defined by
$$(N\theta)(i_0,\ldots, i_{n-1}) = N(\theta(i_0), \theta(i_1), \ldots,  \theta(i_{n-1})),$$
for $i_0, \ldots, i_{n-1}\in\theta^{-1}(\nodes(N))$.

Concerning \pa\ s moves:

\pa\ can play a cylindrifier move, like before but now played on $\lambda_0$ neat hypernetworks.

\pa\ can play a \emph{transformation move} by picking a
previously played hypernetwork $N$ and a partial, finite surjection
$\theta:\omega\to\nodes(N)$, this move is denoted $(N, \theta)$.  \pe\
must respond with $N\theta$.

Finally, \pa\ can play an
\emph{amalgamation move} by picking previously played hypernetworks
$M, N$ such that $M\equiv^{\nodes(M)\cap\nodes(N)}N$ and
$\nodes(M)\cap\nodes(N)\neq \emptyset$.
This move is denoted $(M,
N)$. Here there is {\it no restriction} on the number of overlapping nodes, so in principal
the game is harder for \pe\ to win, if there was (Below we will encounter such restricted
amalgamation moves).

To make a legal response, \pe\ must play a $\lambda_0$-neat
hypernetwork $L$ extending $M$ and $N$, where
$\nodes(L)=\nodes(M)\cup\nodes(N)$.

Now assume that \pe\ has a \ws\ in $J$ on the $\PEA_n$ atom structure $\alpha$.

Fix some $a\in\alpha$. Using \pe\ s \ws\ in the game of neat hypernetworks, one defines a
nested sequence $N_0\subseteq N_1\ldots$ of neat hypernetworks
where $N_0$ is \pe's response to the initial \pa-move $a$, such that
\begin{enumerate}
\item If $N_r$ is in the sequence and
and $b\leq {\sf c}_lN_r(f_0, \ldots, x, \ldots  f_{n-2})$,
then there is $s\geq r$ and $d\in\nodes(N_s)$ such
that $N_s(f_0, f_{l-1}, d, f_{l+1},\ldots f_{n-2})=b$.
\item If $N_r$ is in the sequence and $\theta$ is any partial
isomorphism of $N_r$ then there is $s\geq r$ and a
partial isomorphism $\theta^+$ of $N_s$ extending $\theta$ such that
$\rng(\theta^+)\supseteq\nodes(N_r)$.
\end{enumerate}
Now let $N_a$ be the limit of this sequence, that is $N_a=\bigcup N_i$, the labelling of $n-1$ tuples of nodes
by atoms, and the hyperedges by hyperlabels done in the obvious way.
This limit is well-defined since the hypernetworks are nested.
We shall show that $N_a$ is the base of a weak set algebra having unit  $V={}^{\omega}N_a^{(p)}$,
for some fixed sequence $p\in {}^{\omega}N_a$.

Let $\theta$ be any finite partial isomorphism of $N_a$ and let $X$ be
any finite subset of $\nodes(N_a)$.  Since $\theta, X$ are finite, there is
$i<\omega$ such that $\nodes(N_i)\supseteq X\cup\dom(\theta)$. There
is a bijection $\theta^+\supseteq\theta$ onto $\nodes(N_i)$ and $\geq
i$ such that $N_j\supseteq N_i, N_i\theta^+$.  Then $\theta^+$ is a
partial isomorphism of $N_j$ and $\rng(\theta^+)=\nodes(N_i)\supseteq
X$.  Hence, if $\theta$ is any finite partial isomorphism of $N_a$ and
$X$ is any finite subset of $\nodes(N_a)$ then
\begin{equation}\label{eq:theta}
\exists \mbox{ a partial isomorphism $\theta^+\supseteq \theta$ of $N_a$
 where $\rng(\theta^+)\supseteq X$}
\end{equation}
and by considering its inverse we can extend a partial isomorphism so
as to include an arbitrary finite subset of $\nodes(N_a)$ within its
domain.
Let $L$ be the signature with one $n$ -ary relation symbol ($b$) for
each $b\in\alpha$, and one $k$-ary predicate symbol ($\lambda$) for
each $k$-ary hyperlabel $\lambda$. We work in usual first order logic.

Here we have a sequence of variables of order type $\omega$, and two 'sorts' of formulas,
the $n$ relation symbols using $n$ variables; roughly
these that are built up out of the first $n$ variables will determine the atoms of
neat reduct, the $k$ ary predicate symbols
will determine the atoms of algebras of higher dimensions as $k$ gets larger;
the atoms in the neat reduct will be no smaller than the atoms in the dilations.

This process will be interpreted in an infinite weak set algebra with base $N_a$, whose elements are
those  assignments satisfying such formulas.

For fixed $f_a\in\;^\omega\!\nodes(N_a)$, let
$U_a=\set{f\in\;^\omega\!\nodes(N_a):\set{i<\omega:g(i)\neq
f_a(i)}\mbox{ is finite}}$.
Notice that $U_a$ is weak unit (a set of sequences agreeing cofinitely with a fixed one.)

We can make $U_a$ into the universe an $L$ relativized structure ${\cal N}_a$;
here relativized means that we are only taking those assignments agreeing cofinitely with $f_a$,
we are not taking the standard square model.
However, satisfiability  for $L$ formulas at assignments $f\in U_a$ is defined the usual Tarskian way, except
that we use the modal notation, with restricted assignments on the left:

For $b\in\alpha,\; l_0, \ldots, l_{n-1}, i_0 \ldots, i_{k-1}<\omega$, \/ $k$-ary hyperlabels $\lambda$,
and all $L$-formulas $\phi, \psi$, let
\begin{eqnarray*}
{\cal N}_a, f\models b(x_{l_0}\ldots,  x_{n-1})&\iff&N_a(f(l_0),\ldots,  f(l_{n-1}))=b\\
{\cal N}_a, f\models\lambda(x_{i_0}, \ldots,x_{i_{k-1}})&\iff&  N_a(f(i_0), \ldots,f(i_{k-1}))=\lambda\\
{\cal N}_a, f\models\neg\phi&\iff&{\cal N}_a, f\not\models\phi,\\
{\cal N}_a, f\models (\phi\vee\psi)&\iff&{\cal N}_a,  f\models\phi\mbox{ or }{\cal N}_a, f\models\psi\\
{\cal N}_a, f\models\exists x_i\phi&\iff& {\cal N}_a, f[i/m]\models\phi, \mbox{ some }m\in\nodes(N_a).
\end{eqnarray*}

For any $L$-formula $\phi$, write $\phi^{{\cal N}_a}$ for
$\set{f\in\;^\omega\!\nodes(N_a): {\cal N}_a, f\models\phi}$.  Let
$D_a= \set{\phi^{{\cal N}_a}:\phi\mbox{ is an $L$-formula}}$ and
\[\D_a=(D_a,  \cup, \sim, {\sf d}_{ij}, {\sf c}_i, {\sf s}_{[i,j]})_{ i, j<\omega}\]
where ${\sf d}_{ij}=(x_i= x_j)^{{\cal N}_a},\; {\sf c}_i(\phi^{{\cal N}_a})=(\exists
x_i\phi)^{{\cal N}_a}$ and ${\sf s}_{[i,j]}$ swaps the variables $x_i$ and $x_j$,
Observe that $\top^{{\cal N}_a}=U_a,\; (\phi\vee\psi)^{{\cal N}_a}=\phi^{\c
N_a}\cup\psi^{{\cal N}_a}$, etc
then  $\D_a\in\sf RQEA_\omega$. (Weak set algebras are representable).
In fact $\D_a\in {\sf Lf}_{\omega}\cap {\sf Ws}_\omega$, for each atom $a\in \alpha$.

Let $\D=\prod_{a\in \alpha} \D_a$. (This is not necessarily locally finite).
Then  $\D\in\RQEA_\omega$  will be shown to be is the desired generalized weak set algebra,
that is the desired dilation.
Note that unit of $\D$ is the disjoint union of the weak spaces.
So
$\Nr_{n}\D$ is atomic and $\alpha\cong\At\Nr_{n}\D$ --- the isomorphism
is $b \mapsto (b(x_0, x_1,\dots, x_{n-1})^{\D_a}:a\in A)$.
To make $\D$ locally finite, one can assume that $\D$ is generated by $\Nr_n\D$.

Now we can work in $L_{\infty,\omega}$ so that $\D$ is complete
by changing the defining clause for infinitary disjunctions to
$$N_a, f\models (\bigvee_{i\in I} \phi_i) \text { iff } (\exists i\in I)(N_a,  f\models\phi_i).$$

By working in $L_{\infty, \omega},$ we assume that arbitrary joins hence meets exist,
so $\D_a$ is complete, hence so is $\D$. But $\mathfrak{Cm}\alpha\subseteq \Nr_n\D$ is dense and complete, so
$\mathfrak{Cm}\alpha=\Nr_n\C$.

\end{proof}

\begin{remark}\label{game}

The game $J$ is stronger than the usual atomic game, since it allows \pa\ more moves.
But it is also an atomic game played on {\it atomic hypernetworks}; which carry more information than
networks.
But all the same,
because it an atomic game, it will turn out that it captures {\it only} the {\it atom structures} of neat reducts. This is
substantially weaker than capturing the  neat reduct itself.

However, one can define another  stronger {\it neat game}, denoted by $H$,
that {\it captures a neat reduct.}

Part of $H$ is an atomic game,
as far as \pe\ s response to \pa\ s moves, playing $\lambda$ neat hypernetworks, is concerned.
In the hypernetwork part of $H$, like in the weaker neat game $J$,
\pa\ can deliver only atoms for $n$ hyperedges and, like $J$, labels for $\lambda$ neat
hyperedges.

But there is also `another part' of the game.
On the board there are networks with no consistency conditions
and  \pe\ is allowed to label the $n$ hyperedges of these ($n$ is the dimension)
by arbitrary elements of the algebra.
These {\it do not} exist in the game $J$.

The main play of the stronger game $H(\B)$, $\B\in \CA_n$
is a play of the game $J(\B).$
Recall that $J(\B)$ was played on $\lambda$ neat hypernetworks.
The base of the main board at a certain point will be the the neat $\lambda$ hypernetwork, call its
network part $X$ and we write
$X(\bar{x})$ for the atom that labels the edge $\bar{x}$ on the main board.
But now \pa\ can make other moves too,
which makes it harder for \pe\ to win and so a \ws\ for \pe\ in this new $\omega$ rounded game
will give a stronger result.

An $n$  network is a finite complete graph with nodes including $0, \ldots, n-1$
with all edges labelled by {\it arbitrary elements} of $\B$. No consistency properties are assumed.

\pa\ can play an arbitrary $n$ network $N$, \pe\ must replace $N(0, \ldots, n-1),$  by
some element $a\in \B$. The idea, is that the constraints represented by $N$
correspond to an element of the $\sf RCA_\omega$ being constructed on $X$,
generated by $\B$.

The final move is that \pa\ can pick a previously played $n$ network $N$ and pick any  tuple $\bar{x}$
on the main board whose atomic label is below $N(0, \ldots, n-1)$.

\pe\ must respond by extending the main board from $X$ to $X'$ such that there is an embedding $\theta$ of $N$ into $X'$
 such that $\theta(0)=x_0\ldots , \theta(n-1)=x_{n-1}$ and for all $i_0, \ldots i_{n-1} \in N,$ we have
$X(\theta(i_0)\ldots, \theta(i_{n-1}))\leq N(i_0,\ldots, i_{n-1})$.
This ensures that in the limit, the constraints in
$N$ really define the element $a$.

If \pe\ has a \ws\ in the atomic game $J_k$
($J$ truncated to $k$ rounds) for all $k$ on an atomic
algebra $\A$ with countably many atoms,
then this algebra will have an elementary equivalent algebra $\B$ such that
$\At\B\in \At\Nr_n\CA_{\omega}.$
It can be shown that \pe\ has a \ws\ in $J_k$ for any finite $k$ on $\A\in \CA_{\N^-1,\N}$.

In \cite{r} similar games are devised for relation algebras and Robin Hirsch deduces from the fact that \pe\ can win the
$k$ rounded game (analogous to $J$) on a  certain relation algebra $\A$,
then $\A$ is elementary equivalent to $\B\in \Ra\CA_{\omega}$. This is a mistake.
All we can deduce from \pe\ s \ws\ is that $\B\in S_c\Ra\CA_{\omega}$.

But if \pe\ has a \ws\ in $H(\B)$, for $\B\in \CA_n$,
then the extra move ensure that  that every $n$ dimensional element generated by
$\B$ in the $\D\in \RCA_\omega$
constructed in the play is an element of $\B$, so that $\B$ exhausts all $n$ dimensional elements
of $\D$, hence $\B\cong \Nr_n\D$, and so $\B\in \Nr_n\CA_{\omega}$.

Hence if \pe\ has a \ws\ in $H_k(\A)$, $\A\in \CA_n$ atomic with countably many atoms,
for all finite $k$, then $\A$ will have an elementary equivalent algebra $\B$,
such that $\B\in \Nr_n\CA_{\omega}$.

This is much stronger, for like the case of representable algebras, and unlike the case of {\it completely}  representable algebra
even in the relativized sense,
we may well have algebras $\A, \B\in \CA_n$ $(n>1$),
and even more in $\sf RCA_n$
such that $\At\A=\At\B$, $\A\in \Nr_n\CA_{\omega}$ but $\B\notin \Nr_n\CA_{n+1}$, {\it a fortiori}
$\B\notin \Nr_n\CA_{\omega}$.
Unfortunately, \pe\ does not have a \ws\ on $\CA_{\N^-1, \N}$ for every finite rounded game of $H$.
So an open question is whether there is an atomic algebra $\A\in \CA_n$ with countably many
atoms such that \pe\ can win $H_k$ for every finite $k$, while \pe\ can win $F^m$ for some $m>n$.

If there is such an algebra, 
then any $\sf K$ between $\Nr_n\CA_{\omega}$ and $S_c\Nr_n\CA_{m}$ 
is not elementary. This is stronger than the result proved in theorem \ref{rainbow}.
Indeed from our present proof of the statement in \ref{rainbow} 
we cannot replace $S_c\Nr_n\CA_{\omega}$ by the strictly smaller 
$\Nr_n\CA_{\omega}$.
\end{remark}

Example \ref{SL} alerts us to the fact that it is possble to separate classes that are `controlled by atom structure of their members' 
but there those that cannot like the class of represenatble algebras of finite dimension
$n>2$, and 
more generally, the varieties  $S\Nr_n\TCA_{n+k}$ for $k\geq 3$ theorem \ref{can}.
So we can distinguish  between two `types' of classes of algebras,
those that are  {\it gripped by their atom structures}, call such a class $\K$,  meaning
that if $\A\in \sf K\subseteq TCA_{\alpha}$ and $\B\in \TCA_{\alpha}$ such that 
$\At\A=\At\B$,  iff $\B\in \K$.

An example here is the classes $S_c\Nr_n\CA_m$ for $1<m<n$; in particular this is true of the class of copmpletel;y representable algebras
of dimension $n$.  This notion will be elaborated upon below. 

Another example is the class $\Nr_{\alpha}\CA_\beta$ 
for any pair of ordinals $1<\alpha<\beta$.

If $\A, \B\in \K\subseteq \sf TCA_n$ then of coure $\At\A$ and $\At\B$ is in 
$\At\K$. The 
converse as we have already seen  shows that the converse is false.
There are classes of algebras thatr are not gripped by their atom structures.

Example \ref{SL} also told us 
that $\Nr_{\alpha}\TCA_{\alpha+\omega}$ is propery contained in 
$S_c\Nr_{\alpha}\TCA_{\alpha+\omega}$ for any ordinal $\alpha>1$. 

The next example 
shows that when $\alpha=n>2$ is finite, then the strictness of such inclusion can be actually witnessed by
a finite algebra.

\begin{example}\label{finitepa}
Take the finite cylindric algebra $\A$ consisting of three dimensional matrices (as defined by Monk)
over any integral non-permutational relation algebra $\R$, discretely topologized.
Such relation algebras exist  \cite[theorem 36]{r}. The algebra $\A$ is finite, 
hence completely representable, hence  $\A\in S_c\Nr_3\TCA{\omega}.$ 

Suppose for contradiction that $\A\in \Nr_3\TCA_{\omega}$,
so that $\At\A\in \At\Nr_n\CA_{\omega}.$
Then we claim that $\A$ has a $3$-homogeneous complete representation, which is impossible, because $\R$
does not have a homogeneous representation.

It can be shown, using arguments similar to \cite[theorem 33]{r},
that \pe\ has a \ws\ in an $\omega$ rounded game $K$ but played on atomic networks 
where \pa\ is offered a cylindrifier move 
together with an amalgmation move except that in amalgamation moves on networks there is an additional
restriction. 
The networks he chooses can overlap only on
at most $3$ nodes.
\pe\ uses her \ws\  to define a sequence
of networks $N_0\subseteq \ldots N_r$ such  that this sequence respects the cylindrifier
move in the sense that if $N_r(\bar{x})\leq {\sf c}_ia$ for $\bar{x}\in \nodes(N_r)$, then there
exists $N_s\supseteq N_r$ and a node $k\in \omega\sim N_r$ such that $N_s(\bar{y})=a$;
and also respects the partial isomorphism move, in the sense that if
if $\bar{x}, \bar{y}\in \nodes(N_r)$ such that $N_r(\bar{x})=N_r(\bar{y})$,
then there is a finite surjective map extending $\{(x_i, y_i): i<n\}$ mapping onto $\nodes(N)$
such that $\dom(\theta)\cap \nodes(N_r)=\bar{y}$,
and we seek an extension $N_s\supseteq N_r$, $N_r\theta$ (some
$s\geq r$).
Then if $\tau$ is a partial isomorphism
of $N_a$ and $X$ is any finite
subset of $\nodes(N_a)$ then there is a
partial isomorphism $\theta\supset \tau, \rng(\theta)\supset X$.

Define the representation  $\cal N$ of $\A$ with domain $\bigcup_{a\in A}\nodes(N_a)$, by
$$S^{\cal N}=\{\bar{x}: \exists a\in A, \exists s\in S, N_a(\bar{x})=s\},$$
for any subset $S$ of $\At\A$. 
Then this representation of $\A$,
is obviously complete, and by the definition of the game is $n$ homogeneous.

For higher dimension one uses the result in \cite{AU} by lifting the used relation algebra to 
a representable $\TCA_n$ for any $n>2$ preserving 
$n$-homogeneouty.
\end{example}

\begin{theorem}\label{tense} Let $n$ be finite $>2$. Then the following hold:
\begin{enumarab}
\item $\RTeCA_n$ is not finitely axiomatizable by any universal formulas
contaning only finitely many variables.
\item It is undecidable to tell whether a finite algebra is representable. In particular, the equational theory of
$\RTeCA_n$
is undecidable, the set of models having infinite bases is not recursively
enumerable, and $\RTeCA_n$ is  not finitely axiomatizable in $m$th
order logic for any finite $m$.
\end{enumarab}
\end{theorem}
\begin{proof}
\begin{enumarab}
\item  Assume that $\A\in \bold \RTeCA_n$.
Then $\Rd_{ca}\A\subseteq  \prod_{t\in T}\wp(^nU_t)\in \sf Gs_n$.
Conversely, ${\sf Cs}_n\subseteq \Rd_{ca}\TeCA_n$ (by fixing a moment in time).
Hence $\RCA_n={\bf SP}\Rd_{ca}\bold \RTeCA_n$, so  $\RTeCA_n$
is a finite  expansion of
$\sf RCA_n$ by only two unary modalities, namey $P$ (past) and $F$ (future)
hence the required  follows from \cite[Theorem 5]{Andreka}.

\item We show that from every simple atomic $\A\in \CA_n$ we can construct recursively
a ${\sf TeCA}_n$, such that
the former is representable if and only if
the latter is. This suffices by the main result in \cite{AU}. Expand $\A$ recursively and temporally
to a static $\sf TeCA_n$
as done before Take $G=H=Id$ and
$T=\{t\}$ with $<=\emptyset$; then $<$ is obviously irreflexive and vacuously
transitive and $G$ and $H$
distribute over the Boolean meet.
\end{enumarab}
\end{proof}

\subsection{ Neat embeddings in connection to
complete and strong representations}

Here we approach the notion of complete representations and strong representations using neat embeddings.
But before embarking on such connections we make an observation. 

The algebras $\C(m,n)$ for $2<m<n$, to be defined next witness the strictness
of this inclusion are based on  relations algebra similar to the algebras in \cite{HHbook2} 
but such algebras are  infinite
since there are an infinite set of hypernetworks with hyperlabels from $\omega$.

Now  the algebras consists of the set of the set of  $m$
basic matrices ($m$ is the dimension) on this finite relation
algebras hence, as stated, they are finite.

\begin{example}\label{thm:cmnr} Let $3\leq m< n<\omega.$
Then there are finite algebras $\C(m,n)\in \PEA_m$ such that
\begin{enumerate}
\renewcommand{\theenumi}{\Roman{enumi}}
\item $\C(m, n)\in \Nr_m\PEA_n$,\label{en:one}
\item $\Rd_{\Sc}\C(m, n)\not\in {\sf S}\Nr_m\Sc_{n+1}$. \label{en:two}
\end{enumerate}
In particular, for any class $\K$ between $\Sc$ and $\PEA$,
for any finite $m>2$ and any finite $k\geq 1$,
we have ${\sf S}_c\Nr_m\K_{m+k+1}\subset {\sf S}_c\Nr_n\K_{m+k}$ and the strictness of the inclusion is witnessed
on finite algebras.
We use a simpler version of algebras construction in \cite{t}, and they also have affinity
to algebras constructed in \cite[section 15.2]{HHbook}.
So we will be sketchy highlighting  the idea of proof.

Define a function $\kappa:\omega\times\omega\rightarrow\omega$ by $\kappa(x, 0)=0$
(all $x<\omega$) and $\kappa(x, y+1)=1+x\times\kappa(x, y))$ (all $x, y<\omega$).
For $n, r<\omega$ let
\[\psi(n)=
\kappa((n-1)\times 5, (n-1)\times 5)+1.\]
All of this is simply to ensure that $\psi(n)$ is sufficiently big compared to $n$ for the proof of non-embeddability to work.

For any  $n<\omega$, let
\[{\sf Bin(n)}=\set{\Id}\cup\set{a^k(i, j):i< n-1,\;j\in 5,\;k<\psi(n)}\]
where $\Id, a^k(i, j)$ are distinct objects indexed by $k, i, j$.

Let $F(m, n)$ be the set of all  functions $f:m\times m\to \sf Bin(n)$
such that $f$ is symmetric ($f(x, y)=f(y, x)$ for all $x, y<m$)
and for all $x, y, z<m$ we have $f(x, x)=\Id,\;f(x, y)=f(y, x)$, and $(f(x, y), f(y, z), f(x, z))\not\in {\sf Forb}$,
where ${\sf Forb}$ (the \emph{forbidden} triples) is the following set of triples
 \[ \begin{array}{c}
 \set{(Id, b, c):b\neq c\in {\sf Bin(n)}}\\
 \cup \\
 \set{(a^k(i, j), a^{k'}(i,j), a^{k^*}(i, j')): k, k', k^*< \psi(n, r), \;i<n-1, \; j'\leq j<5}.
 \end{array}\]
Here ${\sf Bin(n)}$ is an atom structure of a finite relation relation
and ${\sf Forb}$ specifies its operations by specifying forbidden triples.

Now any such $f\in F(m,n)$ is a basic matrix on this atom structure
in the sense of \cite[definition 12.35]{HHbook2}, a term due to Maddux, and the whole lot of them will be a
cylindric symmetric basis (closed under substitutions), a term also due to Maddux,
constituting the atom structure of algebras we want.
Now accessibility relations corresponding to substitutions, cylindrifiers are defined, as expected on matrices,
as follows.
For any $f, g\in F(m, n)$ and $x, y<m$ we write $f\equiv_{xy}g$ if for all $w, z\in m\setminus\set {x, y}$ we have $f(w, z)=g(w, z)$.
We may write $f\equiv_x g$ instead of $f\equiv_{xx}g$.  For $\tau:m\to m$ we write $(f\tau)$ for the function defined by
\begin{equation}\label{eq:ftau}(f\tau)(x, y)=f(\tau(x), \tau(y)).
\end{equation}
Clearly $(f\tau)\in F(m, n)$.
Accordingly, the universe of $\C(m, n)$ is the power set of $F(m, n)$ and the operators (lifting from the atom structure)
are
\begin{itemize}
\item  the Boolean operators $+, -$ are union and set complement,
\item  the diagonal $\diag xy=\set{f\in F(m, n):f(x, y)=Id}$,
\item  the cylindrifier $\cyl x(X)=\set{f\in F(m, n): \exists g\in X\; f\equiv_xg }$ and
\item the polyadic $\s_\tau(X)=\set{f\in F(m, n): f\tau \in X}$,
\end{itemize}
for $x, y<m,\;  X\subseteq F(m, n)$ and  $\tau:m\to m$.

It is straightforward to see
that  $3\leq m,\; 2\leq n$
the algebra $\C(m, n)$ satisfies all of the axioms defining $\PEA_m$
except, perhaps, the commutativity of cylindrifiers $\cyl x\cyl y(X)=\cyl y\cyl x(X)$, which it satisfies because
$F(m,n)$ is a symmetric cylindric basis, so that overlapping
matrices amalgamate.
Furthermore, if  $3\leq m\leq m'$ then $\C(m, n)\cong\Nr_m\C(m', n)$
via $$X\mapsto \set{f\in F(m', n): f\restr{m\times m}\in X}.$$

Now we prove  \ref{thm:cmnr}(\ref{en:two}), which is the heart and soul of the proof, and it is quite similar
to its $\CA$ analogue 4.69-475 in \cite{HHbook2}, and the proof in \cite{t}. In the latter two
proofs there is  a
third parameter $r$ which we fix to be $5$ here so that our proof is simpler.
We will refer to the proof in \cite{HHbook2} when the proofs overlap, or are very similar.
Assume for contradiction  that
$\Rd_{\Sc}\C(m, n)\subseteq\Nr_m\C$
for some $\C\in \Sc_{n+1}$, some finite $m, n$.
Then it can be shown inductively
that there must be a large  set $S$ of distinct elements of $\C$,
satisfying certain inductive assumptions, which we outline next.
For each $s\in S$ and $i, j<n+2$ there is an element $\alpha(s, i, j)\in \sf Bin(n)$ obtained from $s$
by cylindrifying all dimensions in $(n+1)\setminus\set{i, j}$, then using substitutions to replace $i, j$ by $0, 1$.
Then one shows that $(\alpha(s, i, j), \alpha(s, j, k), \alpha(s, i, k))\not\in {\sf Forb}$.

The induction hypothesis say, most importantly, that $\cyl n(s)$ is constant, for $s\in S$,
and for $l<n$  there are fixed $i<n-1,\; j<5$ such that for all $s\in S$ we have $\alpha(s, l, n)\leq a(i, j)$.
This defines, like in the proof of theorem 15.8 in \cite{HHbook2} p.471, two functions $I:n\rightarrow (n-1),\; J:n\rightarrow 5$
such that $\alpha(s, l, n)\leq a(I(l), J(l))$ for all $s\in S$.  The \emph{rank} ${\sf rk}(I, J)$ of $(I, J)$ (as defined in definition 15.9 in \cite{HHbook2}) is
the sum (over $i<n-1$) of the maximum $j$ with $I(l)=i,\; J(l)=j$ (some $l<n$) or $-1$ if there is no such $j$.

Next it is proved that there is a set $S'$ with index functions $(I', J')$, still relatively large
(large in terms of the number of times we need to repeat the induction step)
where the same induction hypotheses hold but where ${\sf rk}(I', J')>{\sf rk}(I, J)$.  (See \cite{HHbook2}, where for $t<n\times 5$,
$S'$ was denoted by $S_t$
and proof of property (6) in the induction hypothesis  on p.474 of \cite{HHbook2}.)

By repeating this enough times (more than $n\times 5$) we obtain a non-empty set $T$
with index functions of rank strictly greater than $(n-1)\times 4$, an impossibility.
(See \cite{HHbook2}, where for $t<n\times 5$, $S'$ was denoted by $S_t$.)

We sketch the induction step.  Since $I$ cannot be injective there must be distinct $l_1, l_2<n$
such that $I(l_1)=I(l_2)$ and $J(l_1)\leq J(l_2)$.  We may use $l_1$ as a "spare dimension"
(changing the index functions on $l$ will not reduce the rank).
 Since $\cyl n(s)$ is constant, we may fix $s_0\in S$
and choose a new element $s'$ below $\cyl l s_0\cdot \sub n l\cyl  l s$,
with certain properties.  Let $S^*=\set{s': s\in S\setminus\set{s_0}}$.
We wish to re-establish the induction hypotheses for $S^*$, and many of these are simple to check.
Although suitable functions $I', J'$ may not exist on the whole of $S$, but $S$ remains
large enough to enable selecting a
subset $S'$ of $S^*$, still large in terms of the number of remaining times the induction step must be applied.
The required functions $I', J'$ now exist (for all but one value of $l<n$ the values $I'(l), J'(l)$ are determined by $I, J$,
for  one value of $l$ there are at most $5\times (n-1)$ possible values, hence on a large subset the choices agree).

Next it can be shown that $J'(l)\geq J(l)$ for all $l<n$.   Since
$$(\alpha(s, i, j), \alpha(s, j, k), \alpha(s, i, k))\not\in Forb$$
and by the definition of ${\sf Forb}$
either $\rng(I')$ properly extends $\rng(I)$ or there is $l<n$ such that $J'(l)>J(l)$, hence  ${\sf rk}(I', J')>{\sf rk} (I, J)$,
and we are done.

\end{example}

\end{example}

Lifting a well known  notion from atom structures \cite{HHbook2}, 
we set:

\begin{definition} An atomic algebra $\A\in \TCA_n$ is strongly representable if it is completely additive 
its \de\ completion, namely $\Cm\At\A$ is representable.
\end{definition}

It is not hard to see that a completely representable algebra is strongly representable. However, the converse is false,
as we show in our next
theorem.  Also not every representable atomic algebra is strongly representable. The algebra 
$\mathfrak{Tm}\At$ 
constructed on the rainbow atom structure $\At$ 
in theorem  \ref{can} is such.

Let ${\sf CRTA_n}$ denote the class of completely representable $\K$ algebras of dimension $n$.
We have proved above that $S_c\Nr_n\TCA_{\omega}$ and ${\sf CRTCA}_n$ coincide on countable atomic algebras.
The result can slightly generalized to allow algebras with countably many atoms, that may not be countable.

In our next theorem we show that this does not generalize any
further as far as cardinalities are concerned.

\begin{theorem}\label{complete} For any $n>2$ we have, $\Nr_n\K_{\omega}\nsubseteq {\sf CRTCA}_{n}$,
while for $n>1$, ${\sf CRTCA}_n\nsubseteq {\sf UpUr}\Nr_{n}\TCA_{\omega}.$
In particular, there are completely representable, hence strongly representable algebras that are not in $\Nr_n\TCA_{\omega}$.
Furthermore, such algebras can be countable.
However, for any $n\in \omega$, if $\A\in \Nr_n\TCA_{\omega}$ is atomic, 
then $\Rd_{ca}\A$ is strongly
representable.
\end{theorem}
\begin{proof}
\begin{enumarab}
\item We first show that $\Nr_n\TCA_{\omega}\nsubseteq {\sf CRTCA}_{n}$.
This cannot be witnessed on countable algebras, so our constructed neat reduct that is not
completely representable,  must be uncountable.

In \cite{r} a sketch  of constructing  an uncountable relation algebra
$\R\in \Ra\CA_{\omega}$ (having an $\omega$ dimensional cylindric basis)
with no complete representation is given. It has a precursor in \cite{BSL}
which is the special case of this example when $\kappa=\omega$
but the idea in all three proofs are very similar
using a variant of the rainbow  relation algebra
$\R_{\kappa, \omega}$ where $\kappa$ is an uncountable cardinal $\geq 2^{\aleph_0}$.
We do not know whether it works for the least uncounatble cardinal, namely, $\aleph_1$

Assume that $\R=\Ra\B$ and $\B\in \PEA_{\omega}$, then
$\Rd_{K}\Nr_n\B$ is as required, for a complete representation of it, induces easily a complete
representation of $\Ra\CA_{\omega}$.
The latter example shows that ${\sf UpUr} \Nr_n\K_{\omega}\nsubseteq {\sf CRA}_n$,

Now give the details of the construction in \cite[remark 31]{r}.
We prove more, namely,  there exists an {\it uncountable} neat reduct, that is, an algebra in $\Nr_n\sf QEA_{\omega}$,
such that its $\sf Df$ reduct, obtained by discarding all operations except 
for cylindrifiers, is not completely representable.

Using the terminology of rainbow constructions,
we allow the greens to be of cardinality $2^{\kappa}$ for any
infinite cardinal $\kappa$, and the reds to be of cardinality $\kappa$.
Here a \ws\ for \pa\ witnesses that  the algebra has {\it no} complete
representation. But this is not enough because we want our algebra to be
in $\Ra\CA_{\omega}$; we will show that it will be.

We specify the atoms and forbidden triples.
The atoms are $\Id, \; \g_0^i:i<2^{\kappa}$ and $\r_j:1\leq j<
\kappa$, all symmetric.  The forbidden triples of atoms are all
permutations of $(\Id, x, y)$ for $x \neq y$, \/$(\r_j, \r_j, \r_j)$ for
$1\leq j<\kappa$ and $(\g_0^i, \g_0^{i'}, \g_0^{i^*})$ for $i, i',
i^*<2^{\kappa}.$  In other words, we forbid all the monochromatic
triangles.

Write $\g_0$ for $\set{\g_0^i:i<2^{\kappa}}$ and $\r_+$ for
$\set{\r_j:1\leq j<\kappa}$. Call this atom
structure $\alpha$.

Let $\A$ be the term algebra on this atom
structure; the subalgebra of $\Cm\alpha$ generated by the atoms.  $\A$ is a dense subalgebra of the complex algebra
$\Cm\alpha$. We claim that $\A$, as a relation algebra,  has no complete representation.

Indeed, suppose $\A$ has a complete representation $M$.  Let $x, y$ be points in the
representation with $M \models \r_1(x, y)$.  For each $i< 2^{\kappa}$, there is a
point $z_i \in M$ such that $M \models \g_0^i(x, z_i) \wedge \r_1(z_i, y)$.

Let $Z = \set{z_i:i<2^{\kappa}}$.  Within $Z$ there can be no edges labeled by
$\r_0$ so each edge is labelled by one of the $\kappa$ atoms in
$\r_+$.  The Erdos-Rado theorem forces the existence of three points
$z^1, z^2, z^3 \in Z$ such that $M \models \r_j(z^1, z^2) \wedge \r_j(z^2, z^3)
\wedge \r_j(z^3, z_1)$, for some single $j<\kappa$.  This contradicts the
definition of composition in $\A$ (since we avoided monochromatic triangles).

Let $S$ be the set of all atomic $\A$-networks $N$ with nodes
 $\omega$ such that $\{\r_i: 1\leq i<\kappa: \r_i \text{ is the label
of an edge in N}\}$ is finite.
Then it is straightforward to show $S$ is an amalgamation class, that is for all $M, N
\in S$ if $M \equiv_{ij} N$ then there is $L \in S$ with
$M \equiv_i L \equiv_j N.$
Hence the complex cylindric algebra $\Ca(S)\in \QEA_\omega$.

Now let $X$ be the set of finite $\A$-networks $N$ with nodes
$\subseteq\omega$ such that
\begin{enumerate}
\item each edge of $N$ is either (a) an atom of
$\A$ or (b) a cofinite subset of $\r_+=\set{\r_j:1\leq j<\kappa}$ or (c)
a cofinite subset of $\g_0=\set{\g_0^i:i<2^{\kappa}}$ and
\item $N$ is `triangle-closed', i.e. for all $l, m, n \in \nodes(N)$ we
have $N(l, n) \leq N(l,m);N(m,n)$.  That means if an edge $(l,m)$ is
labeled by $\Id$ then $N(l,n)= N(mn)$ and if $N(l,m), N(m,n) \leq
\g_0$ then $N(l,n)\cdot \g_0 = 0$ and if $N(l,m)=N(m,n) =
\r_j$ (some $1\leq j<\omega$) then $N(l,n).\r_j = 0$.
\end{enumerate}
For $N\in X$ let $N'\in\Ca(S)$ be defined by
\[\set{L\in S: L(m,n)\leq
N(m,n) \mbox{ for } m,n\in nodes(N)}\]
For $i\in \omega$, let $N\restr{-i}$ be the subgraph of $N$ obtained by deleting the node $i$.
Then if $N\in X, \; i<\omega$ then $\cyl i N' =
(N\restr{-i})'$.
The inclusion $\cyl i N' \subseteq (N\restr{-i})'$ is clear.

Conversely, let $L \in (N\restr{-i})'$.  We seek $M \equiv_i L$ with
$M\in N'$.  This will prove that $L \in \cyl i N'$, as required.
Since $L\in S$ the set $X = \set{\r_i \notin L}$ is infinite.  Let $X$
be the disjoint union of two infinite sets $Y \cup Y'$, say.  To
define the $\omega$-network $M$ we must define the labels of all edges
involving the node $i$ (other labels are given by $M\equiv_i L$).  We
define these labels by enumerating the edges and labeling them one at
a time.  So let $j \neq i < \omega$.  Suppose $j\in \nodes(N)$.  We
must choose $M(i,j) \leq N(i,j)$.  If $N(i,j)$ is an atom then of
course $M(i,j)=N(i,j)$.  Since $N$ is finite, this defines only
finitely many labels of $M$.  If $N(i,j)$ is a cofinite subset of
$a_0$ then we let $M(i,j)$ be an arbitrary atom in $N(i,j)$.  And if
$N(i,j)$ is a cofinite subset of $\r_+$ then let $M(i,j)$ be an element
of $N(i,j)\cap Y$ which has not been used as the label of any edge of
$M$ which has already been chosen (possible, since at each stage only
finitely many have been chosen so far).  If $j\notin \nodes(N)$ then we
can let $M(i,j)= \r_k \in Y$ some $1\leq k < \kappa$ such that no edge of $M$
has already been labeled by $\r_k$.  It is not hard to check that each
triangle of $M$ is consistent (we have avoided all monochromatic
triangles) and clearly $M\in N'$ and $M\equiv_i L$.  The labeling avoided all
but finitely many elements of $Y'$, so $M\in S$. So
$(N\restr{-i})' \subseteq \cyl i N'$.

Now let $X' = \set{N':N\in X} \subseteq \Ca(S)$.
Then the subalgebra of $\Ca(S)$ generated by $X'$ is obtained from
$X'$ by closing under finite unions.
Clearly all these finite unions are generated by $X'$.  We must show
that the set of finite unions of $X'$ is closed under all cylindric
operations.  Closure under unions is given.  For $N'\in X$ we have
$-N' = \bigcup_{m,n\in \nodes(N)}N_{mn}'$ where $N_{mn}$ is a network
with nodes $\set{m,n}$ and labeling $N_{mn}(m,n) = -N(m,n)$. $N_{mn}$
may not belong to $X$ but it is equivalent to a union of at most finitely many
members of $X$.  The diagonal $\diag ij \in\Ca(S)$ is equal to $N'$
where $N$ is a network with nodes $\set{i,j}$ and labeling
$N(i,j)=1'$.  Closure under cylindrification is given.
Let $\C$ be the subalgebra of $\Ca(S)$ generated by $X'$.
Then $\A = \Ra(\C)$.
Each element of $\A$ is a union of a finite number of atoms and
possibly a co-finite subset of $a_0$ and possibly a co-finite subset
of $a_+$.  Clearly $\A\subseteq\Ra(\C)$.  Conversely, each element
$z \in \Ra(\C)$ is a finite union $\bigcup_{N\in F}N'$, for some
finite subset $F$ of $X$, satisfying $\cyl i z = z$, for $i > 1$. Let $i_0,
\ldots, i_k$ be an enumeration of all the nodes, other than $0$ and
$1$, that occur as nodes of networks in $F$.  Then, $\cyl
{i_0} \ldots
\cyl {i_k}z = \bigcup_{N\in F} \cyl {i_0} \ldots
\cyl {i_k}N' = \bigcup_{N\in F} (N\restr{\set{0,1}})' \in \A$.  So $\Ra(\C)
\subseteq \A$.
$\A$ is relation algebra reduct of $\C\in\CA_\omega$ but has no
complete representation; so $\C^{\sf top}\in \TCA_{\omega}$.

Let $n>2$. Let $\B=\Nr_n \C^{\sf top}$. Then
$\B\in \Nr_n\TCA_{\omega}$, is atomic, but has no complete representation; 
in fact because it is binary generated its
$\Df$ reduct is not completely representable.

\item Now we show that ${\sf CRK}_n\nsubseteq {\sf UpUr}\Nr_n\CA_{\omega}$.
The $\K$ reduct of the algebra $\A$ in example  \ref{SL} is such;
it is completely representable, hence strongly representable, but it is not in ${\sf UpUr}\Nr_n\K_{n+1}$, {\it a  fortiori}
it is not in ${\sf UpUr}\Nr_n\K_{\omega}$. If the field $\F$ is countable, then $\A$ is countable.

\item Now for the last part, namely, that  neat reducts are strongly representable.
Let $\A\in \Nr_n\K_{\omega}$ be atomic. Then \pe\ has a \ws\ in $F^{\omega}$
as in theorems \ref{can} and \ref{rainbow}, 
hence it has a \ws\ in $G$ (the usual $\omega$ rounded atomic game) and so it has
a \ws\ for $G_k$ for all finite $k$ ($G$ truncated to $k$ rounds.)
Thus $\A\models \sigma_k$ which is the $k$ th Lyndon sentence coding that \pe\ has a
\ws\ in $G_k$, called the $k$the Lyndon condition. Since $\A$ satisfies the $k$th
Lyndon conditions for each $k$,  then any algebra on its atom structure is representable,
so that $\Cm\At\A$ is representable, hence  it is strongly representable, and we are done.
\end{enumarab}
\end{proof}
Recall that $S_c$ is the operation of forming complete subalgebras. We write $\A\subseteq_c \B$ if $\A$ is a complete subalgebra of
$\B$.

\begin{theorem} For $n>1$ the class $\Nr_n\TCA_{\omega}$ is not closed under $S_c$. For every $n\in \omega$ the class
$S_c\Nr_n\TCA_{\omega}$ is closed under $S_c$ but for $n>2$ it is not elementary and is
not closed under forming subalgebras, hence is not pseudo-universal.
For $n>2$, the class ${\sf UpUr}S_c\Nr_n\K_{\omega}$
is not finitely axiomatizable.
\end{theorem}
\begin{proof}  For $n>1$ $\Nr_n\TCA_{\omega}$ is not closed under $S_c$ follows, from example 
{SL} since the the algebra $\A$
constructed therein is countable and completely representable but is not in $\Nr_n\TCA_{n+1}$, {\it a fortiori} 
it is not in $\Nr_n\TCA_{\omega}$.
 $S_c\Nr_n\TCA_{\omega}$ is not elementary follows from theorem \ref{rainbow}.

We now show that $S_c\Nr_n\TCA_{\omega}$ is not closed under forming subalgebras, hence it is not pseudo-universal.
That it is closed under $S_c$ follows directly from the definition.
Consider rainbow algebra used in theorem \ref{rainbow}, call  it $\A$. Because $\A$ has countably many atoms, $\Tm\A\subseteq \A\subseteq \Cm\At\A$,
and all three are completely representable or all three not completely representable sharing the same atom structure $\At\A$,
we can assume without loss that $\A$ is countable.
Now $\A$ is not  completely representable hence it is not $S_c\Nr_n\TCA_{\omega}$.

On the other hand, $\A$ is strongly
representable (from the argument used in the second item of theorem \ref{complete} since \pe\ can win the finite rounded
atomic game with $k$ rounds for every $k$, hence it satisfies the Lyndon conditions),
so its canonical extension is representable, indeed completely representable.
On the other hand, the canonical extension of $\A$
is in $S_c\Nr_n\sf TCA_{\omega}$ and $\A$ embeds into its canonical
extension, hence $S_c\Nr_n\TCA_{\omega}$ is not closed under forming subalgebras.
The first of these statements follow from the fact that if $\D\subseteq \Nr_n\B$,
then $\D^+\subseteq_c \Nr_n\B^+$.

To prove non-finite axiomatizability of the elementary closure of
$S_c\Nr_n\TCA_{\omega}$ we give two entirely different proofs.

\begin{enumroman}

\item  First we use the flexible construction in \cite{ANT}.

Now let $l\in \omega$, $l\geq 2$, and let $\mu$ be a non-zero cardinal. Let $I$ be a finite set,
$|I|\geq 3l.$ Let
$J=\{(X,n): X\subseteq I, |X|=l,n<\mu\}.$
Here  $I$ is the atoms of $\M$. $J$ is the set of blurs, consult  \cite[definition 3.1]{ANT} for the definition of blurs.
Pending on $l$ and $\mu$, let us call these atom structures ${\cal F}(I, l,\mu).$
If $\mu\geq \omega$, then $J$ would be infinite,
and $\Uf$, the set of non principal ultrafilters corresponding to the blurs, will be a proper subset of the ultrafilters.
It is not difficult to show that if $l\geq \omega$
(and we relax the condition that $I$ be finite), then
$\Cm{\cal F}(I, l,\mu)$ is completely representable,
and if $l<\omega$, then $\Cm{\cal F}(I, l,\mu)$ is not representable.

Let ${\D}$ be a non-trivial ultraproduct of the atom structures ${\cal F}(I, i,1)$, $i\in \omega$. Then $\Cm{\D}$
is completely representable.
Thus $\Tm{\cal F}(I, i,1)$ are ${\sf RRA}$'s
without a complete representation while their ultraproduct has a complete representation.

Also the sequence of complex algebras $\Cm{\cal F}(I, i,1)$, $i\in \omega$
consists of algebras that are non-representable with a completely representable ultraproduct.

Then because such  algebras posses
symmetric $n$  dimensional cylindric basis (closed under substitutions),
the result lifts easily to discretely toplogized $\CA$s

\item  Alternatively, one can prove the cylindric case directly as follows.Take $\G_i$ to be the disjoint union of cliques of size $(n(n-1)/2)+i$, or
let $\G_i$ be the graph with nodes $\N$ and edge relation $(i,j)\in E$ if $0<|i-j|<n(n-1)/2+i$.
Let $\alpha_i$ be the corresponding atom structure, as defined in \cite{weak} and $\A_i$ be the term polyadic equality
algebra based on $\alpha_i$. Then $\Cm \A_i$ is not representable because, as proved in \cite{weak},
$\alpha$ is weakly but not strongly representable. In fact, the diagonal free 
reduct of $\Cm\A_i$ is not representable.

But $\Pi_{i\in \omega}\Cm\A_i/F=\Cm(\Pi_{i\in \omega}\A_i/F)$, so
the one  graph is based on the disjoint union of the cliques which is arbitrarily large, and the second on graphs
which have arbitrary large chromatic
number, hence both, by discrete topologizing  are completely representable.

\end{enumroman}

\end{proof}

We prove our theorem on failure of omitting types for $\TL_n$, the first order topological logic with equality
restricted to the first $n$ variables,  viewed as a multi-modal logic.  The corresponding class
of modal algebras are $\PEA_n.$ $n$ remains to be finite $>2$.
But first a definition:

\begin{definition} let $L$ be a signatue and $M$ a structure for $L$ endowed with an Alexandrov topology. 
Let $\leq$ be the corresponding pre-order on $M$ and let 
$1<m<n$ be finite.
\begin{enumarab}
\item The {\it Gaifman hypergraph $\C^n(M)$} of $M$ is the graph $(\dom(M), E)$ 
where $E$ is the $m$ hyperedge relation such that for $a_1, \ldots a_{m}\in M$, 
$E(a_1,\ldots a_m)$ holds if there 
are $n$ and  an $n$ ary  relation symbol  $R$ and $a_1, \ldots a_m\in M$, 
such that $M\models R(a_1,\ldots, a_n)$ and $a_0,\ldots a_{m-1}\in  \{a_0,\ldots a_n\}.$

2. An $m$ clique in $\C^n(M)$ is a set $C\subseteq M$ such 
$E(a_1\ldots a_{m-1})$
for distinct $a_0, \ldots, a_{m-1}\in C.$

\item The clique guarded semantics 
$M\models _G, \phi(\bar{a})$ where
$\phi$ an $L$ formula, and $\bar{a}\in M$ and $\rng\bar{a}$ is an
$m$ clique are defined by:

\begin{itemize}
\item For atomic $\phi$,  $M\models_G \phi(\bar{a})$ iff $M\models \bar{a}.$

\item The semantics of the Boolean connectives are defined
the usual way.

\item For $s\in {}^nM$, $M, s\models \exists x_i\phi$ iff there is a $t\in {}^nM$, $t\equiv_i s$ such that  
$M, t\models \phi$ and 
$\rng(t)$ is  in 
$m$ clique in $\C^n(M)$.

\item  For $s\in {}^nM$, $M\models \Diamond_i \phi(s)$ iff there exists $t\in {}^nM$ such
that $t\equiv_i s$, $t_i\leq s_i$, $M, t\models \phi$ 
and $\rng(t)$ is an $m$ clique in $\C^n(M)$. 
\end{itemize}
\end{enumarab}
\end{definition}

\begin{definition}

Let $\A\in {\sf TCA}_n$.
Assume that $2<n <k$. Let $\L(\A)^k$ be the first order signature consisting of an $n$ ary relation symbol for each
$ a\in A$, using $k$ variables. Let $\C^k(\M)$ be the $k$ Giafman graph of an $\L_n$ structure
$\M$ carrying a topolgy.

\begin{enumarab}
 
\item A topological space 
$\M$ is a relativized representation of $\A$ if there exists an injective homomorphism $f:\A\to \wp(V)$ where $V\subseteq {}^nM$
and $\bigcup_{s\in V}\rng(s)=M$.

\item For $\bar{s}\in V$ and $a\in \A$, we write
$\M\models a(\bar{s})$ iff $\bar{s}\in f(a)$.
Then $\M$  is said to be {\it $k$ square},
if whenever $\bar{s}\in \C^k(\M)=\{s\in {}^kM: \text { $\rng(s)$ is an $n$ clique}\}$,
$a\in A$, $i<n$,
and injection  $l:n\to k$, if $\M\models {\sf c}_ia(s_{l(0)}\ldots, s_{l(n-1)})$,
then there is a $t\in \C^k(\M)$ with $\bar{t}\equiv _i \bar{s}$,
and $\M\models a(t_{l(0)},\ldots, t_{l(n-1)})$.

\item $\M$ is said to be {\it $k$ flat} if  it is $k$ square and
for all $\phi\in \L(\A)^k$, for all $\bar{s}\in \C^k(\M)$, for all distinct $i,j<k$,
we have
$$\M\models_G [\exists x_i\exists x_j\phi\longleftrightarrow \exists x_j\exists x_i\phi](\bar{s}).$$
Here the subscript $G$ refers to the $\bold G$iafman 
clique guarded semantics defined above, which we henceforth refer to as {\it Giafman}
semantics.
\end{enumarab}
\end{definition}

$\bold K_{\sf n, square}$ denotes the class of square frames
of dimension $n$, $\bold K_{\sf n, k, square}$ and $\bold K_{\sf n, k, flat}$ denote the class of Kripke frames of dimension $n$
whose modal algebras have  $k$ square, $k$
flat representations, respectively.

That is 
\begin{align*}
\bold K_{\sf n, square}&= \sf Str{{\sf RTCA}_n}\\
\bold K_{\sf n, k, square}&=\{\F: \Cm\F \text{ has a $k$ square representation}\}.\\
\bold K_{\sf n, k, flat}&=\{\F: \Cm\F \text{ has a $k$ flat representation}\}.
\end{align*}
We have the following deep results hat can obtained 
from their $\CA$  analogues by discrete topologizing and same holds for tense and temporal cylindric algebras by staic temporalizing.
We formulate the plethora of negative results, with very few exception like undecidability
of the equation theory of any class between $\sf RTCA_2$ and $\TCA_2$, a significant 
(negative) deviation from the the theory
of $\CA_2$s,  but admittedly utterly unsurprising due to
the presence of two $\sf S4$ modalities induced by the topology, 
witness theorem \ref{ud}. 
However, in interesting and for that matter sharp contrast,  
like $\sf CA_2$, the equational theories of $\sf TeCA_2$ and $\sf TemCA_2$ are decidable, theorem 
\ref{ud} and the corresponding 2 dimensional modal logic has the finite model property, 
and they (the corresonding modal algebras) have the finite algebra finite base property, any finite algebra has a finite representation.

If $\TL_m$ is the multi modal topological logic corresponding to cubic frames, that is, frames whose domains are of the form $^mU$,
with $m$ still $>2$. It is not to hard to distill from the literature the following:
\begin{itemize}
\item $\TL_m$ is  undecidable (this can be proved from non atomicity of free algebras
or by coding the word problem for finitely presented semigroups), for $m=2$ it is also undecidable; 
which is interesting (but expected) deviation from $\CA_2$,
because roughly the existence of the two $\S4$ modalities induced by the topology on the base of a representation,
is stronger than the undecidable product logic $\sf S4\times \sf S4,$
\item $\TL_m$ is not finitely axiomatizable,
\item it is undecidable to tell whether a finite frame is a frame for $\L_m$, this implies that its modal algebras cannot be finitely axiomatizable in
$k$th order first order logic for any finite $k>0,$
\item $\TL_m$ lacks Craig interpolation and Beth definability; this will happen too for $m=2$,
\item the class of such frames cannot be Sahlqvist axiomatizable,
\item even more  it cannot be axiomatized by any set of first order sentences,
\item though canonical, any axiomatization would necessarily
contain infinitely many non canonical sentences
(canonical sentences are sentences whose algebraic equivalents
are preserved in canonical
extensions).
\item $\TL_m$ has Godel's incompleteness theorem.
\end{itemize}

The algebraic equivalences of the last  theorem are (in the same order):
$m$ is finite $>2$ and ${\sf RTCA}_m$
denotes the variety of representable algebras.

\begin{itemize}

\item The equational theory of ${\sf RTCA}_m$ 
is undecidable \cite[Theorem 5.1.66]{HMT2}; 
this holds  even for $m=2$, witness theorem \ref{ud} below,

\item ${\sf RTCA}_m$ is not finitely axiomatizable. This follows easily by topologizing Monk's classical 
result proved for $\CA$s, \cite{HMT2},
There are sharper results proved by 
Biro, Hirsch, Hodkinson, Maddux and others 
that can  also be topologized, like for example 
making Monk's algebras binary generated, thus it automorphisms group becomes 
rigid, this is far from being trivial. 

Another such sharp result is that 
$S\Nr_m\TCA_{m+k+1}$ is not finitely axiomatizable over $S\Nr_m\TCA_{m+k}$ for any finite 
$k\geq 1$, and this lifts to infinite dimensions replacing finitely axiomatizable by axiomatizable by a Monk's schema 
as defined in \cite{HMT2} under the name of {\it systems of varieties definable by a finite schema} \cite{HHbook2, t}.

We have already dealt with Andr\'eka's complexity results indicating 
how they can all lift to the topological 
addition, witness theorem \ref{tense} above,

\item It is undecidable to tell whether a finite algebra is representable \cite{d, HHbook} witness theorem \ref{tense}
above, 

\item ${\sf RTCA}_m$ does not have the amalgamation property, this also works for $m=2$ \cite{ACMNS},

\item ${\sf RTCA}_m$ is not atom-canonical, hence not closed under \de\ completions, theorem \ref{can},
Actually theorem \ref{can} proves this result for $S\Nr_m\TCA_{m+k}$ with $k\geq 3$, 
which mean that for any $k\geq 3$ (infinite included giving $\sf Str(TRCA_m)$ )
the class of frames
$\bold K_{\sf m, m+k, flat}=\{\F: \Cm \F\\in S\Nr_m\TCA_{m+k}\}$, and 
$\bold K_{\sf m, m+k, square}$ 
are  not Sahlqvist axiomatizable. When $k=\omega$; it is furthermore not first order axiomatizable at all, 
as indicated in the next item so that the 
standard second order translation from modal formulas to second order formulas contain genuine 
second order formulas.

When $k$ is finite 
the question is open. It is not known whether there exists
a sequence of weakly representable cylindric atom structures (that can be easily topologized) 
whose ultrapoduct is outside $S\Nr_m\TCA_{m+k}$ for some $3\leq k<\omega$. 

For $k=\omega$ this is proved by constructing a sequence of good Monk-like algebras 
based on graphs with infinite chromatic numder converging to a bad Monk-like algebra based on
a graph that is $2$ colourable. Witness \cite{strong} for the terminology good and bad Monk algebras.  
Roughly a bad Monk's algebra is one that is based on a graph that has a finite colouring and a good one is
based on a graph that does not have  finite chromatic number, in other words its chromatic number is infinite.
Finite colorins forbit represenations of Monk algebras based on them. Conversely, Monk-like algebras based on graphs
with infinite chromatic number are representable \cite{HHbook2}.

Roughly this is a reverse process to Monk's original construction back in 1969.
Monk constructed a sequence of bad Monk's algebras converging to
a good one which is more believable, and intuitive. It boils down to constructing 
graphs having larger and larger finite chromatic number converging to one with infinite 
chromatic number. Indeed, for the reverse process, it took Hirsch and Hodkinson the use 
of Erdos' probabilistic graphs to
obtain their amazing result.

\item  The class of strongly representable atom structures, namely the class
${\sf Str}({\sf RTCA}_m)=\{\F: \Cm\F\in {\sf RTCA}_m\}$
is not elementary \cite{strong}. 

\item Though canonical, i.e closed under canonical extensions, any
axiomatization of ${\sf RTCA}_m$  must contain
infinitely
many non-canonical sentences (sentences that are not preserved in canonical extensions).
Quoting Hodkinson: ${\sf RTCA}_m$ is only {\it barely canonical} \cite{bh}. 

\item The free finitely generated algebras 
of any $\K$ between $\sf RTCA_m$ and $\TCA_m$ are not atomic since they are stronger than 
that of $\TCA_3$, hence one can use the same construction
of N\'emeti's  by stimulating quasi-projections and coding $ZFC$ in 
the equational theory of 
$\sf TCA_m$ \cite{an}.

\item We use a deep (unpublished) result of 
N\'emeti's.  N\'emeti  shows that there are three $\sf CA_3$ terms  $\tau(x)$, $\sigma(x)$
and $\delta (x)$ such that for $m\geq 3$, we have
${\sf RCA}_m\models \sigma(\tau(x))=x$ and ${\sf RCA}_m\models \delta(\tau(x))=1$ but not 
${\sf Cs}_m\models \delta(x)=1$.
Then for every $m\geq 3$ we have 
(a) ${\sf TRCA}_m\models \sigma(\tau(x))=x$ 
and 
(b) ${\sf TRCA}_m\models \delta(\tau(x))=1$ but not ${\sf TCs}_m\models \delta(x)=1.$
The first two validities follow from the fact that these terms do not contain modalities, and 
the last is obtained by giving the base of the the set algebra falsifying 
the equation  the discrete toplogy; the resulting algebra flasifies the same equation since it also contains no 
modalities.

Let $0<\beta$, and $n\geq 3$ and let $\{g_i:i<\beta\}$ be an arbitray generator set of $\Fr_{\beta}\sf TRCA_m.$
Then $\{\tau(g_0)\}\cup \{g_i: 0<i<\beta\}$ generates $\Fr_{\beta}\sf TRCA_n$ by (a) but not freely by (b).
Let $\sf K$ be the class of all finite algebras in $\sf TRCA_m$.
Then ${\bf HSP}\sf K\neq \sf TRCA_m.$
In particular, there are equations valid in all finite $\sf TRCA_m$'s 
but they are not valid in all $\sf TRCA_m$'s.

\end{itemize}

\begin{theorem}\label{ud}
\begin{enumarab}
\item $\sf TRCA_1$ is finitely axiomatizable, has $fmp$  and is decidable.
\item $\sf TRCA_2$ and $\sf T^{al}RCA_2$ are finitely axiomatizable but they do not have the finite base property and 
their equational theory is undecidable. Furthermore, $\sf S54^{m}$ 
does not have abstract $fmp$.
\item However, the equational theory of any class between 
$\sf TeCA_2$ and $\sf TemCA_2$, and their concrete versions, namely,
the representable algebras.
\item For $\sf K$ as in the previous item the free algebras are atomic, 
and ${\sf K}={\bf HSP}\{\A\in  {\sf K}: |A|<\omega\}.$
\end{enumarab}
\end{theorem}
\begin{proof}
\begin{enumarab}
\item Follows from the fact that the bi-modal logic of Kripke frames  of the form $(U, U\times U, R)$ where $R$ is a pre-order 
is decidable and has $fmp$

\item  Let $\Sigma$ be a finite set of equations that axiomatize $\sf RCA_2$, together with the $\sf S4$ axioms.
Then if $\A\models \Sigma$, then $\Rd_{ca}\A\in \sf RCA_2$. $\A$ itself can be 
proved representable as follows:

We have $\B\subseteq \Nr_2\C$ where $\C\in {\sf Lf}_{\omega}.$
Let $\bar{h}:\C\to \D$ where $\D\in {\sf Cs_{\omega}}^{reg}\cap \sf Lf_{\omega}$
be such that $\bar{h}\upharpoonright m=h$. Let $G$ be the corresponding 
Henkin ultrafilter in $\D$, so that $\bar{h}=h_G$. 
For $p\in \A$ and $i<m$, let 
$O_{p,i}=\{k\in \omega: {\sf s}_i^kI(i)^{\A}p\in G\}$   
and let ${\cal B}=\{O_{p,i} : i\in n, p\in A\}.$
Then ${\cal B}$ 
is the base for a topology on $\omega$ and the concrete interior operations are defined 
for each $i<m$ via
$J_i: \wp(^2\omega)\to \wp (^2\omega)$
$$x\in J_iX\Longleftrightarrow \exists U\in {\cal B}(x_i\in U\subseteq \{u\in \omega: x^i_{u}\in X\}),$$
where $X\subseteq {}^{n}\omega$. 
Then $h^*: \A\to \wp(^2U, J_i)$ defined by $h^*(x)=h(x)$ for $x\in B$
is an isomorphism, such that $h^*(a)\neq 0$, which contradicts that $\Gamma\models \phi$, because 
$a=\neg\phi/\Gamma$.

There is formula using
only the boxes such that if $\phi$ is satsifiable in a 
frame $\F=(U, R_1)\times (U, R_2)$ where $R_1, R_2$ are symmetric and transitive
that is $\F\models \phi$ iff $\F$ is an $\infty$ chessboard. For such a frame $\F$, 
let $\F^+=(U, U\times  U, R_1)\times (U,  U\times U, R_2)$, then because 
the constructed 
formula does not contain the two modalities
${\sf c}_0$ and  ${\sf c_1}$,  we also have 
$\F^+\models \phi$ iff $\F$ is an $\infty$ chessboard board.
Let $e$ be the equation corresponding to $\phi$ in the sublanguage
of that of $\sf TCA_2$ obtained 
by dropping cylindrifiers. 
Such a formula forces 
infinite frames, and so the required follows for $fmp$ because $\Cm\F^+$ cannot be represented 
on finite sets, since $\Cm\F^+\models e$. 
Hence such varieties do not have $fmp$.

To prove undecidability a suffiently complex problem 
for Turing machines or tilings is reduced to the satisfibaility problem
for the multimodal logic at hand.

We use  the result proved in \cite{transitive}, namely, that $\sf S4\times S4$ is undecidable 
There is a formula  $\psi_M$, without using the modalities ${\sf c}_0$ and ${\sf c}_1$ such that 
$\psi_M$ is satisfiable in ${\sf log}\{\F: \Cm\F\in \sf TRCA_n\}$ 
iff $M$ stops having started for an all blank tape 
\cite[p. 1014]{transitive}.

Abother way is to show that there are three  modal formulas using $\Box_0, \Box_1$, 
denoted by  $\phi_{\infty}$, $\phi_{grid}$, and $\phi_{\Theta}$ in \cite[p. 1006-13]{transitive} 
where $\Theta$ is a finite set of tile types, {\it encoding tilings; that is encoding the $\mathbb{N}\times \mathbb{N}$ grid},  
in the sense that  $\Theta$ tiles $\mathbb{N}\times \mathbb{N}$, 
iff their conjunction is satisfiable 
in any $\infty$ chessboard.  More explicity  $\Theta$ tiles $\mathbb{N}\times \mathbb{N}$, iff 
for every frame $\F=(U\times U, T_0, T_1, \Box_0, \Box_1, \delta)$, there exists 
$s=(s_0, s_1)\in U\times U$ 
such that $\Rd \F=(U\times U, \Box_0, \Box_1), s\models \phi_{\infty}\land \phi_{grid}\land \phi_{\Theta}$
and furthermore    $\Rd \F$ is an $\infty$ chessboard.

\item Follows from the fact that the first order correspondance of the modal formulas of modal logic with $U$ and
$S$ lands in the loosely guarded fragment. 
Hence we have a disjoint union of decidable first order theories 
that of $\{\F: \Cm\F \in \sf RCA_2\}$  together with first order correspondants
of the modal formulus axiomatizing $\TeCA_2$ 
which is decidable. Since both have $fmp$ so does $\sf TeCA_2$.

\item The last required is exactly like the $\CA_2$ case \cite{HMT1}, since every 
finite algebra has a finite representation.
\end{enumarab}
\end{proof}

\subsection{Further negative results}

Now we prove that the omitting types theorem fails in a very strong sense for $TL_m$, when $m\geq 2$.
We have $\bold S\Cm\bold K_{\sf n, square}={\sf RTCA}_n$.
It can also be proved that  
$\bold S\Cm\bold K_{\sf n, k, flat}=S\Nr_n{\sf TCA}_{k}$; we will prove $\subseteq$ below,
which is known to be strictly contained in ${\sf TRCA}_n$
for any $n, k$ such that $2<n<k<\omega.$

Given $v:\omega \to {\sf Tm (X)},$ $\A\in \TCA_m$ and any map $s:\omega\to \A$, then by
$\bar{s}$ we denote the unique extension of $s$, $\bar{s}: {\sf Tm(X)}\to \A$, such that
$\bar{s}\circ i=s$.
Here we are looking at ${\sf Tm(X)}$
as the absolutely free algebra on $X$; $i$ is the inclusion map from $X$ into
$\sf Tm(X)$.

\begin{definition}
\begin{enumarab}
\item Let $\bold K$ be a class of frames, and $\A\in \TCA_m$ be countable
and $\F\in \bold K$. Then a model $\M=(\F,v)$ is {\it a representation of $\A$},
if there exists a bijection $s:X\to \A$ ($X$ is an infinite countable set of variables)
and  an injective homomorphism
$f:\A\to \mathfrak{Cm}\F$  such that $f\circ \bar{s}=v$. 
In this case, we say that $f$ is a representing function.

\item If $\A$ is atomic, then
$f$ is called a {\it complete representation of $\A$} if furthermore
$\bigcup_{x\in \At\A} f(x)=F$.

\end{enumarab}
\end{definition}

We show that when we broaden considerably the class of allowed models,
permitting  $m+3$ flat ones, there still might not be countable models omitting a single non-principal
type. Furthermore, this single-non principal type can be chosen to be the set of co-atoms in an atomic
theory. In the next theorem by $T\models \phi$ we mean that $\phi$ is valid in any (classical) topological 
model of $T$. That is for any topological model 
$\M$ of $T$,  any $s\in {}^mM$, $\M\models \phi[s]$; using the modal notation 
$\M, s\models \phi$.
\begin{definition}
\begin{enumarab}
\item Let $T$ be an $\L_m$ theory. A set $\Gamma$ of formulas is {\it isolated} if there is a formula $\psi$ consistent
with $T$ such that $T\models \psi\to \alpha$ for all $\alpha\in \Gamma$.
Otherwise, $\Gamma$ is {\it non-principal.}

\item Let $T$ be an $\L_m$ theory. A frame $\M=(\F, v)$ is a {\it model of $T$}
if $v:\omega\to F$, is such that for all $\phi\in T$, and $s\in F$, $\M, s\models_v \phi$.

\item Let $\bold L$ be a class of frames.
$\L_m$ has {\it $OTT$ with respect to $\bold L$}, if for every countable
theory $T$, for any non-principal type
$\Gamma$, there is a model $\M=(\F, v)\in \bold L$ of $T$
that omits $\Gamma$, in the sense that
for any $s\in F$, there exists $\phi\in \Gamma$
such that not $\M, s\models_v \phi$.
\end{enumarab}
\end{definition}
The proof of the  following lemma can be easily distilled from its relation 
algebra analogue. The statement of the lemma can 
be can seen as a weak soundness theorem truncating
syntactically  
the well know neat embedding theorem of Henkin 
to $n$ if $n<\omega$ and classical represenations to `$n$ localized ones.'
Otherwise if $n\geq \omega$ 
it  is the easy implication in the usual neat embedding theorem, for countable algebras 
having $\omega$ 
square or flat representations 
are representable in the classical sense. The converse also holds (using step by step constructions) 
so we have a weak {\it completeness} theorem, too, approximating the hard implication in the neat embedding theorem, 
but we do not need that much.

\begin{lemma}\label{flat}
Assume that $2<m<n$
\begin{enumarab}

\item If $\A$ has $n$ flat representation, then $\A\in S\Nr_m\TCA_n$,
\item If $\A$ has a complete $n$ flat representation, then $\A\in S_c\Nr_m\TCA_n.$

\end{enumarab}
We are now ready to formulate and prove 
the metalogical result that can be inferred from both \ref{can} and \ref{rainbow}, thus we give two proof.

The proofs also depend on the last lemma. 
\end{lemma}
\begin{theorem}\label{OTT}
For any $k\geq 3$, $\TL_m$ 
does not have $OTT$ with respect $\bold K_{\sf m, m+k, flat}$.

\end{theorem}

\begin{proof}
We give two proofs depending on our previous two rainbow constructions; one proving non atom canonicty of 
$S\Nr_n\TCA_{n+3}$ and the other priving non-elementarity of any class $\sf K$ containg the completely representable algebras and
contained in $S_c\Nr_n\TCA_{n+3}.$

{\bf First proof:}
Let $\A=\Tm\At$ be the term algebra as in \ref{can}. 
Then $\A$ is countable, simple and atomic. Let $\Gamma'$ be the set of  co-atoms.
Then $\A=\Fm_T$ for some countable consistent $L_m$ complete theory $T$.
Then $\Gamma=\{\phi: \phi_T\in \Gamma'\}$ is not principal.
If it can be omitted in an $m+3$ flat model $\M$, then
this gives an $m+3$ complete flat representation of $\A=\Fm_T$,
which gives
an $m+3$ flat representation of $\mathfrak{Cm}\At\A$. By lemma 
\ref{flat} we get that  $\Cm\At\A\in S\Nr_m\CA_{m+3}$ which contradicts theorem \ref{can}.

{\bf Second proof:} Let $\A$ be the term algebra of the rainbow atom structure of that 
of ${\sf TCA}_{\Z, \N}$  
constructed in theorem \ref{rainbow}. 
Then $\A\notin S_c\Nr_n\TCA_{m+3}$.
Let $T$ be such that $\A\cong \Fm_T$ and let $\Gamma$ be the set of co-atoms.
Then $\Gamma$ cannot be omitted in an $m+3$ flat $\M$,
for else, this gives a complete $m+3$ flat representation of $\A$, which means by lemma \ref{flat}
that $\A\in S_c\Nr_m\TCA_{m+3}$ and this contradicts theorem \ref{rainbow}. 

\end{proof}
The same result on failure of $OTT$ holds for the multi dimensional modal logics 
corresponding to both $\TeCA_m$ and $\sf TemCA_m$ for finite $m>2$.

By noting that \pa\ can win the game $G^{n+3}_{\omega}$ in theorem \ref{can} which is the 
usual $\omega$ rounded atomic game restricted to  $n+3$ nodes (pebbles), though $n+3$ round suffice for him to 
force \pe\ an inconsistent red, this excludes ever  $n+3$ square models omitting types.

This class is substantially larger; for example the first order theory of its 
Kripke  frames can be coded in the packed fragment, which is not the case with  
flat model, basicaly  due to the additional Church Rosser condition of commutativity
of cylindrifiers. 

In fact, it can be shown that for $m=3$ and $k\geq 3$,
it is undecidable to tell whether a finite frame $\F$ has its modal algebra 
$S\Nr_3\TCA_{3+k}$,  that is whether $\F$ is a frame for $\TL_3$ with a proof system, more specifically a Hilbert style axiomatization,
using $3+k$ variables; corresponding to  the equational axiomatization of $\TCA$ where formulas are built up 
from $3$ variable restricted 
atomic formulas, that is, variables occur in the order $x_0, x_1,x_2$, i.e in their natural order,  
applying usual connectives of first order logic and 
$\Box_i$, $i<3$. 

This follows from the fact that $S\sf Ra\CA_5$ 
has the same property and this property lifts to $\sf TCA_3$s,
by discretely topologizing a well known construction of Monk's 
associating to every 
simple atomic $\CA_3$ a simple atomic relation algeba such that the relation algebra embeds into the 
relation algebra reduct  of the constructed cylindriic algebras $\sf CA_3$.

There are finite algebras that have infinite 
$n$ flat representations, for any $m<n$ but not finite ones.
This is not the casse with square models. 

Any 
finite algebra  that has an $n>m$  square model has a finite one because 
the packed fragment of finite variable first order logic has the finite model property.
However, the equational theory of its modal algebras in undecidable.

\subsection{Various notions of representabilty}

Here we restrict ourselves to 
Alexandov topologies stimulated by pre-order on the base of Kripke frames and 
we deal with diamonds that are dual to boxes 
(the interior operators).

The definition of atomic networks on atomoc algebras 
are obtained  from those for  $\CA$s by adding a further consistency condition, namely, 
$N(\delta^i_d)\leq \Diamond_i N(\delta)$, when $d\leq \delta(i)$.

Now we introduce quite an interesting class that 
lies strictly between $\sf SRTCA_n$ and the 
class of atomic completely additive representable algebras in $\sf RTCA_n$ 
which we call weakly representable algebras and denote by $\sf TWCA_n$. 

The introduction of this new elementary class is  motivated by its relation algebra anolgue proposed by Goldblatt and further invesitigated
by Hirsch and Hodkinson.
For a cylindric algebra atom structure $\F$ the first order algebra over $\F$ is the subalgebra
of $\Cm\F$ consisting of all sets of atoms that are first order
definable with parameters from $S$. 
${\sf TFOCA_n}$ denotes the class of atomic such algebras of dimension $n$.

This class is strictly bigger than ${\sf SRTCA_n}$.
Indeed, let $\A$ be any the  rainbow term algebra, constructed in the proof of theorem \ref{can};
recall that such an algebra is  obtained by blowing up and blurring the  finite rainbow algebra
$T\CA_{n+1, n}$ proving that
$S\Nr_{n}\TCA_{n+3}$ is not atom canonical.

This algebra was defined using first order formulas in the rainbow signature
(the latter is first order since we had only
finitely many greens).  Though the usual semantics was perturbed, first order logic did not see the relativization, only
infinitary formulas saw it,
and thats why the complex algebra could not be represented.
Other simpler examples, that can be topologized, are given in \cite{weak, Hodkinson}.

These examples all show that ${\sf SRTCA_n}$ is properly contained in ${\sf TFOCA_n}.$
Another way to view this is to notice that ${\sf TFOCA_n}$ is elementary, by definition, 
that ${\sf STRCA_n}\subseteq {\sf TFOCA_n}$, but
but as mentioned above one can obtain by discretly topologizing  Monk-like algebras, denoted by $\M(\Gamma)$ in \cite{HHbook2}
based on Erdos' graphs 
that ${\sf STRCA_n}$ is not elementary as proved by Hirsh and Hodkinson for $\CA$s.

Fix finite $n>2$. 
Let $\sf CRTCA_n$ denote the non elementary class of completely representable algebras, $\sf TLCA_n$
denotes the class satifying 
the {\it diamond  Lyndon conditions}. Here, however, 
the  $m$th Lyndon condition is a first order sentence that codes 
that \pe\ has a \ws\ in the $m$ rounded 
atomic game  defined like the $\CA$ case, but here we have the  extra diamond move, namely,

\pa\ chooses $i<n$, $a\in \At\A$, a tuple $\bar{x}$ and a previously played
atomic network $N$
such that $N(\bar{x})\leq \Diamond_ia$. Now \pe\ has to refine $N$ by a network $M$ 
and $\bar{y}\in M$ possible xtending the pre-order $\leq$ to $\leq_i$ 
with $\bar{x}\equiv_i \bar{y}$ and $x_i\leq_i y_i$ such that 
$M(\bar{y})=a$. 

The Lyndon conditions characterizes representability for finite algebras but if 
an infinite algebra satisfies Lyndon conditions
then all we can conclude is that it is elementary equivalent
to a completely representable algebra. 

Using  this observation 
it can be easily shown
that the elementary closure of $\sf CRTCA_n$, $EL\sf CRTRA_n$ for short,  coincides with $\sf TLCA_n$. 
Recall that the former class is not elementary by theorem \ref{rainbow}, while $\sf TLCA_n$ is elementary by definition.

$\sf TWRCA_n$ 
is the class of weakly representable algebras defined simply 
to be the  class consisting 
of atomic algebras in $\sf RTCA_n$.  This class is obviously elementary because atomicity is a first order definable property.
and $\sf RTCA_n$ is a variety.

It can be proved using Monk-like algebras that the elementary closure of all such classes is not finitely 
axiomatizable \ref{complete} and  an that $El\sf CRTCA_n=\sf TLCA_n$. 

Let t $\M(\Gamma)$ be the the discretely topologized Monk like algebras constructed in \cite{HHbook2} based on $\Gamma$.

For an undirected graph $\Gamma$ let $\chi(\Gamma)$ denotes its chromatic number
which is 
the size of the smallest finite set $C$ such that there exist a $C$ colouring of $\Gamma$ 
and $\infty$ otherwise. $C$ is a colouring of $\Gamma$ 
if there exits $f:\C\to \Gamma$ such that whenever $(v, w)$ is an edge 
then $f(v)\neq f(w)$.

Recall that $\M(\Gamma)$ is {\it good} if $\chi(\Gamma)=\infty$ this is equivalent to its representability, 
othewise, that is when $\chi(\Gamma)<\infty$, 
it is {\it bad}. This means that $\Gamma$ has a finite colouring prohibiting
a representation for $\M(\Gamma)$.

To show that $\sf SRTCA_n$ is not finitely axiomatizable one constructs a sequence 
of algebras that are weakly but not strongly representable with an ultraproduct
that is in $El(\sf SRTCA_n)$; which is a plausible task and not very hard to accomplish.
Roughly one constructs Monk-like  algebras based on graphs with arbirary large chromatic number converging to one 
with infinite chromatic number (via an ultraproduct construction). 

Monk's original algebras can be seen this way.
Indeed it not hard to show that the limit of  Monk's original finite 
algebars \cite{HMT2} is strongly representable; by observing first that it is atomic because it is an ultarproduct of atomic algebras.
In fact, the limit is completely representable.
Now we have the strict inclusions, that can be proverd the topological coiunterpat of the arguments in  \cite{HHbook2}:

$$\sf CRTCA_n \subseteq \sf TLCA_n \subseteq \sf SRCTA_n \subseteq \sf TFOCA_n \subseteq \sf WRTCA_n.$$

Strictly speaking, \cite{HHbook2} did 
not deal with $\sf TFOCA_n$, but the relation algebra analogue of 
such a class is investigated in 
\cite{HHbook}, and the thereby obtained results lifts to the $\CA$ case without much ado. 

Last inclusion follows from the following  $\sf RA$ to $\CA$ adaptation of an example of 
Hirsch and Hodkinson which we use to show that
${\sf TFOCA_n}\subset  {\sf WRCA_n}$ and it will  be used for another purpose as well. 
We note that the strictness of the inlusion is 
not so obvious because they are both elementary.

\begin{example}\label{fo}
Take an $\omega$ copy of the  $n$ element graph with nodes $\{1,2,\ldots, n\}$ and edges
$1\to 2\to\ldots\to n$. Then of course $\chi(\Gamma)<\infty$. Now  $\Gamma$  has an $n$ first order definable colouring.
Since $\M(\Gamma)$ as defined above and in \cite[top of p. 78]{HHbook2} is not representable, then the algebra of first order
definable sets, call it $\A$, is also not representable because $\Gamma$ is first order interpretable in
$\rho(\Gamma)$, the atom structure constructed from $\Gamma$ as defined in \cite{HHbook2}.
However, it can be shown that the term algebra is representable. (This is not so easy to prove).

Now $\At\Tm\A =\At\A$ and $\A$ is not representable, least strongly representable.
But since $\sf SRCA_n\subseteq \sf FOCA_n$, then 
$\At(\sf SRTCA_n)$ grips $\sf TFOCA_n$ 
but it does not grip  
$\sf WRTCA_n$, for $\A$ is not strongly representable.
\end{example}

Foe a class of algebras $\K$ having a Boolean reduct $\K\cap \At$ denotes the class of atomic algebras in $\K$. 
\begin{theorem}
Let $n>2$ be finite. Then we have the following inclusions (note that $\At$ commutes with ${\bf UpUr})$:
$$\Nr_n\TCA_{\omega}\cap \At\subset El\Nr_n\TCA_{\omega}\cap \At$$
$$\subset {El}\bold S_c\Nr_n\TCA_{\omega}\cap \At={El}{\sf CRTCA_n}={\sf TLCA}_n\subset {\sf SRTCA_n}$$
$$\subset
{Up}{\sf SRTCA_n}={Ur}{\sf  STRCA_n}={El}{\sf STRRCA}_n\subseteq {\sf TFOCA_n}$$
$$\subset S\Nr_n\TCA_{\omega}\cap \At={\sf WRTCA_n}= \RCA_n\cap \At.$$
\end{theorem}
\begin{proof}
The majority of  inclusions, and indeed their strictness,
can be distilled without much difficulty from our previous work.
It is known that ${\bf Up}{\sf STRCA_n}={\bf Ur}{\sf  STRCA_n}$ \cite{strong, HHbook}. 

The first  inclusion is witnessed by a slight modification
of the algebra $\B$ used in the proof of \cite[Theorem 5.1.4]{Sayedneat},
showing that for any pair of ordinals $1<n<m\cap \omega$, the class $\Nr_n\CA_m$ is not elementary.

In the constructed model $\sf M$ \cite[lemma 5.1.3]{Sayedneat} on which (using the notation in {\it op.cit}),
the two algebras $\A$ and $\B$  are based,
one requires (the stronger) that the interpretation of the $3$ ary relations symbols in the signature
in $\sf M$ are {\it disjoint} not only distinct as above.
Atomicity of $\B$ follows immediately,  since its Boolean reduct  is now a product of atomic algebras.
For $n=3$ these are denoted by $\A_u$ except for one countable component $\B_{Id}$, $u\in {}^33\sim \{Id\}$, cf. \cite{Sayedneat}
p.113-114. Second inclusion follows from example \ref{SL}. Third follows from theorem \ref{can},  
and fourth  inclusion follows from theorem \ref{rainbow}. 
Last one follows from example \ref{fo}.
\end{proof}

We characterize the class of strongly representable atom structures via neat embeddings, modulo an
{\it inverse} of Erdos' theorem.

For an atomic algebra $\A$, by an atomic subalgebra we mean a subalgebra of $\A$ containing all its atoms, equivalently a superalgebra of
$\Tm\At\A$.
We write  $S_{at}$ to denote this operation applied to an algebra or to a
class of algebras. $\M(\Gamma)$ denotes the Monk algebra based on the graph
$\Gamma$ as in \cite{HHbook2}.
Notice that although we have ${\sf UpUr}{\sf STRCA}_n\subseteq {\sf TFOCA_n}$, the latter is not closed under forming atomic subalgebras, since
forming subalgebras does not preserve first order sentence.
The next characterization therefore seems to be plausible.
Hoever it is formulaed only for cyindric algebras. The notion is obtained by removing the ' $T$ ' to get the corresponding class of cylindic algebra,
for example $\sf SRCA_n=\Rd_{ca}\sf SRTCA_n$, because ever $\CA$ 
can be expanded with the interior idenity operations. The other inclusion is problematic for 
there are algebras in $TRCA_n$ that are not completely additive.

\begin{theorem} Assume that for every atomic representable algebra that is not strongly representable, there exists 
a graph $\Gamma$ with finite
chromatic number such that $\A\subseteq \M(\Gamma)$ and $\At\A=\rho(\M(\Gamma))$.
Assume also that for every graph $\Gamma$ with $\chi(\Gamma)<\infty$, there exists
$\Gamma_i$ with  $i\in \omega$, such that $\prod_{i\in F}\Gamma_i=\Gamma$, for some non principal ultrafilter $F$.
Then $S_{at}{\sf Up}{\sf SRCA}_n={\sf WRCA_n}=S\Nr_n\CA_{\omega}\cap \At.$
\end{theorem}
\begin{proof}
Assume that $\A$ is atomic, representable but not strongly representable. Let $\Gamma$ be a graph with $\chi(\Gamma)<\infty$ such that
$\A\subseteq \M(\Gamma)$ and $\At\A=\rho(\M(\Gamma))$.
Let $\Gamma_i$ be a sequence of graphs each with infinite chromatic number converging to $\Gamma$, that is, their ultraproduct
is $\Gamma$
Let $\A_i=\M(\Gamma_i)$. Then $\A_i\in {\sf SRSA_n}$, and we have:
$$\Pi_{i\in \omega}\M(\Gamma_i)=\M(\Pi_{i\in \omega} \Gamma_i)=\M(\Gamma).$$
And so $\A\subseteq_{at} \prod_{i\in \omega}\A_i$, and we are done.
\end{proof}

\subsection{Neat embeddings}

For the well known definitions of pseudo universal and pseudo elementary classes,
the reader is referred  to \cite[definition 9.5, definition 9,6]{HHbook2}.
In fact all our results below hold for any class whose signature is between
$\Sc$ and $\PEA$. (Here $\sf Df$ is not counted in
because the notion of neat reducts for this class is trivial \cite[Theorem 5.1.31]{HMT2}).

One can topologize rainbow atom structures by
defining for bijections $f, g:n \to \Gamma$, $\Gamma$ a coloured graph, and for $i<n$
$[f] In_i[g]$ iff $f=g$.

\begin{theorem}\label{maintheorem}
\begin{enumarab}

\item For $n>2$, the inclusions $\Nr_n\TCA_{\omega}\subseteq S_c\Nr_n\TCA_{\omega}\subseteq S\Nr_n\TCA_{\omega}$ are proper.
The first strict inclusion can be witnessed by a finite algebra for $n=3$, while the second cannot by witnessed by a finite
algebra for any $n$.
In fact, $m>n>1$, the inclusion $\Nr_n\TCA_m\subseteq S_c\Nr_n\TCA_m$ is proper
and for $n>2$ and  $m\geq n+3$ the inclusion $S_c\Nr_n\TCA_m\subseteq S\Nr_n\TCA_m$ is also proper.

\item For any pair of ordinals $1<\alpha<\beta$ the class $\Nr_{\alpha}\TCA_{\beta}$ is not elementary.
In fact, there exists an uncountable atomic algebra $\A\in \Nr_{\alpha}\QEA_{\alpha+\omega}$, hence $\A\in \Nr_{\alpha}\QEA_{\beta}$
for every $\beta>\alpha$, and $\B\subseteq_c \A$,
such that $\B$ is completely representable, $\A\equiv \B$,  so that $\B$ is also atomic,
but $\Rd_{sc}\B\notin \Nr_{\alpha}\Sc_{\alpha+1}$.
For finite dimensions, we have $\At\A\equiv_{\infty} \At\B$.

\item For finite $n$, the elementary theory of $\Nr_n\TCA_{\omega}$ is recursively enumerable.

\item For $n>1$, the class ${\sf S}_c\Nr_n\TCA_{\omega}$ is not elementary (hence not pseudo-universal)
but it is pseudo-elementary, and the  elementary theory of $S_c\Nr_n\TCA_{\omega}$ is  recursively enumerable.

\item For $n>1$, the class ${\sf S}_c\Nr_n\TCA_{\omega}$ is closed under forming
strong subalgebras but is
not closed under forming subalgebras

\item For $n>2$, the class ${\sf UpUr}{\sf S}_c\Nr_n\TCA_{\omega}$ is not finitely axiomatizable. For any $m\geq n+2,$
$S\Nr_n\TCA_m$
is not finitely axiomatizable and for $p\geq 2$, $\alpha>2$ (infinite included) and $\TCA$
the variety $S\Nr_{\alpha}\TCA_{\alpha+2}$ cannot be finitely axiomatized by a universal set of
formulas containing only finitely many variables.

\item For $n>2$, both ${\sf UpUr}\Nr_n\TCA_{\omega}$ and ${\sf UpUr}S_c\Nr_n\TCA_{\omega}$ are properly contained in $\sf RK_n$;
by closing under forming subalgebras both
resulting classes coincide with $\sf RK_n.$
Furthermore, the former is properly contained in the latter.
\end{enumarab}
\end{theorem}

\begin{proof}

\begin{enumarab}

\item The  inclusions are obvious.

The strictness of the inclusion follows from item example \ref{SL} since
$\A$ is dense in $\wp(V)$, hence it is in $S_c\Nr_n\TCA_{\omega}$, but it is not in $\Nr_n\TCA_{n+1}$, {\it a fortiori} it is not
in $\Nr_n\TCA_{\omega}$ for the latter is clearly contained in the former.

For the strictness of the second  inclusion, let $\A$ be the rainbow algebra $\PEA_{\Z, \N}$. Then $\PEA_{\Z, \N}$ is representable,
in fact, it satisfies the Lyndon conditions, hence is also strongly representable, but it is  {\it not} completely
representable, in fact its $\sf Df$ reduct is not completely representable, and its $\Sc$ reduct  is not
in $S_c\Nr_n\Sc_{\omega}$, for had it been in this class, then it would be completely representable, inducing a complete
representation of $\A$.

We have proved that there is a  representable algebra, namely, $\PEA_{\Z, \N}$
such that  $\Rd_{Sc}\PEA_{\Z, \N}\notin S_c\Nr_n\Sc_{\omega}$. The required now follows
by using the neat embedding theorem that says that ${\sf RTCA}_n=S\Nr_n\TCA_{\omega}$
for any $\TCA$ as specified above \cite{HMT1}, indeed we have ${\Rd_K}{\sf PEA}_{\Z, \N}$ is not in $S_c\Nr_n\TCA_{\omega}$ but it is
(strongly) representable, that is, it is  in $S\Nr_n\TCA_{\omega}$.
On the other hand, its canonical extension
is completely representable a 
result of Monk \cite[Corollary 2.7.24]{HMT2}.

For the second part concerning finite $3$ dimensional algebras witnessing the strictness of inclusions by above example. 

The strictness of the last inclusion cannot be witnessed by a finite representable
algebra for any finite dimension $n>2$, because any such algebra, is atomic (of course)
and completely representable, hence it will be necessarily in $S\Nr_n\TCA_{\omega}$.

Concerning the second strictness of inclusions (concerning neat embeddability in finitely many extra dimensions),
the first is witnessed by $\A$ constructed in \ref{SL} or the $\B$ constructed in the next item,
while the second follows by noting  $S_c\Nr_n\TCA_m$ is not elementary while $S\Nr_n\TCA_m$ is a
variety theorem \ref{rainbow}.

\item We consider the case when $\A$ and $\B$ are cylindric algebras as constructed in \cite[Theorem 5.1.3]{Sayedneat}
but discretly topologized. We show that $\A\equiv_{\infty}\B$.
A finite  atom structure,
namely, an $n$ dimensional cartesian square,  with accessibility relations corresponding to the concrete interpretations of
cylindrifiers and diagonal elements, is fixed in advance.

Then its atoms are
split twice.  Once, each atom is split into uncountably many, and once each into uncountably many except for one atom which
is only split  into {\it countably} many atoms. These atoms are called big atoms, which mean that they are cylindrically equivalent to their
original. This is a general theme in splitting arguments.
The first splitting gives an algebra $\A$ that is a full neat reduct of an algebra in arbitrary extra dimensions;
the second gives an algebra $\B$ that is not a full neat reduct
of an algebra in just one extra dimensions, hence in any higher
dimensions. Both algebras are representable, and elementary equivalent
because first order logic cannot see this cardinality twist.
We will show in a minute that their atom structures are $L_{\infty, \omega}$ equivalent.

We, henceforth, work with dimension $3$. The proof for higher finite dimensions
is the same. However, we make a slight perturbation
to the construction in {\it op.cit}; we require that the interpretation of the uncountably ternary
relation symbols in the signature of $\sf M$ on which the set algebras $\A$ and $\B$ are based
are {\it disjoint}, not only distinct \cite[Theorems 5.3.1. 5.3.2]{Sayedneat}.

The Boolean reduct of $\A$ can be viewed as a finite direct product of disjoint Boolean relativizations of $\A$,
denoted in \cite[theorem 5.3.2]{Sayedneat} by $\A_u$; $\A_u$ is the finite-cofinite algebra on a set having the same cardinality as the signature;
it is relativized to $1_u$ as defined in {\it opcit}, $u\in {}^33$.

Each component will be atomic by our further restriction on $\sf M$,
so that $\A$ itself, a product of atomic algebras is also atomic. The language of Boolean algebras can now be expanded
so that $\A$ is interpretable in an expanded structure $\P$,
based on the  same atomic Boolean product. Now $\B$ can be viewed as obtained from $\P$,
by replacing one of the components of the product with an elementary
{\it countable} Boolean subalgebra, and then giving it the same interpretation.
By the Feferman Vaught theorem (which says that replacing in a product one of its components by an elementary
equivalent one, the resulting product remains elementary equivalent to the original product) we have $\B\equiv \A$.
In particular, $\B$ is also atomic.

First order logic will not see this cardinality twist, but a suitably chosen term
not term definable in the language of
$\CA_3$, namely, the substitution operator, $_3{\sf s}(0,1)$  will
witnessing that the twisted algebra $\B$ is not a neat reduct.

The Boolean structure of both algebras are in fact very simple. For a set $X$, let  ${\sf Cof} (X)$ denote
the finite co-finite Boolean algebra on $X$, that is ${\sf Cof}(X)$ has universe $\{a\in \wp(X): |a|<\omega, \text { or } |X\sim a|<\omega\}$.
Let $J$ be any set having the same cardinality as the signature of $\sf M$
so that $J$ can simply be ${\sf M}$.
Then $\A\cong \prod_{u\in {}^33}\A_u$, where $\A_u={\sf Cof}(J)$
and $\B=\prod_{u\in {}^33}\B_u$, where $\B_{Id}={\sf Cof(\sf N)}$ ($\sf N$ can in fact
be any set such that  $|\sf N|<|J|$ but for definiteness let it be the least infinite cardinal)
and otherwise $\B_u=\A_u$.

We show that \pe\ has a \ws\ in an \ef-game over $(\A, \B).$
At any stage of the game, if \pa\ places a pebble on one of
$\A$ or $\B$, \pe\ must place a matching pebble on the other
algebra.  Let $\b a = \la{a_0, a_1, \ldots, a_{n-1}}$ be the position
of the pebbles played so far (by either player) on $\A$ and let $\b b = \la{b_0, \ldots, b_{n-1}}$ be the the position of the pebbles played
on $\B$.  \pe\ maintains the following properties throughout the
game.
\begin{itemize}
\item For any atom $x$ (of either algebra) with
$x. 1_{Id}=0$ then $x \in a_i\iff x\in b_i$.
\item $\b a$ induces a finite partion of $1_{Id}$ in $\A$ of $2^n$
 (possibly empty) parts $p_i:i<2^n$ and the $\b b$
induces a partion of  $1_{Id}$ in $\B$ of parts $q_i:i<2^n$.  $p_i$ is finite iff $q_i$ is
 finite and, in this case, $|p_i|=|q_i|$.
\end{itemize}
It is easy to see that \pe\ can maintain these properties in every round.
Therefore she can win the game.  Therefore $\A\equiv_{\infty}\B$.

\item We  show that $\Nr_n\TCA_{\omega}$ is pseudo-elementary. 
This is similar to the proof of \cite[theorem 21]{r} using a three sorted first order theory.
from which we can infer the elementary theory
$\Nr_n\CA_{\omega}$ is recursively  enumerable for any finite $n$.

To show that $\Nr_n\TCA_{\omega}$  i
s pseudo-elementary, we use a three sorted defining theory, with one sort for a toplogical cylindric algebra of dimension $n$
$(c)$, the second sort for the Boolean reduct of a cylindric algebra $(b)$
and the third sort for a set of dimensions $(\delta)$; the argument is analogous to that of Hirsch used for relation algebra reducts \cite[theorem 21]{r}.
We use superscripts $n,b,\delta$ for variables
and functions to indicate that the variable, or the returned value of the function,
is of the sort of the cylindric algebra of dimension $n$, the Boolean part of the cylindric algebra or the dimension set, respectively.
We do it for $\CA$s.  The other cases can be dealt with in exactly the same way.

The signature includes dimension sort constants $i^{\delta}$ for each $i<\omega$ to represent the dimensions.
The defining theory for $\Nr_n{\sf TCA}_{\omega}$ includes sentences stipulating
that the constants $i^{\delta}$ for $i<\omega$
are distinct and that the last two sorts define
a cylindric algebra of dimension $\omega$. For example the sentence
$$\forall x^{\delta}, y^{\delta}, z^{\delta}(d^b(x^{\delta}, y^{\delta})=c^b(z^{\delta}, d^b(x^{\delta}, z^{\delta}). d^{b}(z^{\delta}, y^{\delta})))$$
represents the cylindric algebra axiom ${\sf d}_{ij}={\sf c}_k({\sf d}_{ik}.{\sf d}_{kj})$ for all $i,j,k<\omega$.
We have have a function $I^b$ from sort $c$ to sort $b$ and sentences requiring that $I^b$ be injective and to respect the $n$ dimensional
cylindric operations as follows: for all $x^r$:
\begin{align*}
I^b({\sf d}_{ij})&=d^b(i^{\delta}, j^{\delta}),\\
I^b({\sf c}_i x^r)&= {\sf c}_i^b(I^b(x^r)),\\
I^b(\Diamond_i x^r)&=\Diamond_i^b(I^b(x^r)).
\end{align*}

Finally we require that $I^b$ maps onto the set of $n$ dimensional elements
$$\forall y^b((\forall z^{\delta}(z^{\delta}\neq 0^{\delta},\ldots (n-1)^{\delta}\rightarrow c^b(z^{\delta}, y^b)=y^b))\leftrightarrow \exists x^r(y^b=I^b(x^r))).$$

In all cases, it is clear that any algebra of the right type is the first sort of a model of this theory.
Conversely, a model for this theory will consist of an $n$ dimensional cylindric algebra type (sort c),
and a cylindric algebra whose dimension is the cardinality of
the $\delta$-sorted elements, which is at least $|m|$.
Thus the three sorted theory defines the class of neat reduct, furthermore, it is clearly recursive.

Finally, if $\TCA$ be a pseudo elementary class, that is
$\K=\{M^a|L: M\models U\}$ of $L$ structures, and $L, L^s, U$ are recursive.
Then there a set of first order recursive theory  $T$ in $L$,
so that for any $\A$ an $L$ structure, we have
$\A\models T$ iff there is a $\B\in \K$ with $\A\equiv \B$. In other words,
$T$ axiomatizes the closure of $\K$ under elementary equivalence, see
\cite[theorem 9.37]{HHbook} for unexplained notation and proof.

\item That $S_c\Nr_n\TCA_{\omega}$ is not elementary follows from theorem 
\ref{rainbow}.
It is pseudo-elementary because of the following reasoning.
For brevity, let ${\sf L}= S_c\Nr_n\TCA_{\omega}\cap \At$, and let $\sf  CRK_n$ denote the class of completely representable algebras.
Let $T$ be the first order theory that axiomatizes ${\sf Up Ur CRK}_n={\sf LCK_n}$.
It suffices to show, since $\sf CRK_n$ is pseudo-elementary \cite{HHbook}, that
$T$ axiomatizes ${\sf UpUr  L}$ as well.
First, note that $\sf CRA_n\subseteq \sf L$.
Next assume that $\A\in {\sf UpUr L}$, then $\A$ has a countable elementary (necessarily atomic)
subalgebra in $\sf L$  which is completely representable by the above argument, and we
are done.

\item We show that $S_c\Nr_n\TCA_{\omega}$ is not closed under forming subalgebras, hence it is not pseudo-universal.
That it is closed under $S_c$ follows directly from the definition.
Consider the  $\Sc$ reduct of either ${\sf PEA}_{\Z, \N}$ or $\PEA_{\omega, \omega}$ or $\PEA_{\TCA,_{\omega},\TCA}$ of the previous item.
Fix one of them, call it $\A$. Because $\A$ has countably many atoms, $\Tm\A\subseteq \A\subseteq \Cm\At\A$,
and all three are completely representable or all three not completely representable sharing the same atom structure $\At\A$,
we can assume without loss that $\A$ is countable.
Now $\A$ is not  completely representable, hence by the above 'omitting types argument' in item (1),
its $\Sc$ reduct is not in $S_c\Nr_n\Sc_{\omega}$.

On the other hand, $\A$ is strongly
representable, so its canonical extension
is representable, indeed completely representable, hence as claimed
this class is not closed under
forming subalgebras, because the $\Sc$  reduct of the canonical extension of $\A$
is in $S_c\Nr_n\sf Sc_{\omega}$, $\Rd_{sc}\A$ is not in $S_c\Nr_n\Sc_{\omega}$ and $\A$ embeds into its canonical
extension. The first of these statements follow from the fact that if $\D\subseteq \Nr_n\B$,
then $\D^+\subseteq \Nr_n\B^+$.

\item Witness theorem \ref{complete}.
The rest is known.

\item The algebra $\A$ in item (1) in theorem \ref{SL}, witnesses the strictness of the stated  last inclusion.
The strictness of second inclusion from the fact that the class in question coincides
with the class of algebras satisfying the Lyndon conditions, and we have proved in
that this class this is properly contained in even the class of  (strongly) representable algebras.

Indeed,  let $\Gamma$ be any graph with infinite chromatic number,
and large enough finite girth. Let $\rho_k$ be the $k$ the Lyndon condition for $\Df$s.
Let $m$ be also large enough so that any $3$ colouring of the edges of a complete graph
of size $m$ must contain a monochromatic triangle; this $m$ exists by Ramsey's theorem.
Then $\M(\Gamma)$, the complex algebra constructed on $\Gamma$, as defined in \cite[definition sec 6.3,  p.78]{HHbook}
will be representable as a polyadic equality algebra but it will fail $\rho_k$ for all $k\geq m$. The idea is
that \pa\ can win in the $m$ rounded atomic game coded by $\sigma_m$, by forcing a forbidden monochromatic triangle.

The last required from from the fact that the resulting classes coincide with $S\Nr_n\TCA_{\omega}$
which, in turn,  coincides with the class of representable algebras by the
neat embedding theorem of Henkin.
\end{enumarab}

\end{proof}
\subsection{Neat embeddings for infinite dimensional algebras}

\begin{theorem}\label{infinite} Let $\sf K$ be any class between $\Sc$ and $\PEA$. 
Let $\alpha$ be an infinite ordinal. Then the following hold.
\begin{enumarab}

\item Assume that for any $r\in \omega$ and $3\leq m\leq n<\omega$, there is 
an algebra $\C(m,n,r)\in \Nr_m{\sf PEA}_n,$ with $\Rd_{sc}\C(m,n,r)\notin {\sf S}\Nr_m{\sf Sc_{n+1}}$
and $\Pi_{r/U}\C(m,n,r)\in {\sf RPEA}_m.$ Furthermore, assume that if $3\leq m<n$, $k\geq 1$ is finite,  
and $r\in \omega$, there exists $x_n\in \C(n,n+k,r)$
such that $\C(m,m+k,r)\cong \Rl_{x_n}\C(n, n+k, r)$ and ${\sf c}_ix_n\cdot {\sf c}_jx_n=x_n$
for all $i,j<m$. Then  for any any $r\in \omega$, for any
finite $k\geq 1$, for any $l\geq k+1$ (possibly infinite),
there exist $\B^{r}\in {\sf S}\Nr_{\alpha}\QEA_{\alpha+k}$, $\Rd_{sc}\B^r\notin {\sf S}\Nr_{\alpha}\Sc_{\alpha+k+1}$ such
$\Pi_{r\in \omega}\B^r\in {\sf S}\Nr_{\alpha}\QEA_{\alpha+l}$.
In partcular, the same result holds for topological cylidric algebras.

\item  Let $k\geq 1$. Assume that for each finite $m\geq 3$, there exists $\C(m)\in {\sf TCA}_m$ 
such that $\Rd_{sc}\C(m)\notin {\sf S}_c\Nr_m\Sc_{m+k}$ and $\A(m)\in \Nr_m\TCA_{\omega}$
such that $\A(m)\equiv \C(m).$ Furthermore, assume that for $3\leq m<n<\omega$, $\C(m)\subseteq _c\Rd_m\C(n)$.
Then any class $\L$ between $\Nr_{\alpha}\K_{\omega+\omega}$ and ${\sf S}_c\Nr_{\alpha}\K_{\alpha+k}$ is not elementary

\item Assume that for each finite $m\geq 3$, there exists $\C(m)\in \Nr_{k}\TCA_{\omega}$ whose $\Df$ 
reduct is not completely representable
such that for $3\leq m<n<\omega$, there exists $x_m\in \C(n)$ such that ${\sf c}_jx_m{\sf c}_jx_m=x_k$ for all $j<m$ 
and $\C(k)\cong \Rd_k\Rl_x\C(m)$. Then there exists an algebra $\A$ in 
and $\Nr_{\alpha}\TCA_{\alpha+\omega}$ 
whose $\Sc$ reduct is not completely representable.
\end{enumarab}
\end{theorem}

\begin{proof}

\begin{enumarab}

\item Fix such $r$.
Let $I=\{\Gamma: \Gamma\subseteq \alpha,  |\Gamma|<\omega\}$.
For each $\Gamma\in I$, let $M_{\Gamma}=\{\Delta\in I: \Gamma\subseteq \Delta\}$,
and let $F$ be an ultrafilter on $I$ such that $\forall\Gamma\in I,\; M_{\Gamma}\in F$.
For each $\Gamma\in I$, let $\rho_{\Gamma}$
be a one to one function from $|\Gamma|$ onto $\Gamma.$

Let ${\C}_{\Gamma}^r$ be an algebra similar to $\QEA_{\alpha}$ such that
\[\Rd^{\rho_\Gamma}{\C}_{\Gamma}^r={\C}(|\Gamma|, |\Gamma|+k,r).\]
Let
\[\B^r=\Pi_{\Gamma/F\in I}\C_{\Gamma}^r.\]
We will prove that
\begin{enumerate}
\item\label{en:1} $\B^r\in S\Nr_\alpha\QEA_{\alpha+k}$ and
\item\label{en:2} $\Rd_{sc}\B^r\not\in S\Nr_\alpha\Sc_{\alpha+k+1}$.
\end{enumerate}

For the first part, for each $\Gamma\in I$ we know that $\C(|\Gamma|+k, |\Gamma|+k, r) \in\K_{|\Gamma|+k}$ and
$\Nr_{|\Gamma|}\C(|\Gamma|+k, |\Gamma|+k, r)\cong\C(|\Gamma|, |\Gamma|+k, r)$.
Let $\sigma_{\Gamma}$ be a one to one function
 $(|\Gamma|+k)\rightarrow(\alpha+k)$ such that $\rho_{\Gamma}\subseteq \sigma_{\Gamma}$
and $\sigma_{\Gamma}(|\Gamma|+i)=\alpha+i$ for every $i<k$. Let $\A_{\Gamma}$ be an algebra similar to a
$\CA_{\alpha+k}$ such that
$\Rd^{\sigma_\Gamma}\A_{\Gamma}=\C(|\Gamma|+k, |\Gamma|+k, r)$.
We claim that
 $\Pi_{\Gamma/F}\A_{\Gamma}\in \QEA_{\alpha+k}$.

For this it suffices to prove that each of the defining axioms for $\QEA_{\alpha+k}$
hold for $\Pi_{\Gamma/F}\A_\Gamma$.
Let $\sigma=\tau$ be one of the defining equations for $\QEA_{\alpha+k}$,
and we assume to simplify notation that
the number of dimension variables is one. Let $i\in \alpha+k$, we must prove
that $\Pi_{\Gamma/F}\A_\Gamma\models \sigma(i)=\tau(i)$.  If $i\in\rng(\rho_\Gamma)$,
say $i=\rho_\Gamma(i_0)$,
then $\Rd^{\rho_\Gamma}\A_\Gamma\models \sigma(i_0)=\tau(i_0)$,
since $\Rd^{\rho_\Gamma}\A_\Gamma\in\QEA_{|\Gamma|+k}$,
so $\A_\Gamma\models\sigma(i)=\tau(i)$.
Hence $\set{\Gamma\in I:\A_\Gamma\models\sigma(i)=\tau(i)}\supseteq\set{\Gamma\in I: i\in\rng(\rho_\Gamma)}\in F$,
hence $\Pi_{\Gamma/F}\A_\Gamma\models\sigma(i)=\tau(i)$.
Thus, as claimed, we have $\Pi_{\Gamma/F}\A_\Gamma\in\QEA_{\alpha+k}$.

We prove that $\B^r\subseteq \Nr_\alpha\Pi_{\Gamma/F}\A_\Gamma$.  Recall that $\B^r=\Pi_{\Gamma/F}\C^r_\Gamma$ and note
that $\C^r_{\Gamma}\subseteq A_{\Gamma}$
(the universe of $\C^r_\Gamma$ is $C(|\Gamma|, |\Gamma|+k, r)$, the universe of $\A_\Gamma$ is $C(|\Gamma|+k, |\Gamma|+k, r)$).
So, for each $\Gamma\in I$,
\begin{align*}
\Rd^{\rho_{\Gamma}}\C_{\Gamma}^r&=\C((|\Gamma|, |\Gamma|+k, r)\\
&\cong\Nr_{|\Gamma|}\C(|\Gamma|+k, |\Gamma|+k, r)\\
&=\Nr_{|\Gamma|}\Rd^{\sigma_{\Gamma}}\A_{\Gamma}\\
&=\Rd^{\sigma_\Gamma}\Nr_\Gamma\A_\Gamma\\
&=\Rd^{\rho_\Gamma}\Nr_\Gamma\A_\Gamma
\end{align*}
Thus (using a standard Los argument) we have:
$\Pi_{\Gamma/F}\C^r_\Gamma\cong\Pi_{\Gamma/F}\Nr_\Gamma\A_\Gamma=\Nr_\alpha\Pi_{\Gamma/F}\A_\Gamma$,
proving \eqref{en:1}.

The above isomorphism $\cong$ follows from the following reasoning.
Let $\B_{\Gamma}= \Nr_{\Gamma}\A_{\Gamma}$. Then universe of the $\Pi_{\Gamma/F}\C^r_\Gamma$ is
identical to  that of $\Pi_{\Gamma/F}\Rd^{\rho_\Gamma}\C^r_\Gamma$
which is identical to the universe of $\Pi_{\Gamma/F}\B_\Gamma$.
Each operator $o$ of $\QEA_{\alpha}$ is the same
for both ultraproducts because $\set{\Gamma\in I:\dim(o)\subseteq\rng(\rho_\Gamma)} \in F$.

Now we prove \eqref{en:2}.
For this assume, seeking a contradiction, that $\Rd_{sc}\B^r\in S\Nr_{\alpha}\Sc_{\alpha+k+1}$,
$\Rd_{sc}\B^r\subseteq \Nr_{\alpha}\C$, where  $\C\in \Sc_{\alpha+k+1}$.
Let $3\leq m<\omega$ and  $\lambda:m+k+1\rightarrow \alpha +k+1$ be the function defined by $\lambda(i)=i$ for $i<m$
and $\lambda(m+i)=\alpha+i$ for $i<k+1$.
Then $\Rd^\lambda(\C)\in \Sc_{m+k+1}$ and $\Rd_m\Rd_{sc}B^r\subseteq \Nr_m\Rd^\lambda(\C)$.

For each $\Gamma\in I$,\/  let $I_{|\Gamma|}$ be an isomorphism
\[{\C}(m,m+k,r)\cong \Rl_{x_{|\Gamma|}}\Rd_m {\C}(|\Gamma|, |\Gamma+k|,r).\]

Let $x=(x_{|\Gamma|}:\Gamma)/F$ and let $\iota( b)=(I_{|\Gamma|}b: \Gamma)/F$ for  $b\in \C(m,m+k,r)$.
Then $\iota$ is an isomorphism from $\C(m, m+k,r)$ into $\Rl_x\Rd_m\Rd_{sc}\B^r$.
Then by \cite[theorem~2.6.38]{HMT1} we have $\Rl_x\Rd_{m}\B^r\in S\Nr_m\Sc_{m+k+1}$.
It follows that  $\Rd_{sc}\C(m,m+k,r)\in S\Nr_{m}\Sc_{m+k+1}$ which is a contradiction and we are done.

Now we prove the third part of the theorem, putting the superscript $r$ to use.
Let $k$ be as before; $k$ is finite and $>0$ and let $l$ be as in the hypothesis of the theorem,
that is, $l\geq k+1$, and we can assume without loss that $l\leq \omega$.
Recall that $\B^r=\Pi_{\Gamma/F}\C^r_\Gamma$, where $\C^r_\Gamma$ has the signature of $\QEA_{\alpha}$
and $\Rd^{\rho_\Gamma}\C^r_\Gamma=\C(|\Gamma|, |\Gamma|+k, r)$.
We know (this is the main novelty here)
from item (2) that $\Pi_{r/U}\Rd^{\rho_\Gamma}\C^r_\Gamma=\Pi_{r/U}\C(|\Gamma|, |\Gamma|+k, r) \subseteq \Nr_{|\Gamma|}\A_\Gamma$,
for some $\A_\Gamma\in\QEA_{|\Gamma|+\omega}$.

Let $\lambda_\Gamma:|\Gamma|+k+1\rightarrow\alpha+k+1$
extend $\rho_\Gamma:|\Gamma|\rightarrow \Gamma \; (\subseteq\alpha)$ and satisfy
\[\lambda_\Gamma(|\Gamma|+i)=\alpha+i\]
for $i<k+1$.  Let $\F_\Gamma$ be a $\QEA_{\alpha+l}$ type algebra such that $\Rd^{\lambda_\Gamma}\F_\Gamma=\Rd_l\A_\Gamma$.
As before, $\Pi_{\Gamma/F}\F_\Gamma\in\QEA_{\alpha+\omega}$.  And
\begin{align*}
\Pi_{r/U}\B^r&=\Pi_{r/U}\Pi_{\Gamma/F}\C^r_\Gamma\\
&\cong \Pi_{\Gamma/F}\Pi_{r/U}\C^r_\Gamma\\
&\subseteq \Pi_{\Gamma/F}\Nr_{|\Gamma|}\A_\Gamma\\
&=\Pi_{\Gamma/F}\Nr_{|\Gamma|}\Rd^{\lambda_\Gamma}\F_\Gamma\\
&\subseteq\Nr_\alpha\Pi_{\Gamma/F}\F_\Gamma,
\end{align*}
But $\B=\Pi_{r/U}\B^r\in S\Nr_{\alpha}\QEA_{\alpha+l}$
because $\F=\Pi_{\Gamma/F}\F_{\Gamma}\in \QEA_{\alpha+l}$ and $\B\subseteq \Nr_{\alpha}\F$.

\item
For $k\geq 1$. For finite $m>2$, let  $\C(m), \A(m)$ be  the atomic algebras in $\sf TCA_m$
such that $\Rd_{sc}\C(m)\notin S_c\Nr_{m}\Sc_{m+k},$
$\A(m)\in \Nr_k\TCA_{\omega}$ and $\C(m)\equiv \A(m)$ as in the hypothesis.
For $m<n$ we also have
$\C(m)\subseteq_c \Rd_{m}\C(n)$.

Let $\alpha$ be an infinite ordinal. Then we claim that there exists $\B\in \TCA_{\alpha}$
such that $\Rd_{sc}\B\notin S_c\Nr_{\alpha}\Sc_{\alpha+k}$
and $\A\in{\Nr}_{\alpha}\TCA_{\alpha+\omega}$, such that $\A\equiv \B$.

Now we use the same lifting argument as before, we only fix finite $m>2$. We do not have the parameter $r$.
Let $I=\{\Gamma: m\subseteq \Gamma\subseteq \alpha,  |\Gamma|<\omega\}$.
For each $\Gamma\in I$, let $M_{\Gamma}=\{\Delta\in I: \Gamma\subseteq \Delta\}$,
and let $F$ be an ultrafilter on $I$ such that $\forall\Gamma\in I,\; M_{\Gamma}\in F$.
For each $\Gamma\in I$, let $\rho_{\Gamma}$
be a one to one function from $|\Gamma|$ onto $\Gamma.$
Let ${\C}_{\Gamma}$ be an algebra similar to $\TCA_{\alpha}$ such that
$\Rd^{\rho_\Gamma}{\C}_{\Gamma}={\C}(|\Gamma|)$. In particular, ${\C}_{\Gamma}$ has an atomic Boolean reduct.
Let $\B=\prod_{\Gamma/F\in I}\C_{\Gamma}.$

We claim that $\Rd_{sc}\B\notin {\sf S}_c\Nr_{\alpha}\Sc_{\alpha+k}$.
For assume, seeking a contradiction, that $\Rd_{sc}\B\in S\Nr_{\alpha}\Sc_{\alpha+k}$,
$\Rd_{sc}B\subseteq_c \Nr_{\alpha}\C$, where  $\C\in \Sc_{\alpha+3}$.
Let $3\leq m<\omega$ and  $\lambda:m+k\rightarrow \alpha +k$ be the function defined by $\lambda(i)=i$ for $i<m$
and $\lambda(m+i)=\alpha+i$ for $i<k$.
Then $\Rd^\lambda\C\in \PEA_{m+k}$ and $\Rd_m\B^r\subseteq_c \Rd_m\Rd^\lambda\C$.

For each $\Gamma\in I$,   $|\Gamma|\geq m$, we have
\[{\C}(m)\subseteq _c\Rd_m {\C}(|\Gamma|).\]
Let $I_{|\Gamma|}$ be an injective complete homomorphism, witnessing this complete embedding.
Let $x=(x_{|\Gamma|}:\Gamma)/F$ and let $\iota( b)=(I_{|\Gamma|}b: \Gamma)/F$ for  $b\in \C(m)$.
Then $\iota$ is an injective homomorphism that embeds $\C(m)$ into $\Rd_m\B^r$, and
this embedding is complete.
Now $\Rd_{m}\B^r\in {\sf S}_c\Nr_m\Sc_{m+k}$, hence  $\Rd_{sc}\C (m)\in
{\sf S}_c\Nr_{m}\Sc_{m+k}$ which is a contradiction and we are done.

We now show that there exists $\A\in \Nr_{\alpha}\TCA_{\alpha+\omega}$ such that $\A\equiv \B$.
We use the $\A(k)$s.
For each $\Gamma\in I$ we can assume that $\Nr_{|\Gamma|}\A(|\Gamma|+k)\cong\A(|\Gamma|)$; if not then replace $\A(|\Gamma|+k))$ by an algebra
$\D(|\Gamma|+k)$ such that $\Nr_{|\Gamma|}\D{(|\Gamma|+k)}\cong \A(\Gamma)$. Such a $\D$ obviously exists.

Let $\sigma_{\Gamma}$ be an injective map
 $(|\Gamma|+\omega)\rightarrow(\alpha+\omega)$ such that $\rho_{\Gamma}\subseteq \sigma_{\Gamma}$
and $\sigma_{\Gamma}(|\Gamma|+i)=\alpha+i$ for every $i<\omega$. Let
$\A_{\Gamma}$ be an algebra similar to a
$\QEA_{\alpha+\omega}$ such that
$\Rd^{\sigma_\Gamma}\A_{\Gamma}=\A(|\Gamma|+k)$, and $\Nr_{\alpha}\A_{\Gamma}\equiv \C_{\Gamma}$.
Then $\Pi_{\Gamma/F}\A_{\Gamma}\in \CA_{\alpha+\omega}$.

We now prove that $\A= \Nr_\alpha\Pi_{\Gamma/F}\A_\Gamma$.
Using the fact that neat reducts commute with forming ultraproducts, for each $\Gamma\in I$, we have
\begin{align*}
\Rd^{\rho_{\Gamma}}\A_{\Gamma}&=\A(|\Gamma|)\\
&\cong\Nr_{|\Gamma|}\A(|\Gamma|+k)\\
&=\Nr_{|\Gamma|}\Rd^{\sigma_{\Gamma}}\A_{\Gamma}\\
&=\Rd^{\sigma_\Gamma}\Nr_\Gamma\A_\Gamma\\
&=\Rd^{\rho_\Gamma}\Nr_\Gamma\A_\Gamma
\end{align*}
We deduce that
$$\A=\Pi_{\Gamma/F}\C_\Gamma\cong\Pi_{\Gamma/F}\Nr_\Gamma\A_\Gamma=
\Nr_\alpha\Pi_{\Gamma/F}\A_\Gamma\in \Nr_{\alpha}\QEA_{\alpha+\omega}.$$

Finally, we have $\A\equiv \B$, because $\A=\Nr_{\alpha}\Pi_{\Gamma/F}\A_{\Gamma}=\Pi_{\Gamma/F}\Nr_{\alpha}\A_{\Gamma},$
$\B=\Pi_{\Gamma/F}\C_{\Gamma}$, and we chose $\Nr_{\alpha}\A_{\Gamma}$ to be elementary equivalent
to $\C_{\Gamma}$ for each finite subset $\Gamma$ of $\alpha$.

\item  Let  $I$ and $\C(\Gamma)$ be defined as above for every
$\Gamma$ finite subset of
$\alpha$.
For each $\Gamma\in I$, let $\rho_{\Gamma}$
be a one to one function from $|\Gamma|$ onto $\Gamma.$
Let ${\C}_{\Gamma}$ be an algebra similar to $\QEA_{\alpha}$ such that
$\Rd^{\rho_\Gamma}{\C}_{\Gamma}={\C}(|\Gamma|)$. In particular, ${\C}_{\Gamma}$ has an atomic Boolean reduct.
Let $\B=\Pi_{\Gamma/F\in I}\C_{\Gamma}.$
Then $\B$ is atomic, because it is an ultraproduct of atomic algebras.
We claim that
$\Rd_{sc}\B$ is not completely representable.

Assume for contradiction
that $\Rd_{sc}\B$ is completely representable, with complete representation $f$.
Let $3\leq m<\omega$. Then of course $\Rd_m\Rd_{sc}\B$ is completely representable with the same $f$; notice that both
$\Rd_{sc}\B$ and $\Rd_m\Rd_{sc}\B$ have the same universe
as $\B$.

For each $\Gamma\in I$,\/  let $I_{|\Gamma|}$ be an isomorphism
${\C}(m)\cong \Rl_{x_{|\Gamma|}}\Rd_m {\C}(|\Gamma|).$
Let $x=(x_{|\Gamma|}:\Gamma)/F$ and let $\iota( b)=(I_{|\Gamma|}b: \Gamma)/F$ for  $b\in \C(m)$.
Then $\iota$ is an isomorphism from $\C(m)$ into $\Rl_x\Rd_m\B^r$.
Then by \cite[theorem~2.6.38]{HMT1} we have $\Rl_x\Rd_{m}\B\in \PEA_{m}$ and it is atomic.
Indeed, if $a\leq x$ is non-zero, then there is an atom $c\in \Rd_m\B$ below $a$, so that $c\leq a\leq x$,
hence $c\in \Rl_x\Rd_m\B$ is an atom.
Atomicity is not enough to guarantee complete representability because the class of completely representable algebras
is not axiomatizable even for $\sf Df$s, but as it happens
we do have that $\Rl_x\Rd_m\Rd_{sc}\B$ is a completely representable $\Sc_m$, as we proceed to show.
First of all, we know that $\Rd_m\Rd_{df}B$ is completely representable, so let
$f:\Rd_m\Rd_{sc}B\to \wp(V)$ be a complete representation;  that is a representation
that preserves joins. Here $\rng(f)$ is a generalized set algebra; that is $V$, a generalized space,
is of the form $\bigcup_{i\in I}{}^{\alpha}U_i$, where
for distinct $i$ and $j$, $U_i\cap U_j=\emptyset.$

Define $g:\Rl_x\Rd_m\B\to \Rl_{f(x)}\wp(V)$ by $g(b)=f(b)$ for $b\leq x$.
We have $f(x)\subseteq V$, and ${\sf s}_l^kx\cdot {\sf s}_k^lx=x$ for all $l, k<m$, $k\neq l$,
so this equation holds also for $f(x)$, that is, we have ${\sf s}_l^kf(x)\cap {\sf s}_k^lf(x)=f(x)$.
By \cite[theorem 3.1.31]{HMT2}, we have
$f(x)$ is a generalized space, too, and indeed
$\Rl_{f(x)}\wp(V)\cong \wp(f(x))$.

Assume that $X\subseteq \Rl_x\B=\Rl_x\Rd_m\Rd_{sc}\B$ is such that $\sum X=x$,
then $$g(\sum X)=f(\sum X)=\bigcup_{y\in X}f(y)=f(x)=1^{\wp(f(x)}.$$
Hence $g$ is a complete representation of $\Rl_x\Rd_m\Rd_{sc}\B$.
The latter is isomorphic to $\Rd_{sc}C (m)$, hence $\Rd_{sc}C(m)$ 
is completely representable, too.
This is a contradiction, and we are done.

Finally, $\B\in \Nr_{\alpha}\TCA_{\alpha+\omega}$ exactly as above.

The second part follows from the second item in theorem \ref{SL}.
\end{enumarab}
\end{proof}
Let us see how close we are to the hypothesis adressing finite dimensional algebras.

\begin{enumarab}
\item For the first  item for $\CA$ and $\PEA$ we have 
$\mathfrak{C}(m,n,r)=\Ca(H_m^{n+1}(\A(n,r), \omega)),$
consisting of all $n+1$-wide $m$-dimensional  
wide $\omega$ hypernetworks \cite[definition 12.21]{HHbook} on $\A(n,r)$, is in $\CA_m$ and it can be easily expanded
to a $\PEA_m$. 

Furthermore, for any $r\in \omega$ and $3\leq m\leq n<\omega$, we
have $\C(m,n,r)\in \Nr_m{\sf PEA}_n,$ $\Rd_{ca}C(m,n,r)\notin {\sf S}\Nr_m{\sf CA_{n+1}}$
and $\Pi_{r/U}\C(m,n,r)\in {\sf RPEA}_m.$

Lastly, let  $3\leq m<n$. Take $$x_n=\{f:{}^{\leq n+k+1}n\to \At\A(n+k, r)\cup \omega:  m\leq j<n\to \exists i<m, f(i,j)=Id\}.$$
Then $x_n\in C(n,n+k,r)$ and ${\sf c}_ix_n\cdot {\sf c}_jx_n=x_n$ for distinct $i, j<m$.
Furthermore
\[{I_n:\C}(m,m+k,r)\cong \Rl_{x_n}\Rd_m {\C}(n,n+k, r),\]
via the map, defined for $S\subseteq H_m^{m+k+1}(\A(m+k,r) \omega)),$ by
$$I_n(S)=\{f: {}^{\leq n+k+1}n\to \At\A(n+k, r)\cup \omega: f\upharpoonright {}^{\leq m+k+1}m\in S,$$
$$\forall j(m\leq j<n\to  \exists i<m,  f(i,j)=Id)\}.$$
We have shown that we have {\it all conditions} in the hypothesis of the first item of theorem 
\ref{infinite} for cylindric and polyadic equality algebras.

However for $\Sc$ and $\QA$ the problem remains open. 
The algebras constructed in \cite{t} and recalled in theorem \ref{thm:cmnr} for another purpose,
do not have representable ultarproducts, so the diagonal free cases in this stronger form 
is not yet confirmed.

\item By theorem \ref{rainbow} there exist for every finite $m>2$, polyadic equality atomic algebras $\C(m)=\PEA_{\N^{-1}, \N}$ 
and $\B(m)$ such that $\Rd_{sc}\C(m)\notin S_c\Nr_n\Sc_{m+3}$,
$\C(m)\equiv \B(m)$, $\B$ is a countable (atomic) completely representable $\TCA_m$ except
that we could only succeed to show  that $\At\B(m)\in \At\Nr_m\TCA_{\omega}$; we do not know whether we can remove $\At$ from both sides of 
the equation; that is, we do not know whether $\B(m)$ can be chosen to be in $\Nr_n\QEA_{\omega}$. 
A conditional theorem depending on a game $H$, stronger than the game $J$ used 
was given to ensure the latter condition, but it can be shown without too much 
difficulty that \pa\ can win this finite rounded game on $\At\TCA_{\N^{-1}, \N}$ as long as the number of rounds are 
$>m$.
Also we not know whether $\C(m)$ 
embeds completely into $\C(n)$ for $n>m$. However, we tend to think 
that the sitaution is not hopeless and that such algebras can be found, but
further research is needed.

\item For the third part for $k\geq 3$, let $\C(k)\in \Nr_k\TCA_{\omega}$ be the atomic uncountable ${\sf PEA}_k$
such that its $\sf Df$ reduct is not completely representable constructed in the first item theorem \ref{complete}.
Recall that $\C(k)$ was obtained as a $k$ neat reduct of an $\omega$ dimensional algebra whose atom structure
is an $\omega$ basis over a relation algebra, denoted by $\R$, in the proof of the first item in theorem \ref{complete}. Hence an atom in $\C(k)$
is of the form $\{N\}$, where $N: \omega\times \omega\to \At\R$ is an $\omega$ dimensional network.
For $3\leq m<n$,
let $$x_n=\{N\in C(n):  m\leq j<n\to \exists i<m, N(i,j)=\Id\}.$$
Then $x_n\in C(n)$ and ${\sf c}_ix_n\cdot {\sf c}_jx_n=x_n$ for distinct $i, j<m$.
Furthermore,
\[{I:\C}(m)\cong \Rl_{x_n}\Rd_{m}{\C}(n)\]
via
$$I(S)=\{N\in C(n): N\upharpoonright \omega\times \omega\in S,$$
$$\forall j(m\leq j<n\to  \exists i<m,  N(i,j)=\Id\}.$$
This is similar to the definition in the first item above, for in the former case we have a hyperbasis
and now we have an amalgamation class.
\end{enumarab}

\begin{corollary} Let $\alpha\geq \omega$. Then the following hold: 
\begin{enumarab}
\item For any $k\geq 0$, 
and any $l\geq k+2$ and , then the variety 
$S\Nr_{\alpha}\TCA_{\alpha+l}$ 
cannot be axiomatized by a finite schema over
the variety $S\Nr_{\alpha}\K_{\alpha+k}$
\item There are algebras in 
$\Nr_{\alpha}\TCA_{\alpha+\omega}$ that are not completely representable.
\end{enumarab}
\end{corollary}
\begin{proof} 
The second part follows immediately from the first item of theorem \ref{complete} and the third of the previous theorem \ref{infinite}.  

For the second part let us first recall from \cite{HMT2} what we mean by finite schema. 
If $\rho:\omega\to \alpha$ is an injection, then $\rho$ extends recursively
to a function $\rho^+$ from $\CA_{\omega}$ terms to $\CA_{\alpha}$ terms.
On variables $\rho^+(v_k)=v_k$, and for compound terms
like ${\sf c}_k\tau$, where $\tau$ is a $\CA_{\omega}$ term, and $k<\omega$, $\rho^+({\sf c}_k\tau)={\sf c}_{\rho(k)}\rho^+(\tau)$.
For an equation $e$ of the form $\sigma=\tau$ in the language of $\CA_{\omega}$, $\rho^+(e)$ is
the equation $\rho^+(\tau)=\rho^+(\sigma)$ in the language
of $\CA_{\alpha}$. This last equation, namely, $\rho^+(e)$ is called an $\alpha$ instance of $e$
obtained by applying the injection $\rho$.

Let $k\geq 1$ and $l\geq k+1$. Assume for contradiction that $\bold S\Nr_{\alpha}\TCA_{\alpha+l}$ is axiomatizable
by a finite schema over  $\bold S\Nr_{\alpha}\TCA_{\alpha+k}$.
We can assume that there is only one equation, such that all its $\alpha$ instances,  
axiomatize  $S\Nr_{\alpha}\TCA_{\alpha+l}$ over
$S\Nr_{\alpha}\TCA_{\alpha+k}$.
So let $\sigma$ be such an  equation in the signature of $\TCA_{\omega}$ and let $E$ be its $\alpha$ instances; so that
 for any $\A\in \bold S\Nr_{\alpha}\TCA_{\alpha+k}$ we have $\A\in \bold S\Nr_{\alpha}\TCA_{\alpha+l}$ iff
$\A\models E$.  Then for all $r\in \omega$, there is an instance of $\sigma$,  $\sigma_r$ say,
such that $\B^r$ does not model $\sigma_r$.
$\sigma_r$ is obtained from $\sigma$ by some injective map $\mu_r:\omega\to \alpha$.

For $r\in \omega,$ let $v_r\in {}^{\alpha}\alpha$,
be an injection such that $\mu_r(i)=v_r(i)$ for $i\in {\sf ind}(\sigma_r)$, and
let $\A_{r}= \Rd^{v_r}\B^r$.
Now $\Pi_{r/U} \A_{r}\models \sigma$. But then
$$\{r\in \omega: \A_{r}\models \sigma\}=\{r\in \omega: \B^r\models \sigma_r\}\in U,$$
contradicting that
$\B^r$ does not model $\sigma_r$ for all $r\in \omega$.
\end{proof}

\subsection{Gripping atom structures}

Example \ref{SL} promps the following definition 
$\A, \B\in \sf K\subseteq \TCA_n$ are atomic, then of course $\At\A$ abd $\At\B$ is in $\At\K$. The next definition addreses the converse 
of this.
There are classes of algebras $\K$
such that there are two atomic algebras having the same atom structure, one in $\K$ and the other is not,
examplea  are  classes $\sf TRCA_n$ for finite $n>2$ theorem \ref{can} 
and the class $\Nr_n\TCA_m$ for all $m>n$, examples \ref{SL} and \ref{finitepa}.

Conversely, there are classes of algebras, like the class of completely representable
algebras in any finite dimension, and also $S_c\Nr_n\TCA_{m}$ for any $m>n>2$ $n$ finite,  that do not
have this property. This is {\it not} marked by first
order definability, for the class of neat reducts and the last two  classes not elementary, while
$\sf TRCA_n$ is a variety for any $n$.

\begin{definition}\label{grip}
\begin{enumarab}

\item A class $\K$ is gripped by its atom structures, if whenever $\A\in \K\cap \At$, and $\B$ is atomic such that
$\At\B=\At\A$, then $\B\in \K$.

\item A class $\K$ is strongly gripped by its atom structures, if whenever $\A\in \K\cap \At$, and $\B$ is atomic such that
$\At\B\equiv \At\A$, then $\B\in \K$.

\item A class $\K$ of atom structures is infinitary gripped
by its atom structures if whenever $\A\in \K\cap \At$ and $\B$ is atomic, such that $\At\B\equiv_{\infty,\omega}\At\B$,
then $\B\in \K$.
\item An atomic  game is strongly gripping for $\K$ if whenever $\A$ is atomic, and
\pe\ has a \ws\ for all finite rounded games on $\At\A$, then  $\A\in \K$

\item An atomic  game is gripping if whenever \pe\ has a \ws\ in the $\omega$ rounded game on $\At\A$, then $\A\in \K$.
\end{enumarab}
\end{definition}

Notice that infinitary gripped implies strongly gripped implies gripped (by its atom structures).
For the sake of brevity, we write only (strongly) gripped, without referring to atom structures.

In the next theorem, all items except the first applies to all algebras considered.

\begin{theorem}\label{grip}
\begin{enumarab}
\item The class of neat reducts for any dimension $>1$ is not gripped,
hence is neither strongly gripped nor infinitary gripped.
\item The class of completely representable algebras is gripped but not strongly gripped.
\item The class of algebras satisfying the Lyndon conditions is strongly gripped.
\item The class of representable algebras is not gripped.
\item The Lyndon usual atomic game is gripping but not strongly gripping
for completely representable algebras, it is strongly gripping for ${\sf LCA_m}$, when $m>2$.
\item There is a game, that is {\it not} atomic,  that is strongly gripping for $\Nr_n\CA_{\omega}$, call it $H$.
In particular, if there exists an algebra $\A$ and $k\geq 1$ such \pa\ has a \ws\ in $F^{n+k}$ and \pe\ has
a \ws\ in $H_m$, the game $H$ truncated to $m$ rounds, for every finite $m$,
then for any class $\K$ such that  $\Nr_n\CA_{\omega}\subseteq \K\subseteq  S_c\Nr_n\CA_{n+k},$
$\K$ will not be elementary.
\end{enumarab}
\end{theorem}

\begin{proof}
\begin{enumarab}

\item  Example \ref{SL}

\item The algebra $\TCA_{\Z,\N}$ or $\TCA_{\N^{-1}, \N}$ used in the proof of theorem \ref{rainbow} and its various reducts down to $\Sc$s,
shows that the class of completely representable algebras is not strongly gripped. 

The former rianbow algebra based on the greens $\Z$ and reds $\N$ 
would also work just as well proving theorem \ref{rainbow}.
\pe\ wins the $k$ rounded game $G_k$ is the same and the decreasing sequence in 
$\N$ is forced on \pe\ by \pa\ 's moves using and reusing only $n+ 3$ pebbles
by exploiting  the negative mirror image 
of $\N$, namely, the set  $\{-a: a\in \N\}\subseteq \Z$, so the game 
alternates between postive reds played by \pa\ and negative greens played by \pe\ 
synchronized by order preserving partial function inevitably forcing \pa\ to play
a decreasing sequence in $\N$. But this cannot go on for ever, hence \pa\ wins the $\omega$ rounded game 
$F^{n+3}$ by pebbling succesively $\N$ and its mirror image. 
\item This is straightforward from the definition of Lyndon conditions \cite{HHbook}.

\item Any weakly representable atom structure that is not strongly representable detects this, see e.g.
the main results in \cite[theorem 1.1, corolary 1.2, corolary 1.3]{Hodkinson},
\cite{weak}, \cite[theorems 1.1, 1.2]{ANT} and theorems \ref{can}.
\item From item (2) and the rest follows directly from the definition.

\item  We allow more moves for \pa\ in the neat game  defined in remark \ref{game} way above.
\end{enumarab}
\end{proof}

But nevertheless concerning the last item, we can say more. Indeed, more generally the following.
Assume that If  ${\sf L}$ is a class of algebras, $G$ be a game and  $\A\in {\sf L}$, then \pe\ has a \ws\ in $G$.
Assume also that $\Nr_n\CA_{\omega}\subseteq \sf L$.
If there is an atomic algebra $\A$ such that \pe\ can win all finite rounded games of $J$
and \pa\ has a \ws\ in $G$ then any class between
$\Nr_n\CA_{\omega}$ and $\sf L$ is not elementary.

\end{document}